\documentclass[11pt,a4paper]{article}

\usepackage{titling}
\usepackage{authblk}

\usepackage[margin=1.5cm, font=small,labelfont=bf]{caption}
\usepackage[utf8]{inputenc}
\usepackage{amsmath,amsthm,amsfonts,amssymb}
\usepackage{graphicx}
\usepackage{color}
\usepackage{url}
\usepackage{booktabs}
\usepackage{enumerate}
\usepackage{cases}
\usepackage{multirow}
\usepackage{braket}
\usepackage{mathrsfs}
\usepackage{kbordermatrix}

\usepackage{soul}

\usepackage[normalem]{ulem}

\usepackage{algorithm, algorithmic}

\usepackage{cleveref}

\newtheorem{lemma}{Lemma}[section]
\newtheorem{theorem}{Theorem}[section]
\newtheorem{corollary}{Corollary}[section]
\newtheorem{remark}{Remark}[section]

\newtheorem{definition}{Definition}[section]

\newtheorem{example}{Example}[section]

\usepackage{float}
\def\bbC{\mathbb C}

\def\bbR{\mathbb R}
\def\bbO{\mathbb O}

\def\scrG{\mathscr{G}}
\def\scrL{\mathscr{L}}

\def\tr{\mathrm{tr}}

\def\D{\mathbf D}

\newcommand\al[1]{[\![#1]\!]} % alignment

\DeclareMathOperator{\diag}{diag}

\DeclareMathOperator{\mrank}{rank}

\DeclareMathOperator*{\argmax}{arg\,max}

\def\wtd{\widetilde}

\def\rkd{r_{\!{}_{\!D}}} % define r_D, using two level subscript

\def\cL{{\cal L}}
\def\cR{{\cal R}}
\def\cN{{\cal N}}
\def\cQ{{\cal Q}}

\def\sss{\scriptscriptstyle}

\DeclareMathOperator{\NRes}{NRes}

\DeclareMathOperator{\HH}{H}
\DeclareMathOperator{\T}{T}
\def\rankd{\mbox{rank-preserving}}
\def\dproper{\mbox{$D$-regular}} 	%\def\dproper{\mbox{$D$-consistent}}
\def\bbod{\bbO_{\!{}_D}^{n\times k}} 	%\bbod\def{{\bbod}

\numberwithin{equation}{section}
\numberwithin{figure}{section}
\numberwithin{table}{section}
\numberwithin{algorithm}{section}

\title{
Locally Unitarily Invariantizable NEPv and \\ Convergence Analysis of SCF
}
\author{Ding Lu\thanks{
Department of Mathematics, University of Kentucky,
	Lexington, KY  40506, USA,
	{\tt Ding.Lu@uky.edu}.
    Supported in part by NSF DMS-2110731.}
	~and
	Ren-Cang Li
	\thanks{ Department of Mathematics, University of Texas at Arlington,
Arlington, TX 76019-0408, USA, {\tt rcli@uta.edu}.
   Supported in part by NSF DMS-2009689.}
}
\date{\vspace{-5ex}}

\graphicspath{{./FIGS/},{../FIGS/}}

\begin{document}

\maketitle

\begin{abstract}

We consider a class of
eigenvector-dependent nonlinear eigenvalue problems (NEPv)
without the unitary invariance property.
Those NEPv commonly arise as
the first-order optimality conditions of a particular type of optimization problems over the
Stiefel manifold, and
previously, special cases have been studied in the literature.
Two necessary conditions, a definiteness condition and a rank-preserving
condition, on an eigenbasis matrix of the NEPv that is
a global optimizer of the associated  optimization problem are revealed,
where the definiteness condition has been known for the special cases previously investigated.
We show that, locally close to the eigenbasis matrix satisfying both necessary conditions,
the NEPv can be reformulated as a unitarily invariant NEPv, the so-called {\em aligned NEPv},
through a basis alignment operation ---
in other words, the NEPv is locally unitarily invariantizable.
Numerically, the NEPv is naturally solved by an SCF-type iteration.
By exploiting the differentiability of the coefficient matrix of the aligned NEPv,
we establish a closed-form local convergence rate for the SCF-type iteration
and analyze its level-shifted variant.
Numerical experiments confirm our theoretical results.
\end{abstract}



%\clearpage
%\tableofcontents
%
%
%\newpage

\section{Introduction}

Consider the
{\em eigenvector-dependent nonlinear eigenvalue problem} (NEPv):
Find $X\in\bbR^{n\times k}$ that has orthonormal columns (i.e., $X^{\T}X=I_k$) and a square
$\Lambda\in\bbR^{k\times k}$ satisfying
\begin{equation}\label{eq:nepv}
	H(X)X=X\Lambda,
\end{equation}
where $H(X)\in\bbR^{n\times n}$ is a symmetric matrix continuously
dependent of $X$. When \eqref{eq:nepv} holds, we call $X$ an \emph{eigenbasis matrix}
and $(X,\Lambda)$ an eigen-matrix pair of the NEPv.
Necessarily, $k \leq n$ (usually $k\ll n$), and
the columns of $X$ form an orthonormal basis of the
eigenspace of $H(X)$
%\Blue{evaluated at $X$}, 
evaluated at $X$, 
and the corresponding
eigenvalues are the $k$ eigenvalues of $\Lambda = X^{\T}H(X)X\in\bbR^{k\times k}$,
%\Blue{which is necessarily symmetric}.
which is necessarily symmetric.
Note that $\Lambda$ may not be diagonal and
individual columns of $X$ may not be some eigenvectors of $H(X)$.

NEPv~\eqref{eq:nepv} commonly arises in important real-life applications.
The two most prominent
examples are the Kohn-Sham equation in the density functional
theory~\cite{Martin:2004,Szabo:2012}
and the Gross-Pitaevskii equation for the Bose-Einstein
condensation~\cite{Bao:2004,Jarlebring:2014},
both from computational physics and chemistry.
In recent years, NEPv increasingly show up
in the fields of data science and machine learning
where various optimization problems with orthogonality constraints need to be solved;
see, e.g.,~\cite{Bai:2018,Jost:2014,Meyer:1997,Ngo:2012,Tudisco:2019,zhli:2014a,zhli:2014b,zhln:2010,zhln:2013,zhwb:2020}.

NEPv is a term coined by the authors of \cite{Cai:2018}, where it
%while investigating an often used self-consistent field (SCF) iterative
refers to a problem~\eqref{eq:nepv}
with $H(\,\cdot\,)$ satisfying the so-called \emph{unitary invariance property},
by which we mean\footnote{ %
	The term `unitary' is from the setting of complex matrices,
	and we stick to this name convention in our discussion.
    In fact, the results in this paper can be extended
    to cover the complex case, as we will comment later in~\Cref{sec:conclusion}.
} %
%A coefficient matrix $H(X)$ is called (right) \emph{unitarily invariant} if it satisfies
\begin{equation}\label{eq:invariant}
	H(XQ) = H(X)
	\quad\text{for all orthogonal $Q\in\bbR^{k\times k}$}.
\end{equation}
Under condition~\eqref{eq:invariant}, if $X$ solves
NEPv~\eqref{eq:nepv} then so does $XQ$,
and, hence, a solution $X$ is really a representative of
a set of solutions that share the same column space, i.e., a point in the Grassmannian
$\scrG_k^n$ (the collection of all $k$ dimensional subspaces in $\bbR^{n}$).
This unitary invariance property is possessed by
some
practical NEPv, such as those from computational physics and chemistry mentioned above,
and it facilitates theoretical analysis.
Particularly, it allows us to investigate the NEPv and related numerical algorithms
by the eigenspace perturbation theory of the symmetric eigenvalue
problem \cite{li:2014HLA,stsu:1990}.
By exploiting unitary invariance, sufficient conditions for the existence of
solutions of NEPv have been established in~\cite{Cai:2018},
and estimations of the convergence rate for the
often used self-consistent field (SCF)
iteration have been obtained in~\cite{Bai:2022,Liu:2014,Upadhyaya:2021}.
With~\eqref{eq:invariant}, we can also require that the columns of
a solution representative $X$ are eigenvectors of $H(X)$ and
$\Lambda$ is diagonal. Because otherwise
we can diagonalize $\Lambda$ by $Q^{\T}\Lambda Q$ with an orthogonal $Q\in\bbR^{k\times k}$,
and take $(XQ,Q^{\T}\Lambda Q)$ to be the new eigen-matrix pair,
since $H(XQ)XQ=XQ(Q^{\T}\Lambda Q)$ by \eqref{eq:nepv} and \eqref{eq:invariant}.

Although most existing studies of NEPv \eqref{eq:nepv} focus on the ones with
the unitary invariance property \eqref{eq:invariant},
many recent applications also give rise to NEPv without this property.
One important source of such problems is the multi-view subspace learning
(see, e.g., \cite{cugh:2015,sumd:2019,wazb:2020,wazl:2022,zhwb:2020,Zhang:2020}),
where the problems of interest
appear in some forms of the trace-ratio maximization
over the Stiefel manifold
\begin{equation}\label{eq:O(n,k)}
\bbO^{n\times k}:=\{X\in\bbR^{n\times k}\,:\,X^{\T}X=I_k\},
\end{equation}
the collection of  matrices in $\bbR^{n\times k}$ that have orthonormal columns.
For example, in \cite{wazl:2022} it is considered
\begin{equation}\label{eq:max}
	\max_{X\in\bbO^{n\times k}} f_{\theta} (X)
	\quad\text{with}\quad
	f_{\theta}(X):=
	\frac{\tr(X^{\T}AX + X^{\T}D)}{[\tr(X^{\T}BX)]^\theta},
\end{equation}
where $\theta\in[0,1]$ is a tunable parameter, $A,\,B\in\bbR^{n\times n}$ are
symmetric with $B\succ 0$ (positive definite),
and $D\in\bbR^{n\times k}$.
Depending on the coefficient matrices and the parameter $\theta$,
the optimization~\eqref{eq:max} includes a wide variety of practical cases,
and a few special ones %of optimization problem~\eqref{eq:max}
have been playing important roles in numerical linear algebra, machine learning,
and statistics:
\begin{itemize}
  \item $D=0$ and $\theta=1$ appears in Fisher's linear discriminant analysis (LDA)
\cite{ngbs:2010,zhln:2010,zhln:2013} in the setting of supervised machine learning;
  \item $A=0$ and $\theta=1/2$ shows up in the orthogonal canonical correlation analysis (OCCA) \cite{zhwb:2020};
  \item $B=I_n$ or $\theta=0$ %yields
	  arises in the unbalanced orthogonal Procrustes problem,
	  a fundamental problem in numerical linear algebra,
  optimization, and applied statistics
  \cite{boli:1987,chtr:2001,geer:1984,edas:1999,elpa:1999,godi:2004,huca:1962,liww:2015,Zhang:2020,zhdu:2006}.
\end{itemize}

The optimization problem~\eqref{eq:max} in its general form has been studied in~\cite{wazl:2022},
where it is shown that the first-order optimality condition of~\eqref{eq:max} is
equivalent to an NEPv~\eqref{eq:nepv} with
\begin{equation}\label{eq:hx0}
	H(X) =
	\frac{1}{[\tr(X^{\T}BX)]^\theta} \bigg[
		2\Big(A - \theta  \frac{\tr(X^{\T}AX)}{\tr(X^{\T}BX)}\cdot B\Big) +
		\Big(DX^{\T} + XD^{\T} -2\theta \frac{\tr(X^{\T}D)}{\tr(X^{\T}BX)} \cdot
	B\Big) \bigg].
\end{equation}
% % -- Separate definition
%\begin{equation}\label{eq:hx0}
%	H(X) =
%	\frac{1}{[\tr(X^{\T}BX)]^\theta} \bigg[
%		2\Big(A - \theta \cdot g(X)\cdot B\Big) +
%		\Big(DX^{\T} + XD^{\T} -2\theta \cdot h(X) \cdot B\Big) \bigg],
%\end{equation}
%where
%\begin{equation}\label{eq:gxhx}
%	g(x) := \frac{\tr(X^{\T}AX)}{\tr(X^{\T}BX)}
%	\qquad\text{and}\qquad
%	h(x) := \frac{\tr(X^{\T}D)}{\tr(X^{\T}BX)}.
%\end{equation}
Due to the presences of $DX^{\T}$ and $XD^{\T}$,
this $H(X)$  does not satisfy the unitary invariance property~\eqref{eq:invariant}.
Hence, existing analyses and techniques \cite{Bai:2022,Cai:2018}
developed for NEPv with property \eqref{eq:invariant}
do not directly apply.
Especially, the plain SCF iteration does not work for those problems
and needs a redesign.
In \cite{wazl:2022,zhwb:2020,Zhang:2020},
the authors developed SCF-type iterations to solve their
particular NEPv and established global convergence for the algorithms,
which is rather remarkable since the optimization problems~\eqref{eq:max}
are non-convex where global convergence cannot be guaranteed in general.
However, their analyses and proofs are problem-specific and cannot be
easily extended to other cases,
and they do not lead to useful quantitative estimates for
rates of convergence of their SCF-type iterations.
New convergence theories are needed to better understand
and predict the convergence behaviors of such algorithms.

%Despite the NEPv is not unitarily invariant in its solutions,
%\cite{wazl:2022} showed it
%can be efficiently solved by an SCF-type iteration with
%convergence guarantee.

%Because of lacking the unitary invariance property \eqref{eq:invariant},
%existing analyses and techniques \cite{Bai:2022,Cai:2018} developed for NEPv
%with the property \eqref{eq:invariant}
%do not directly apply to the NEPv \eqref{eq:nepv} with~\eqref{eq:hx0}.
%For example, SCF-type iterations can still be designed to solve
%such NEPv, as have been done in \cite{wazl:2022,Zhang:2020,zhwb:2020}, but
%new theories are needed to understand and predict
%their convergence behavior.
%The authors of \cite{wazl:2022,Zhang:2020,zhwb:2020}
%were able to prove their SCF-type iterations are globally convergent,
%which is rather remarkable since the optimization problems are non-convex
%where global convergence are not guaranteed in general.
%However, their analyses and proofs are problem-dependent and cannot be
%easily extended to more general problems.
%In addition, they do not lead to useful quantitative estimates for
%rates of convergence of the SCF iterations.

%
\smallskip\noindent
{\bf Contribution.}
The major goal of this paper is twofold: i) to develop a general theory for analyzing
a class of NEPv \eqref{eq:nepv} that violates the unitary invariance
property~\eqref{eq:invariant} while includes the ones with~\eqref{eq:hx0} as special cases;
ii) to extend the local convergence analysis of \cite{Bai:2022} to such a class of NEPv.
%Although our results apply to NEPv \eqref{eq:nepv} broader than with ~\eqref{eq:hx0},
%for clarity we summarize our results
%for the case of NEPv \eqref{eq:nepv} with~\eqref{eq:hx0} as follows.
Our analyses will apply to more general NEPv \eqref{eq:nepv} than those 
%\Blue{such as the ones}
such as the ones
with~\eqref{eq:hx0}, but for now and for clarity,
let us first summarize our results for the case of NEPv \eqref{eq:nepv} with~\eqref{eq:hx0} as follows.
\begin{itemize}
	\item
		%By investigating global optimality conditions of~\eqref{eq:max},
		We show that any global optimizer $X_*$ must be a
		\emph{\dproper~eigenbasis matrix},
		i.e., satisfying
		\begin{equation}\label{eq:proper-basis}
		X_*^{\T}D \succeq 0
		\qquad\text{and}\qquad
		\mrank(X_*^{\T}D) = \mrank(D).
		\end{equation}
        The first condition $X_*^{\T}D \succeq 0$ is not hard to establish and has been known
        to \cite{wazl:2022,zhwb:2020,Zhang:2020}, whereas the second condition
        $\mrank(X_*^{\T}D) = \mrank(D)$ is new and is a critical one for
        our local convergence analysis to go through.

	\item
		Our analysis is made possible by a novel transformation of
        the NEPv with \eqref{eq:hx0},
		through a \emph{basis alignment} operation,
		to an equivalent one that does admit
        the unitary invariance property
		and has a differentiable coefficient matrix.
		The resulting NEPv, which we will call the \emph{aligned NEPv\/},
		is well-defined for all $X$ close to a \dproper~eigenbasis matrix.
		Namely,
		locally close to a \dproper~eigenbasis matrix,
		the NEPv with~\eqref{eq:hx0} is
		unitarily invariantizable.

	\item
	 	%\st{By the transformation to  yield an aligned NEPv,}
		We show that the SCF-type iteration for NEPv with \eqref{eq:hx0},
		as developed in \cite{wazl:2022} and others in \cite{zhwb:2020,Zhang:2020},
		is equivalent to the plain SCF iteration for the aligned NEPv.
		By extending the local convergence analysis in~\cite{Bai:2022},
		we establish a sharp estimation of the rate of convergence for the
		SCF-type iteration and build the theoretic foundation
		for a level-shifting scheme to fix the potential divergence issue.

	%	We show that a popular SCF-type iteration, %for the NEPv with \eqref{eq:hx0},
	%	developed in \cite{wazl:2022} and others in \cite{Zhang:2020,zhwb:2020}
	%	for the NEPv with \eqref{eq:hx0},
	%	is equivalent to a plain SCF iteration for the corresponding aligned NEPv,
	%	provided the iteration matrix is sufficiently close to a
	%	\dproper~eigenbasis matrix.
	%	By this new observation,
	%	we extend the local convergence analysis in~\cite{Bai:2022}
	%	to establish a sharp estimation of the rate of convergence for the
	%	SCF-type iterations
	%	and theoretic foundation for a level-shifted SCF to fix
	%	the local divergence issue.
\end{itemize}
Extensive numerical experiments are also provided to demonstrate our theoretical results.

\smallskip\noindent
{\bf Organization.}
The rest of this paper is organized as follows.
In~\Cref{sec:nepv}, we introduce a class of NEPv without unitary invariance
as the KKT condition for an optimization problem over the Stiefel manifold.
In~\Cref{sec:property}, we discuss necessary conditions for an
eigenbasis matrix to be a global maximizer of the optimization problem,
where we will introduce  the notions of basis alignment and \dproper~eigenbasis matrix.
\Cref{sec:scf,sec:local} are devoted to SCF,
where we will propose an SCF-type iteration,
establish its connection to the plain SCF for an aligned NEPv that is unitarily
invariant, and obtain its local convergence rate.
% by exploiting the differentiability in the coefficient matrices.
\Cref{sec:ls} is on a level-shifting scheme with a theoretical foundation
for fixing the potential divergence issue of the SCF-type iteration.
Numerical experiments are presented in~\Cref{sec:egs}
and concluding remarks are made in~\Cref{sec:conclusion}.

%We follow standard notations in matrix computation.
\smallskip\noindent
{\bf Notation.}
$\bbR^{n\times m}$ is the set of $n$-by-$m$ real matrices, and $\bbO^{n\times k}$ defined in \eqref{eq:O(n,k)} is
the Stiefel manifold,
%$$
%\bbO^{n\times k}=\{X\in\bbR^{n\times k}\,:\,X^{\T}X=I_k\},
%$$
where $k\le n$ (usually $k\ll n$) and $I_k$ is the $k\times k$ identity matrix.
For a vector or matrix $B\in\bbR^{m\times n}$,
$B^{\T}$ stands for its transpose,
$\cR(B)$ for its %column subspace,
column space,
and $\cN(B)$ for its null space.
The singular values of $B$ are denoted by
$\sigma_i(B)$, for $i=1,\ldots,\min\{m,n\}$, arranged in the nonincreasing order:
$\sigma_1(B)\ge\sigma_2(B)\ge\cdots\ge\sigma_{\min\{m,n\}}(B)$.
%The thin SVD of $B$ is the one $B=U\Sigma V^{\T}$ such that $\Sigma\succ 0$
%and
%\marginpar{\tiny both norms not used}
%\sout{
%$
%\|B\|_2=\sigma_1(B),\,\,
%\|B\|_{\F}=\sqrt{\sum_{i=1}^{\min\{m,n\}}[\sigma_i(B)]^2}
%%,\,\, \|B\|_{\tr}=\sum_{i=1}^{\min\{m,n\}}\sigma_i(B)
%$
%are the spectral norm and the Frobenius norm of $B$, respectively.}
%and the trace norm (also known as the nuclear norm) of $B$,
$\|B\|$ denotes some consistent matrix norm of $B$ 
%\Blue{such as the spectral norm and the Frobenius norm}.
such as the spectral norm and the Frobenius norm.
For a square matrix $A\in\bbR^{n\times n}$, % (i.e., a linear operator in $\bbR^n$),
$\tr(A)$ and $\rho(A)$ denote, respectively,
its trace and spectral radius
(i.e., the largest absolute value of the eigenvalues of $A$).
If $A$ is also symmetric, then its eigenvalues
are enumerated from largest to smallest as
$\lambda_1(A)\ge\lambda_2(A)\ge\cdots\ge\lambda_n(A)$,
and, in particular, $\lambda_{\max}(A):=\lambda_1(A)$ and $\lambda_{\min}(X):=\lambda_n(A)$.
$A\succeq 0$ ($A\succ 0$) means that $A$ is  symmetric and positive
semidefinite (definite).
Other notations will be explained at their first appearances.

%\newpage

%
\section{A Class of NEPv Without Unitary Invariance Property}\label{sec:nepv}
In this section, we introduce a class of NEPv as the first-order
optimality condition, also known as the KKT condition,
to a particular optimization problem over the Stiefel manifold.

Throughout this paper,
for a scalar function $f$ defined on $\bbO^{n\times k}$,
its gradient at $X=[x_{ij}]$ as a matrix variable in $\bbR^{n\times k}$ is denoted as
and defined by
\begin{equation}\label{eq:f-gradR(nk)}
\frac{\partial f(X)}{\partial X} \in\bbR^{n\times k}
\quad\mbox{with}\quad
\left[ \frac{\partial f(X)}{\partial X} \right]_{ij} := \frac{\partial f(X)}{\partial x_{ij}} .
\end{equation}
More generally, for a (Fr\'echet) differentiable function
$F:\bbR^{n\times k}\to \bbR^{p\times q}$,
its Fr\'echet derivative at $X\in\bbR^{n\times k}$ along direction $Y\in\bbR^{n\times k}$,
denoted as $\D F (X)[Y]$, is defined by
\begin{equation}\label{eq:FD-defn}
\D F(X)[Y] := \lim_{t\to 0} \frac 1t\big[ F(X+tY) - F(X)\big]
     = \left.\frac{d}{dt}
	 F(\,X+tY\,)\right|_{t=0}.
\end{equation}
We can see that $\D F (X)[\,\cdot\,]: \bbR^{m\times n}\to \bbR^{p\times q}$ is a linear
operator.

For the gradient in \eqref{eq:f-gradR(nk)},
we emphasize it being defined at $X$ as a matrix variable in $\bbR^{n\times k}$,
i.e., all entries of $X$ are treated as independent,
although $X$ lives on $\bbO^{n\times k}$.
The reader should not confuse it with the notion
of gradient over the Stiefel manifold $\bbO^{n\times k}$ \cite[(3.37)]{abms:2008}.

For ease of presentation, we formally define the notion of {\em unitary invariance\/} of a
scalar or matrix-valued function on $\bbR^{n\times k}$ as follows.

\begin{definition}\label{dfn:UI}
{\rm
A function $F\colon \bbR^{n\times k}\to \bbR^{p\times q}$
is said {\em right unitarily invariant\/}, or {\em unitarily invariant\/} for short, if
$$
F(XQ)\equiv F(X)\quad\mbox{for $X\in\bbR^{n\times k}$, $Q\in\bbO^{k\times k}$}.
$$
%i.e., $F$ is essentially a function defined on the Grassmannian
%$\scrG_k^n$, the collection of dimension-$k$ subspaces in $\bbR^n$.
}
\end{definition}

%
%\subsection{An Optimization Problem with Orthogonality Constraints}
%\subsection{Quasi-unitary-invariant Optimization on the Stiefel Manifold}
\subsection{An Optimization Problem on the Stiefel Manifold}
Consider the following maximization problem over the Stiefel manifold $\bbO^{n\times k}$:
\begin{equation}\label{eq:gopt}
	\max_{X\in\bbO^{n\times k}} f(X)
 	\qquad\text{with}\qquad f(X):=\phi(X) + \psi(X)\cdot \tr(X^{\T}D),
\end{equation}
where $D\in\bbR^{n\times k}$,  $\phi$ and $\psi$ are continuously
differentiable functions in $X\in\bbR^{n\times k}$ and unitarily invariant.
We assume that $\psi$ is a
positive\footnote {This condition may be relaxed to that
        $\psi$ is non-negative, i.e., $\psi(X)\ge 0$, and $\psi(X)> 0$ at
        the KKT points of \eqref{eq:gopt}. But such conditions are hard
        to verify since the KKT points are in general unknown in the first place.}
function, i.e., $\psi(X)> 0$,
for all $X\in\bbO^{n\times k}$.
Problem~\eqref{eq:gopt} takes~\eqref{eq:max} as a special case
with
$$
\phi(X)= \frac {\tr(X^{\T}AX)}{[\tr(X^{\T}BX)]^\theta}
\quad\mbox{and}\quad
%\psi(X)= [\tr(X^{\T}BX)]^{-\theta}.
\psi(X)= \frac{1}{[\tr(X^{\T}BX)]^{\theta}}.
$$
%and recall that $\tr(X^{\T}AX)$ and $\tr(X^{\T} BX)$
They are unitarily invariant in $X$ because 
both $\tr(X^{\T}AX)$ and $\tr(X^{\T} BX)$ share this property 
(recall that matrix traces are invariant under similarity transformations and, 
in particular, we have $\tr(Q^{\T} M Q) = \tr(M)$ for all $M\in\bbR^{k\times k}$ and
$Q\in\bbO^{k\times k}$).
If $D=0$ in~\eqref{eq:gopt},
then the objective function $f$ becomes unitarily invariant,
and~\eqref{eq:gopt} is a maximization problem over the Grassmannian $\scrG_k^n$.

We will show that the
first-order optimality condition, also known as the KKT condition,
of the optimization problem~\eqref{eq:gopt}
can be formulated as an NEPv and any optimizer is an orthonormal eigenbasis matrix of the NEPv.
Further necessary conditions for the eigenbasis matrix
to be a global maximizer will be derived in~\Cref{sec:property}.

\subsection{From the First-Order Optimality Condition to NEPv}
Traditionally, problem~\eqref{eq:gopt} is viewed
as an optimization problem with orthogonality constraints:
\begin{equation}\label{eq:goptc}
	\max_{X\in\bbR^{n\times k}} f(X)\quad\text{s.t.}\quad X^{\T}X=I_k.
\end{equation}
We will derive the optimality conditions of~\eqref{eq:goptc} by the standard method
of Lagrange's multipliers.\footnote {Alternatively, we can use the
	Riemannian gradient of a scalar function over
	the Stiefel manifold $\bbO^{n\times k}$ to derive the KKT condition, as
  in \cite{wazl:2022} (see also \cite[(3.37)]{abms:2008}).}
Let us begin with a useful expression for the gradients of
unitarily invariant functions.

\begin{lemma}\label{lem:dfx}
	Let $\phi\colon \bbR^{n\times k} \to \bbR$ be a differentiable function
	that is unitarily invariant.
	Then there exists a matrix-valued function $H_{\phi}(X)\in\bbR^{n\times
	n}$, which is continuous and unitarily invariant in $X\in \bbR^{n\times k}$, 
	and symmetric, such that the gradient
	of $\phi$ admits %an expression
	\begin{equation}\label{eq:dfx}
	\frac {\partial \phi(X)}{\partial X}	%\partial \phi(X)
      = H_{\phi}(X)\cdot X
	\end{equation}
	%for all orthonormal $X\in\bbR^{n\times k}$,
	for $X\in\bbR^{n\times k}$ with orthonormal columns.
Furthermore,  $H_{\phi}$ can be made differentiable if $\phi$ is twice differentiable.
In general, such $H_{\phi}(X)$ is not unique, and one of them is
	\begin{equation}\label{eq:hfx}
		H_{\phi}(X) = \frac {\partial \phi(X)}{\partial X}\cdot X^{\T}
           + X\cdot \left[\frac {\partial \phi(X)}{\partial X}\right]^{\T} -
		X\cdot\left(\left[\frac {\partial \phi(X)}{\partial X}\right]^{\T} X
		\right)\cdot X^{\T}.
	\end{equation}
\end{lemma}

\begin{proof}

	We will show that the matrix in~\eqref{eq:hfx} satisfies all the requirement.
	The validity of the expression \eqref{eq:dfx} with \eqref{eq:hfx} is a
	consequence of the following more general fact: Given $X\in\bbR^{n\times
	k}$ with $X^{\T}X=I_k$, it holds for any matrix $C\in\mathbb R^{n\times k}$ that
% $C$ as $X$ left-multiplied by some square matrix via
	\[
		C = CX^{\T} X
		= (CX^{\T} + XC^{\T} - XC^{\T}) \cdot X
		= (CX^{\T} + XC^{\T} - XC^{\T}XX^{\T}) \cdot X.
	\]
	Letting $C=\frac {\partial \phi(X)}{\partial X}$ yields \eqref{eq:dfx} with \eqref{eq:hfx}.
	%By its expression \eqref{eq:hfx}, we can
	It can been seen that $H_{\phi}(\,\cdot\,)$ in \eqref{eq:hfx}
	is continuous and  also differentiable if $\phi$ is twice differentiable.
%	we can construct a matrix $H_{\phi}(X)$ satisfying~\eqref{eq:dfx} given by

	It remains to show that $H_{\phi}$ in~\eqref{eq:hfx} is symmetric and unitarily invariant.
	To that end, we will first establish the following identities for the derivative of
	unitarily invariant function  $\phi$: For $X\in\bbR^{n\times k}$,
	$Q\in\bbO^{k\times k}$, and $Y:=XQ$, we have
	\begin{equation}\label{eq:partialphixx}
		\frac{\partial \phi(XQ)}{\partial Q} = X^{\T} \frac{\partial \phi(X)}{\partial X}Q
		\quad \text{and}\quad
    	\frac {\partial \phi(Y)}{\partial Y}=\frac {\partial \phi(X)}{\partial
		X}\cdot Q\,.
	\end{equation}
	These two identities in \eqref{eq:partialphixx} follow directly from
the following first-order expansions: Perturbing $Q$ to $Q+\delta Q$ and $Y\equiv XQ$ to $Y+\delta Y$, we have
\begin{align*}
\phi(X(Q+\delta Q))&=\phi(X(I+\delta Q\cdot Q^{\T})Q) \\
       &=\phi(X+X\cdot\delta Q\cdot Q^{\T}) \\
	   &= \phi(X) +
		\tr\left(\left[\frac{\partial \phi(X)}{\partial X}\right]^{\T}
		X\cdot\delta Q \cdot Q^{\T}\right) + \mathcal O(\|\delta Q\|^2) \\
	   &= \phi(X) +
		\tr\left(Q^{\T}\left[\frac{\partial \phi(X)}{\partial X}\right]^{\T}
		X\cdot\delta Q  \right) + \mathcal O(\|\delta Q\|^2), \\
\intertext{and}
\phi(Y+\delta Y)&=\phi((X+\delta Y\cdot Q^{\T})Q) \\
       &=\phi(X+\delta Y\cdot Q^{\T}) \\
       &=
		\phi(X) + \tr\left(\left[\frac{\partial \phi(X)}{\partial X}\right]^{\T} \cdot\delta Y\cdot Q^{\T}\right)
+ \mathcal O(\|\delta Y\|^2) \\
       &=
		\phi(X) + \tr\left(Q^{\T}\left[\frac{\partial \phi(X)}{\partial X}\right]^{\T} \cdot\delta Y \right)
+ \mathcal O(\|\delta Y\|^2).
\end{align*}
%	\[
%		\phi(X(Q+\delta Q)) \equiv \phi(X+X\cdot\delta Q\cdot Q^{\T}))
%		= \phi(X) +
%		\tr\left(\left[\frac{\partial \phi(X)}{\partial X}\right]^{\T}
%		X\cdot\delta Q \cdot Q^{\T}\right) + \mathcal O(\|\delta Q\|^2),
%	\]
%	and
%	\[
%		\phi(Y+\delta Y) \equiv \phi(X+\delta YQ^{\T}) =
%		\phi(X) + \tr\left(\left[\frac{\partial \phi(X)}{\partial X}\right]^{\T} \cdot\delta Y Q^{\T}\right) + \mathcal O(\|\delta Y\|^2),
%	\]
%	where $\delta Q\in\bbR^{k\times k}$ and $\delta Y \in \bbR^{n\times k}$ are perturbation
%	matrices, and the `$\equiv$' are by the unitary invariance of $\phi$.

	Next, we claim that $X^{\T}\frac {\partial \phi(X)}{\partial X}$ is always
	symmetric, from which the symmetry of $H_{\phi}(X)$ in \eqref{eq:hfx}
	follows immediately.
	To verify this claim,
	we notice that $\phi(XQ)$ with a fixed $X\in\bbR^{n\times k}$ is a constant over
	$Q\in\bbO^{k\times k}$ by unitary invariance.
	Therefore, any $Q\in\bbO^{k\times k}$ is an optimal solution to
    $$
    \max_{Q\in\bbR^{k\times k}}
	\phi(XQ)
	\quad\text{s.t.}\quad Q^{\T}Q=I_k,
    $$
	because the objective function has a constant value.
	By the method of Lagrange's multipliers, we have
	for any $Q\in\bbO^{k\times k}$, there exists a symmetric multiplier
	$\Gamma\in\bbR^{k\times k}$ such that $\frac {\partial \cL(Q,\Gamma)}{\partial Q}  = 0$, where
	\[
    \cL(Q,\Gamma) := \phi (XQ) -
		\frac{1}{2}\tr\left([Q^{\T}Q - I_k]\cdot \Gamma\right)
	\]
	is the associated Lagrangian function.
	By the first identity in~\eqref{eq:partialphixx}, we obtain immediately
	\[
	\frac {\partial \cL(Q,\Gamma)}{\partial Q}
       = X^{\T} \frac{\partial \phi(X)}{\partial X}Q - Q\Gamma = 0.
	\]
	Letting $Q=I_k$ yields $X^{\T}\cdot\frac {\partial \phi(X)}{\partial X} = \Gamma$, which is
	symmetric, as was to be shown.

	Finally, by the second identity in~\eqref{eq:partialphixx},
	it can be verified that $H_{\phi}(XQ) = H_{\phi}(X)$ for $H_{\phi}(\,\cdot\,)$ given by \eqref{eq:hfx}.
	Namely, $H_{\phi}(\,\cdot\,)$ is unitarily invariant.

	We defer addressing the non-uniqueness of $H_{\phi}(X)$ satisfying
	\eqref{eq:dfx} to Remark~\ref{rmk:dfx} below.
\end{proof}

\begin{remark}\label{rmk:dfx}
{\rm
%\Cref{lem:dfx} shows a novel way to express the %Euclidean
%gradient of a unitarily invariant function as in~\eqref{eq:dfx},
%and in its proof a concrete example
%of $H_{\phi}$ is provided by~\eqref{eq:hfx}.
In general, the choice of continuous and unitarily invariant $H_{\phi}(X)$ to satisfy \eqref{eq:dfx} is not unique.
For a particular functions $\phi$,
one may  select an $H_{\phi}$ that is more convenient and efficient to work with.
For example, for $\phi(X) = h\left(\tr(X^{\T}AX)\right)$, a composition
of a differentiable function $h$ with a quadratic form of a symmetric matrix
$A$, we have
\[
%	
%	\qquad \Rightarrow\qquad
	\frac {\partial \phi(X)}{\partial X} = 2h'\left(\tr(X^{\T}AX)\right)\cdot  A X,
\]
where $h'$ is the first-order derivative of $h$.
In this case, we have immediately
\begin{equation}\label{eq:rk4non-uniqueness}
	\frac {\partial \phi(X)}{\partial X} = H_{\phi}(X)\cdot X
	\qquad\text{with}\qquad
	 H_{\phi}(X) = 2h'\left(\tr(X^{\T}AX)\right) \cdot A.
\end{equation}
In comparison, the general formula~\eqref{eq:hfx} leads to
%a more complicated one:
\[
	H_{\phi}(X) = 2 h'\left(\tr(X^{\T}AX)\right)\cdot
	\left[AXX^{\T} + XX^{\T}A - X(X^{\T}AX)X^{\T}  \right],
\]
a much more complicated one than that in \eqref{eq:rk4non-uniqueness}.
}
\end{remark}

By~\Cref{lem:dfx},
we can write the gradients %\sout{partial derivative}
of unitarily invariant $\phi$ and $\psi$ as
\begin{equation}\label{eq:der}
	\frac {\partial \phi(X)}{\partial X} = H_{\phi}(X)\cdot X
	\qquad \text{and}\qquad
	\frac {\partial \psi(X)}{\partial X} = H_{\psi}(X)\cdot X,
\end{equation}
where $H_{\phi}(X)\in\bbR^{n\times n}$ and $H_{\psi}(X)\in\bbR^{n\times n}$
are symmetric and
unitarily invariant. In what follows, we assume some forms of
$H_{\phi}(X)$ and $H_{\psi}(X)$
have been selected to fulfill \eqref{eq:der}, but we will not specify
which particular ones are used. In fact, our development will work with any choice.

\begin{theorem}\label{thm:nepv}
$X\in\bbO^{n\times k}$ is a KKT point of~\eqref{eq:gopt} if and only if
$D^{\T}X$ is symmetric
and $X$ satisfies the NEPv
\begin{subequations}\label{eq:nepv2}
\begin{equation}\label{eq:nepv2-1}
	H(X)X = X\Lambda,
	%\qquad X^{\T}X = I,
\end{equation}
where $H(X)$ is symmetric and given by
\begin{equation}\label{eq:hx22}
	H(X) = H_{\phi}(X) + \tr(X^{\T}D)\cdot H_{\psi}(X) +\psi(X)\cdot
	(DX^{\T}+XD^{\T}).
\end{equation}
\end{subequations}
\end{theorem}

\begin{proof}
We treat~\eqref{eq:gopt} as a constrained optimization problem
subject to $X^{\T}X=I_k$ as in~\eqref{eq:goptc},
for which the Lagrangian function of multipliers is written as
\[
	\cL(X,\Gamma) :=
	\phi(X) + \psi(X)\cdot \tr(X^{\T}D) -  \frac{1}{2}\tr([X^{\T}X-I_k]\Gamma),
\]
where
$\Gamma\in\bbR^{k\times k}$ is a symmetric matrix of multipliers.
Then $X\in\bbO^{n\times k}$ is a KKT point of~\eqref{eq:gopt}
if and only if it satisfies
the first-order optimality condition $\frac{\partial \cL}{\partial X} = 0$ for
some symmetric $\Gamma$, i.e.,
\[
	\frac {\partial \phi(X)}{\partial X}
	+\tr(X^{\T}D)\cdot \frac {\partial \psi(X)}{\partial X}
	+\psi(X)\cdot D
	= X\Gamma,
\]
which, by~\eqref{eq:der}, is equivalent to
\begin{equation}\label{eq:1stopt2}
	H_{\phi}(X)X
	+\tr(X^{\T}D)\cdot
	H_{\psi}(X)X
	+\psi(X)\cdot D
	= X\Gamma.
\end{equation}

If $X\in\bbO^{n\times k}$ is a KKT point of~\eqref{eq:gopt},
then~\eqref{eq:1stopt2} yields \eqref{eq:nepv2}
with $\Lambda = \Gamma + \psi(X)\cdot D^{\T}X$.
Moreover, since $\Lambda\equiv X^{\T}H(X)\,X$ is symmetric, we have
$D^{\T}X=(\Lambda - \Gamma)/\psi(X)$ is also symmetric.

Conversely, if $X\in\bbO^{n\times k}$ satisfies \eqref{eq:nepv2}
and $D^{\T}X$ is symmetric, then we have that \eqref{eq:1stopt2} holds with
a symmetric $\Gamma=\Lambda-\psi(X)\cdot D^{\T}X$.
Hence $X$ is a KKT point.
\end{proof}

For the special case of $D=0$,
the optimization problem~\eqref{eq:gopt} and
NEPv~\eqref{eq:nepv2} become
\begin{equation}\label{eq:maxuni}
	\max_{X\in\bbO^{n\times k}} \phi(X)
	\qquad \text{and}\qquad
	H_{\phi}(X)X=X\Lambda,
\end{equation}
respectively,
where both $\phi$ and $H_{\phi}$ are unitarily invariant.
So, optimizing a unitarily invariant function on the Stiefel manifold
$\bbO^{n\times k}$
always leads to a unitarily invariant NEPv.
For the general case of $D\neq 0$, due to the presences of $DX^{\T}$ and $X^{\T}D$,
the coefficient matrix $H(\,\cdot\,)$ in \eqref{eq:hx22} is not
unitarily invariant any more;
see the particular example~\eqref{eq:hx0},
obtained by an application of~\Cref{thm:nepv} to~\eqref{eq:max}.

As a first-order optimality condition,
\Cref{thm:nepv} does not specify which $k$ eigenvalues of $H(X)$
correspond to those of $\Lambda$ for the purpose of solving the underlying optimization problem.
For a local maximizer $X$ of~\eqref{eq:gopt}, the corresponding eigenvalues are
typically the $k$ largest ones.
This has been proven for particular optimizations in the form of~\eqref{eq:gopt}
(e.g., in~\cite{wazl:2022,zhli:2014a,zhln:2010,zhln:2013,zhwb:2020,Zhang:2020}),
and it is also a common practice for
handling the NEPv from related optimization problems
in electronic structure calculation (e.g., in~\cite{Cances:2010,Yang:2007}), where
the eigenvalues are the $k$ smallest ones because they are about minimization.
But there exist cases for which not all of the $k$ eigenvalues are the extreme ones,
and when that happens, numerical difficulties arise.
%when it comes to numerically solve the NEPv.
We will come back to this issue in~\Cref{sec:ls}.

%
%\section{Desirable Solution Properties of the NEPv}\label{sec:property}

%Ideally, we should seek those solutions of NEPv \eqref{eq:nepv2} that
%are the global maximizers of
%the optimization problem \eqref{eq:gopt}. In practice, it is a folklore
%that finding a global optimizer with guarantee is in general impossible,
%unless for a convex programming or a very special one,
%so practitioners often settle for some local optimizers.
%
%%
%For our case,  NEPv~\eqref{eq:nepv2} may have infinitely many
%solutions. What we would like to do
%is to find a decent (local) maximizers for \eqref{eq:gopt}.
%For that purpose, we will first establish
%two more necessary conditions for a global maximizer of
%\eqref{eq:gopt}, beyond its KKT condition
%given as the NEPv~\eqref{eq:nepv2}.
%The necessary conditions will serve as guides to what solutions
%to NEPv~\eqref{eq:nepv2} we should look for.

%
\section{Necessary Conditions for Global Maximizers}\label{sec:property}
Ideally, we should seek those solutions of NEPv \eqref{eq:nepv2} that
are the global maximizers of  optimization problem \eqref{eq:gopt}.
For that purpose, we will first establish
two necessary conditions for a global maximizer of \eqref{eq:gopt},
beyond its KKT condition given as  NEPv~\eqref{eq:nepv2}.
Specifically, they are
\begin{align}
\mbox{\em definiteness:} &\quad X^{\T}D\succeq 0, \label{cond:ND} \\
%\mbox{\em rank-consistency:}
\mbox{\em rank-preserving:}
						 &\quad \mrank(X^{\T}D ) = \mrank (D). \label{cond:rank}
\end{align}
The two conditions above will serve as guides to what solutions
to NEPv~\eqref{eq:nepv2} we should look for.
The definiteness condition \eqref{cond:ND}
has been known and successfully exploited
in existing works for a few special cases of~\eqref{eq:gopt};
see, e.g.,~\cite{wazl:2022} and reference therein.
In contrast, the \rankd~condition \eqref{cond:rank} %for the global maximality
is a new discovery,
which is also a crucially missing piece for analyzing the local convergence
of SCF for solving NEPv~\eqref{eq:nepv2}. In fact,
this new condition makes our systematic treatment in the rest of this paper possible.

% summarize

Our main theorem in this section is \Cref{thm:global} below,
and its proof will be given towards the end of this section after
we fully investigate the definiteness condition \eqref{cond:ND}
and the \rankd\ condition \eqref{cond:rank} separately.

\begin{theorem}~\label{thm:global}
	Let $X_*\in\bbO^{n\times k}$ be a global maximizer of~\eqref{eq:gopt}.
	Then $X_*$ satisfies both the definiteness and
	\rankd~conditions, i.e., \eqref{cond:ND} and \eqref{cond:rank} hold with $X=X_*$.
\end{theorem}

\subsection{Definiteness and Basis Alignment}\label{sec:basis}
Let $X\in\bbO^{n\times k}$ be a given approximate solution
to  optimization problem~\eqref{eq:gopt}.
Since the objective function $f(X)$ is not unitarily invariant,
it is possible to find a better approximate solution
(in the sense of a larger objective value) in the form of $XQ$
with $Q\in\bbO^{k\times k}$ determined by
\begin{equation}\label{eq:restrict}
	\max_{Q\in\bbO^{k\times k}} f(X Q).
\end{equation}
As we will show shortly,
the construction of such best $XQ$
has a lot to do with the definiteness condition \eqref{cond:ND}.
We will refer to finding the best $XQ$ via \eqref{eq:restrict} as the problem of
{\em basis alignment\/} because essentially $Q$ picks up a new
orthonormal basis matrix $\widetilde X=XQ$ for the column space of $X$.
%Specifically, we
%aim to improve an approximate solution
%$X$ of \eqref{eq:gopt}  by an orthogonal transformation from the right,
%can be described as
In fact, any maximizer $Q$ of \eqref{eq:restrict} will yield an improved  solution
$\widetilde X$
for optimization problem~\eqref{eq:gopt},
in that $f(\widetilde X)\geq f(X)$, and usually the inequality is strict unless
$Q=I_k$ is a maximizer of \eqref{eq:restrict}. This idea can be regarded as seeking a
best solution to \eqref{eq:gopt} in the orbit $\{XQ\,:\,Q\in\bbO^{k\times k}\}$.
%We call this process the \emph{basis alignment}
For convenience of discussion,
we introduce an \emph{alignment function} for $X\in\bbO^{n\times k}$ as
\begin{equation}\label{eq:align0}
		\al{X} : = \left\{ X Q \,\colon\, Q \in  \argmax_{Q\in\bbO^{k\times k}} f(X Q)
		\right\}.
\end{equation}
Note that $\al{X}$ may be multi-valued.

%%
%From the definition of $f$ in~\eqref{eq:gopt},
%since both $\phi$ and $\psi$ are unitarily invariant,
%we have
%\begin{equation}\label{eq:fxoq}
%	f(X Q) =
% \phi(X) + \psi(X)\cdot \tr(Q^{\T}X^{\T}D).
%\end{equation}
%By assumption $\psi(X)> 0$, problem \eqref{eq:restrict} is therefore equivalent to
%\begin{equation}\label{eq:maxfxs2}
% \max_{Q\in\bbO^{k\times k}} \tr\left(Q^{\T}(X^{\T}D)\right).
%\end{equation}
%Problem \eqref{eq:maxfxs2} is a classical matrix optimization problem,
%which also arises in, e.g.,
%the orthogonal Procrustes~\cite[Section~6.4.1]{Golub:2013},
%and it \Blue{is known to have} a closed-form solution via SVD, as shown in~\Cref{lem:basis},
%\Blue{but the latter
%reveals more structural properties in solution than traditionally known.}
%%the orthogonal Procrustes~\cite[Sec.  12.4]{Golub:1996}.
%We notice that the ideas for the alignment~\eqref{eq:align0}
%and its solution via~\eqref{eq:maxfxs2} have occurred
%in previous works on special cases of~\eqref{eq:restrict};
%see, e.g.,~\cite{wazl:2022,zhwb:2020,Zhang:2020}.
%\Blue{
%Recently, the authors of \cite{teli:2023a} viewed \eqref{eq:maxfxs2} as an optimization problem
%over the orthonormal basis matrices of the given subspace $\cR(X)$ and investigated how the optimal
%basis matrices vary as the subspace varies.
%}
%\Cref{lem:basis} below is essentially \cite[Lemma 3.2]{wazl:2022} but stated differently,
%and a  proof is provided here for self-containedness
%and it is also simpler than the one in~\cite{wazl:2022}.

%
From the definition of $f$ in~\eqref{eq:gopt},
since both $\phi$ and $\psi$ are unitarily invariant,
we have
\begin{equation}\label{eq:fxoq}
	f(X Q) =
 \phi(X) + \psi(X)\cdot \tr(Q^{\T}X^{\T}D).
\end{equation}
By assumption $\psi(X)> 0$, problem \eqref{eq:restrict} is therefore equivalent to
\begin{equation}\label{eq:maxfxs2}
 \max_{Q\in\bbO^{k\times k}} \tr\left(Q^{\T}(X^{\T}D)\right).
\end{equation}
Problem \eqref{eq:maxfxs2} is a classical matrix optimization problem,
which also arises in, e.g.,
the orthogonal Procrustes~\cite[Section~6.4.1]{Golub:2013},
and it is known to have a closed-form solution via SVD, as shown in~\Cref{lem:basis},
but the latter
reveals more structural properties in solution than those traditionally known.
%the orthogonal Procrustes~\cite[Sec.  12.4]{Golub:1996}.
We notice that the ideas for the alignment~\eqref{eq:align0}
and its solution via~\eqref{eq:maxfxs2} have occurred
in previous works on special cases of~\eqref{eq:restrict};
see, e.g.,~\cite{wazl:2022,zhwb:2020,Zhang:2020}.
Recently, the authors of \cite{teli:2023a} viewed \eqref{eq:maxfxs2} as an optimization problem
over the orthonormal basis matrices of the given subspace $\cR(X)$ and investigated how the optimal
basis matrices vary as the subspace varies.

\Cref{lem:basis} below is essentially \cite[Lemma 3.2]{wazl:2022} but stated differently,
and a  proof is provided here for self-containedness
and it is also simpler than the one in~\cite{wazl:2022}.

\begin{lemma}[\cite{wazl:2022}]\label{lem:basis}
Let the singular value decomposition (SVD) of $X^{\T}D\in\bbR^{k\times k}$ be
\begin{equation}\label{eq:svdxtd}
	X^{\T}D = U\Sigma V^{\T}
     \equiv\kbordermatrix{ &\sss\ell &\sss k-\ell\\
                           &U_1 & U_2}
          \times\kbordermatrix{ &\sss\ell &\sss k-\ell\\
                          \sss\ell & \Sigma_1 & \\
                          \sss k-\ell & & 0}
          \times\kbordermatrix{ &\\
                          \sss\ell & V_1^{\T} \\
                          \sss k-\ell & V_2^{\T}},
%[U_1,U_2] \begin{bmatrix} \Sigma_1 & \\ & 0 \end{bmatrix}	[V_1,V_2]^{\T},
\end{equation}
where
%$\Sigma_1\in\bbR^{\ell\times \ell}$ is a diagonal matrix containing the nonzero singular values,
$\ell=\mrank(X^{\T}D)$ and $\Sigma_1=\diag(\sigma_1,\sigma_2,\ldots,\sigma_{\ell})\succ 0$.
%, and
%$U= [U_1,U_2]\in \bbO^{k\times k}$ and $V=[V_1,V_2]\in\bbO^{k\times k}$ with
%$U_1,\,V_1\in\bbR^{k\times \ell}$.
Then,
\begin{enumerate}[{\rm (a)}]
	\item\label{lem:basis:ia}
	$Q\in\bbO^{k\times k}$ is a global maximizer of~\eqref{eq:restrict}
	if and only if
	\begin{equation}\label{eq:qform}
		Q = U \begin{bmatrix} I_\ell & \\ & \Omega \end{bmatrix} V^{\T},
	%	=  U_1V_1^{\T} + U_2\Omega V_2^{\T},
	\end{equation}
	where $\Omega\in\bbO^{(n-\ell)\times (n-\ell)}$ is arbitrary;

	\item\label{lem:basis:ib}
	$Q\in\bbO^{k\times k}$ takes the form of \eqref{eq:qform} if and only if
	$Q^{\T}(X^{\T}D) \succeq 0$.
\end{enumerate}
%Note that the matrix $Q_1:=U_1V_1^{\T}$ is a partial isometry that is independent of the choice of SVD~\eqref{eq:svdxtd}.
\end{lemma}

\begin{proof}
For item~\eqref{lem:basis:ia}, it follows from SVD~\eqref{eq:svdxtd} that
\begin{equation}\label{eq:trmax}
\tr\left(Q^{\T}[X^{\T}D]\right) = \tr\left( \widetilde Q^{\T}  \begin{bmatrix}
		\Sigma_1 & \\ & 0
	\end{bmatrix}\right)
	= \sum_{i=1}^\ell \sigma_i \cdot \widetilde Q_{ii}
	\leq \sum_{i=1}^\ell \sigma_i,
\end{equation}
where $\widetilde Q = U^{\T}QV$ is orthogonal and $ \widetilde Q_{ii}$ is the $i$th diagonal entry of
$\widetilde Q$.
The last equation in \eqref{eq:trmax} is due
to $|\widetilde Q_{ii}|\leq 1$ for all $i$.
Clearly, the right hand side achieves the maximum
$\sum_{i=1}^\ell \sigma_i$
if and only if $\widetilde Q_{ii}=1$ for $1\leq i\leq \ell$.
Therefore, the optimal $\widetilde Q = \begin{bmatrix} I_\ell & \\ & \Omega \end{bmatrix}$,
where $\Omega\in\bbO^{(n-\ell)\times (n-\ell)}$.
Substituting it back to $Q = U \widetilde Q V^{\T}$, we obtain~\eqref{eq:qform}.

For item~\eqref{lem:basis:ib}, it can be verified that $Q^{\T}(X^{\T}D)\succeq 0$ for
$Q$ given by~\eqref{eq:qform}.
Conversely, if  $Q\in\bbO^{k\times k}$ satisfies $Q^{\T}(X^{\T}D)\succeq 0$, then
\[
\tr\left(Q^{\T}(X^{\T}D)\right)
= \sum_{i=1}^\ell \sigma_i \left(Q^{\T}(X^{\T}D)\right)
= \sum_{i=1}^\ell \sigma_i \left(X^{\T}D\right)
= \sum_{i=1}^\ell \sigma_i,
\]
where the first equality is because of the eigenvalues of
$Q^{\T}(X^{\T}D)$ are the same  as its singular values.
By~\eqref{eq:trmax}, $Q$ is a global maximizer of~\eqref{eq:maxfxs2}
and must take the form of~\eqref{eq:qform} by item~\eqref{lem:basis:ia}.
\end{proof}

% connection to polar
The results in~\Cref{lem:basis}
can alternatively be interpreted
using the polar decomposition of $X^{\T}D$,
as done traditionally for the solution of the Procrustes problem.
The polar decomposition of $X^{\T}D$ refers to
\begin{equation}\label{eq:polar}
	X^{\T}D = Q M \quad \text{with $Q\in\bbO^{k\times k}$ and $M\succeq 0$.}
\end{equation}
The polar factors $Q$ and $M$ in~\eqref{eq:polar}
can be expressed as $Q$ in~\eqref{eq:qform}
and $M=V\Sigma V^{\T}$, via SVD \eqref{eq:svdxtd}.
If $X^{\T}D$ is non-singular,
then the polar decomposition~\eqref{eq:polar} is unique,
and thus the orthogonal polar factor $Q=UV^{\T}$ is also uniquely defined
(i.e., independent of any inherent freedom in SVD~\eqref{eq:svdxtd});
see, e.g., ~\cite[p.220]{begr:2003}, \cite[Chapter~8]{Higham:2008},
and also \cite{li:1993b,li:1995,li:2014HLA,lisu:2002}.
Otherwise, when $X^{\T}D$ is singular,
orthogonal polar factor $Q$ is %becomes 
non-unique,
and \Cref{lem:basis} reveals the inherent structure of optimal
$Q$ with the form of~\eqref{eq:qform}. The structure of this solution form %turns out to be
is crucial in our subsequent analysis.

By~\Cref{lem:basis}, any orthogonal polar factor $Q$ in~\eqref{eq:polar}
is a maximizer of~\eqref{eq:maxfxs2}, and vice versa.
Moreover,
\begin{equation}\label{eq:maxfxs3}
 \max_{Q\in\bbO^{k\times k}} \tr\left(Q^{\T}(X^{\T}D)\right) =
 \tr\left(M\right).
\end{equation}
%
%The polar factors $Q$ and $M$, via SVD \eqref{eq:svdxtd}, can be expressed
%as a $Q$ in~\eqref{eq:qform} and $M=V\Sigma V^{\T}$.
%By~\Cref{lem:basis}, the orthogonal polar factor $Q$ in~\eqref{eq:polar}
%is a maximizer of~\eqref{eq:maxfxs2}, and vice versa.
%Moreover,
%\begin{equation}\label{eq:maxfxs3}
% \max_{Q\in\bbO^{k\times k}} \tr\left(Q^{\T}(X^{\T}D)\right) =
% \tr\left(M\right).
%\end{equation}
%If $X^{\T}D$ is non-singular,
%then the polar decomposition~\eqref{eq:polar} is unique
%so that the orthogonal polar factor $Q=UV^{\T}$ is uniquely defined
%(i.e., independent of any inherent freedom in SVD~\eqref{eq:svdxtd});
%see, e.g., ~\cite[p.220]{begr:2003}, \cite[chap.~8]{Higham:2008},
%and also \cite{li:1993b,li:1995,li:2014HLA,lisu:2002}.
%
The following results are direct consequences of~\Cref{lem:basis}.

\begin{corollary}\label{cor:al(X)-characterization}
Given $X\in\bbO^{n\times k}$, let the SVD of $X^{\T}D$ be given by \eqref{eq:svdxtd}
with $\ell=\mrank(X^{\T}D)$, and define $\al{X}$ as in \eqref{eq:align0}.
\begin{enumerate}[{\rm (a)}]
  \item \label{cor:al:ia}
	  We have
        \begin{subequations}\label{eq:al(X)-characterization}
        \begin{align}
        \al{X} &= \left\{ X Q \,\colon\, Q \in  \argmax_{Q\in\bbO^{k\times k}}
                \tr\left(Q^{\T}(X^{\T}D)\right) \right\} \label{eq:al(X)-characterization-1}\\
               &=\left\{
		X Q
		\,\colon\,
		Q =  U \begin{bmatrix} I_\ell & \\ & \Omega \end{bmatrix} V^{\T},\
		\Omega\in\bbO^{(k-\ell)\times (k-\ell)}
		\right\}. \label{eq:al(X)-characterization-2}
        \end{align}
        \end{subequations}
		Thus, $ \al{X} =  \{X\,(UV^{\T})\}$ contains just one
		element if $\ell = k$, but if $\ell < k$ then $\al{X}$ contains infinitely many elements.
  \item \label{cor:al:ib}
	  $\widetilde X \in\al{X}$
		if and only if $\widetilde X = XQ$ for some $Q\in\bbO^{k\times k}$ such that
		$\widetilde X^{\T} D \succeq 0$.
  \item \label{cor:al:ic}
	  $X \in\al{X}$ if and only if $X^{\T} D \succeq 0$.
  \item \label{cor:al:id}
	  $\al{X} = \al{XQ}$ for any $Q\in\bbO^{k\times k}$.
  \item \label{cor:al:ie}
	  $D^{\T} X_a =  D^{\T} X_b$ and
        $X_a D^{\T}X =  X_b D^{\T}X$ for any two elements $X_a,X_b\in \al{X}$.
\end{enumerate}
\end{corollary}

\begin{proof}
For item~\eqref{cor:al:ia},
equation~\eqref{eq:al(X)-characterization-1} follows from the equivalency
between optimization problems~\eqref{eq:restrict} and
\eqref{eq:maxfxs2},
while equation~\eqref{eq:al(X)-characterization-2} just restates
\Cref{lem:basis}(a).
That $\al{X}$ has just one element when $\ell=k$ is a consequence of
the uniqueness of the corresponding polar decomposition~\eqref{eq:maxuni} \cite{li:2014HLA}.
Item~\eqref{cor:al:ib} restates
\Cref{lem:basis}\eqref{lem:basis:ib},
which then leads to item~\eqref{cor:al:ic}.
For item~\eqref{cor:al:id}, we notice that $\al{X}$ consists of all orthonormal
basis matrices of the subspace $\cR(X)$, at which $f$ achieves its maximum
in the orbit $\{XQ\,:\,Q\in\bbO^{k\times k}\}$.
Since $\cR(X) = \cR(XQ)$, we have $\al{X} = \al{XQ}$. Finally,
the equations in item~\eqref{cor:al:ie} can be verified by the general expression for $Q$ in~\eqref{eq:qform}.
\end{proof}

\subsection{Rank-preserving}\label{sec:proper}
We now establish several equivalent statements
to  \rankd~condition~\eqref{cond:rank}.
%  for an $X\in\bbO^{n\times k}$.

% main
\begin{theorem}~\label{thm:rank-consistency}
Let $X\in\bbO^{n\times k}$ and $H(\,\cdot\,)$ be as in~\eqref{eq:hx22}.
	The following statements are equivalent.
	\begin{enumerate}[{\rm (a)}]
		\item\label{lem:rankeq:i:rk}
			$X$
			 satisfies  \rankd~condition \eqref{cond:rank}, i.e.,
            $\mrank(X^{\T}D) = \mrank(D)$.
		\item\label{lem:rankeq:i:null}
			$DV_2=0$, where $V_2$, a basis matrix of the null space
			$\cN(X^{\T}D)$, is from SVD~\eqref{eq:svdxtd}.
		\item\label{lem:rankeq:i:dx}
			$D X_a^{\T} =  DX_b^{\T}$ for all $X_a,X_b\in \al{X}$.
		\item\label{lem:rankeq:i:h}
			$H(X_a) = H(X_b)$ for all $X_a,X_b\in \al{X}$.
		\item\label{lem:rankeq:i:hx}
			$H(X_a)X = H(X_b)X$ for all $X_a,X_b\in \al{X}$.
	\end{enumerate}
\end{theorem}

\begin{proof}
	\eqref{lem:rankeq:i:rk}~$\Leftrightarrow$~\eqref{lem:rankeq:i:null}:
	Apply the rank-nullity theorem to $D$ and $X^{\T} D$, respectively, to get
	\[
		\mrank(D) + \mbox{dim}(\cN(D))=k=
		\mrank(X^{\T}D) +  \mbox{dim}(\cN(X^{\T}D)).
	\]
	Hence, $\mrank(D) = \mrank(X^{\T}D)$
	if and only if
	$\mbox{dim}(\cN(D)) = \mbox{dim}(\cN(X^{\T}D))$.
	The latter, since $\cN(D) \subset \cN(X^{\T}D)$,
	holds if and only if
	$\cN(X^{\T}D)\subset \cN(D)$, which is equivalent to
	$DV_2 = 0$ with $V_2$
	%being the basis of the null space of $X^{\T}D$
	from SVD~\eqref{eq:svdxtd}.

	\eqref{lem:rankeq:i:null}~$\Leftrightarrow$~\eqref{lem:rankeq:i:dx}:
	It follows from~\eqref{eq:qform} that there are $Q_a$ and $Q_b$ taking the
	form of \eqref{eq:qform} such that
	\begin{equation}\label{eq:xaxb}
		X_a = XQ_a = X (U_1V_1^{\T} + U_2\Omega_aV_2^{\T}), \quad
		X_b = XQ_b = X (U_1V_1^{\T} + U_2\Omega_bV_2^{\T}),
	\end{equation}
	for some $\Omega_a,\Omega_b\in\bbO^{(k-\ell)\times (k-\ell)}$.
	If $DV_2 = 0$, then
	\[
		DX_a^{\T} = DV_1U_1^{\T}X^{\T} =  D X_b^{\T},
	\]
	which establishes item~\eqref{lem:rankeq:i:dx}.
	Conversely, if item~\eqref{lem:rankeq:i:dx} holds, then
	by taking two particular $X_a$ and $X_b$ from~\eqref{eq:xaxb}
	with $\Omega_a=I_{k-\ell}$ and $\Omega_b=-I_{k-\ell}$, we have
	\[
		0 = DX_a^{\T} - DX_b^{\T} = 2\cdot D V_2 U_2^{\T}X^{\T}.
	\]
    Post-multiplication by $XU_2$ yields $DV_2=0$, as expected.

	\eqref{lem:rankeq:i:dx}~$\Rightarrow$~\eqref{lem:rankeq:i:h}:
	Recall the definition of $H(\,\cdot\,)$ in~\eqref{eq:hx22}:
	%repeated in the following as
	\[
		H(X) = H_{\phi}(X) + \tr(X^{\T}D)\cdot H_{\psi}(X) +\psi(X)\cdot (DX^{\T}+XD^{\T}),
	\]
	where $H_{\phi}$, $H_{\psi}$, and $\psi(X) >0$ are all unitarily
	invariant.
	For $X_a=XQ_a$ and $X_b=XQ_b$,	
    \begin{subequations}\label{eq:hxab}
    	\begin{align}
    		H(X_a) &= H_{\phi}(X) + \tr(X_a^{\T}D)\cdot H_{\psi}(X) +\psi(X)\cdot
    		(DX_a^{\T}+X_aD^{\T}),\label{eq:hxab-a}\\
    		H(X_b) &= H_{\phi}(X) + \tr(X_b^{\T}D)\cdot H_{\psi}(X) +\psi(X)\cdot
    		(DX_b^{\T}+X_bD^{\T}). \label{eq:hxab-b}
    	\end{align}
    \end{subequations}
	By the assumption of item (c): $DX_a^{\T}=DX_b^{\T}$, we get $X_aD^{\T}=X_bD^{\T}$ by taking transpose, and
    $$
    \tr(X_a^{\T}D) = \tr(DX_a^{\T})=\tr(DX_b^{\T})=\tr(X_b^{\T}D),
    $$
    immediately
    implying $H(X_a) = H(X_b)$ by \eqref{eq:hxab}.
%	and
%	$\tr(X_a^{\T}D) = \tr(X_b^{\T}D)$
%	(by~\Cref{cor:al(X)-characterization}\eqref{cor:al:ie})
%	imply $H(X_a) = H(X_b)$.
	%together with $\tr(X_a^{\T}D) = \tr(X_b^{\T}D)=\|X^{\T}D\|_{\tr}$,
	%implies $H(X_a) = H(X_b)$.

	\eqref{lem:rankeq:i:h}~$\Rightarrow$~\eqref{lem:rankeq:i:hx}:
	This holds trivially.

	\eqref{lem:rankeq:i:hx}~$\Rightarrow$~\eqref{lem:rankeq:i:dx}:
	It follows from~\eqref{eq:hxab} and
	\Cref{cor:al(X)-characterization}\eqref{cor:al:ie} that
	\[
	 0=	H(X_a) X - H(X_b)X = \psi(X)\cdot (DX_a^{\T}X- DX_b^{\T}X).
	\]
	Since $\psi(X) >0$, we conclude $DX_a^{\T}X= DX_b^{\T}X$ and thus
$D(X_a^{\T}X) X^{\T} =  D(X_b^{\T}X) X^{\T}$.
%	\[
%		DX_a^{\T}X =  DX_b^{\T}X
%		\quad\Rightarrow\quad
%		D(X_a^{\T}X) X^{\T} =  D(X_b^{\T}X) X^{\T}.
%		\quad\Rightarrow\quad
%		DX_a^{\T} =  DX_b^{\T},
%	\]
Plug in $X_a=XQ_a$ and $X_b=XQ_b$  to get $DX_a^{\T} =  DX_b^{\T}$.
%	where the last equation is due to $X_a=XQ_a$ and $X_b=XQ_b$.
	The proof is completed.
\end{proof}

\subsection{Proof of \Cref{thm:global}}
Now we are ready to prove our main \Cref{thm:global} of the section.

\begin{proof}[Proof of \Cref{thm:global}]
	Let $\widetilde X\in\al{X_*}$.
	It follows from~\eqref{eq:align0} that
	\begin{equation}\label{eq:ftxfxs}
		f(X_*) \leq f(\widetilde X) \leq \max_{X\in\bbO^{n\times k}} f(X)
		= f(X_*),
	\end{equation}
	where the equality  is by the global maximality of $f$ at $X_*$.
	So all equalities in \eqref{eq:ftxfxs} must hold, and thus we have
	\begin{equation}\label{eq:alopt}
		X_*\in \al{X_*}
		\qquad\text{and}\qquad
		\widetilde X \in \argmax_{X\in\bbO^{n\times k}} f(X),
	\end{equation}
	recalling the definition in~\eqref{eq:align0}.
	By $X_*\in \al{X_*}$
	and~\Cref{cor:al(X)-characterization}\eqref{cor:al:ic},
	we have $X_*^{\T}D\succeq 0$.

	It remains to show $\mrank(X_*^{\T}D) = \mrank(D)$.
	Equivalently, we will prove the statement
	in~\Cref{thm:rank-consistency}\eqref{lem:rankeq:i:hx} for $X=X_*$.
	We first recall~\eqref{eq:alopt}
	that any $\widetilde X\in\al{X_*}$
	is a global maximizer of~\eqref{eq:gopt}
	and, therefore, an eigenbasis matrix of  NEPv~\eqref{eq:nepv2}:
	\[
		H(\widetilde X) \widetilde X = \widetilde X \widetilde \Lambda
		\qquad\text{with}\qquad
		\widetilde \Lambda = \widetilde X^{\T} H(\widetilde X) \widetilde X.
	\]
	Since $\widetilde X = X_*Q$ for some $Q\in\bbO^{k\times k}$, we have
	\[
		H(\widetilde X) X_* =  X_*Q \widetilde \Lambda Q^{\T} = X_* \cdot
		(X_*^{\T} H(\widetilde X) X_*).
	\]
	Hence, for any two $X_a,X_b\in\al{X_*}$,
	\[
		H(X_a) X_* = X_*\cdot\left( X_*^{\T} H(X_a) X_*\right)
		           = X_*\cdot\left( X_*^{\T} H(X_b) X_*\right)
				   = H(X_b) X_*,
	\]
	where for obtaining the second equality, we have used $X_*^{\T} H(X_a) X_* = X_*^{\T} H(X_b) X_*$,
	which can be verified straightforwardly with the help of
	~\eqref{eq:hxab} and
	\Cref{cor:al(X)-characterization}\eqref{cor:al:ie} with $X=X_*$.
\end{proof}

% relating NEPv
By now it is clear that, as far as the global maximality of
optimization problem~\eqref{eq:gopt} is concerned,
we should at least seek a solution $X$ to NEPv~\eqref{eq:nepv2} such
that both definiteness condition \eqref{cond:ND} and
 \rankd~condition \eqref{cond:rank}
%the following two conditions
%\begin{equation}\label{eq:constraint}
%X^{\T}D \succeq 0
%\quad\text{and}\quad
%\mrank(X^{\T}D) = \mrank(D)
%\end{equation}
are satisfied
(although such $X$ is not guaranteed to be a global maximizer).
Our developments in the rest of this paper focus on finding such a solution.
Notice that, although the two conditions are consequences of
studying optimization problem~\eqref{eq:gopt},
they are really about the two matrices $X\in\bbO^{n\times k}$
and $D\in\bbR^{n\times k}$ only.
Given the importance of
both conditions~\eqref{cond:ND} and~\eqref{cond:rank} on $X$,
we hence introduce the notion of {\em $D$-regularity} for the
solution of  NEPv~\eqref{eq:nepv2}.  %{\em $D$-properness}.

%% concepts in terms of NEPv
\begin{definition}\label{dfn:D-proper}
{\rm
An eigenbasis matrix $X\in\bbO^{n\times k}$ of  NEPv~\eqref{eq:nepv2}
is called {\em \dproper\/} if it satisfies
both  definiteness condition \eqref{cond:ND} and
 \rankd~condition \eqref{cond:rank}.
}
\end{definition}

% equivalent solution
In the case of rank-deficient $D$, a \dproper~eigenbasis matrix
$X$ of  NEPv~\eqref{eq:nepv2} will have infinitely many
equivalent solutions given by $\widetilde X = XQ$ with $Q$
of  form~\eqref{eq:qform}, i.e., $\widetilde X \in\al{X}$
is an arbitrary alignment:
Clearly,  $\widetilde X$ so defined satisfies both conditions
\eqref{cond:ND} and \eqref{cond:rank},
and we also have $H(X) = H(\widetilde X)$,
due to the \rankd~condition and~\Cref{thm:rank-consistency}\eqref{lem:rankeq:i:h},
and so $H(\widetilde X) \widetilde X = \widetilde X \widetilde \Lambda$
with $\widetilde \Lambda = Q^{\T}\Lambda Q$.
%So, %if $X$ is a \dproper~eigenbasis matrix of the NEPv~\eqref{eq:nepv2},
Since any alignment $\widetilde X\in\al{X}$ is still
a \dproper~eigenbasis matrix, we view them as equivalent solutions.

\section{SCF Iteration}\label{sec:scf}
The self-consistent field (SCF) iteration is one of the most general and widely used
methods for solving NEPv~\cite{Cai:2018}. The term SCF comes from
the community of  computational physics and
chemistry in solving the Kohn-Sham equation in the density functional
theory~\cite{Martin:2004,Szabo:2012},
and it has since been widely adopted to refer to
similar ideas for solving any general NEPv~\eqref{eq:nepv}.
In this section, we propose an SCF-type iteration for
NEPv~\eqref{eq:nepv2},
with the purpose of solving  optimization problem \eqref{eq:gopt} over the Stiefel manifold
$\bbO^{n\times k}$ in mind. The SCF-type iteration is essentially due to
\cite{wazl:2022,zhwb:2020,Zhang:2020}, where special cases of \eqref{eq:gopt} were
considered.

% goals
From the perspective of the underlying optimization problem~\eqref{eq:gopt},
at each step,
our SCF-type iteration seeks a new approximate solution of NEPv~\eqref{eq:nepv2}
that fulfills  definiteness condition~\eqref{cond:ND}
and meanwhile increases the objective value.
It is rather easy to make $X_i^{\T}D\succeq 0$ for each iterate $X_i$ by alignment,
but monotonically increasing the objective value is highly nontrivial.
Inspired by the special
cases in \cite{wazl:2022,zhwb:2020,Zhang:2020}, where using the eigenspace of $H(X_i)$
associated with its $k$ largest eigenvalues as $X_{i+1}$ can always increase the objective value,
and also for the sake of explaining our convergence analysis,
we will state our SCF-type iteration
using the top eigenspace.
But we also caution that such an approach may not
always increase the objective value in general;
when that happens, the so-called {\em level-shifting\/} can often be applied,
and we will return to this in~\Cref{sec:ls}.

% algorithm description
Our SCF-type iteration for NEPv~\eqref{eq:nepv2}
can be described as follows:
Starting from an initial guess $X_0\in\bbO^{n\times k}$,
generate sequentially $X_1,X_2,\dots$, all in $\bbO^{n\times k}$, by
solving the symmetric eigenvalue problems
\begin{equation}\label{eq:scf000}
	%H( X_i)\widetilde  X_{i+1} =\widetilde  X_{i+1}  \Lambda_{i+1}
	H( X_i)\wtd X_{i+1} =\wtd X_{i+1} \wtd \Lambda_{i+1}
	\qquad \text{for $i=0,1,2,\dots$,}
\end{equation}
where $\wtd X_{i+1}\in\bbO^{n\times k}$ is an orthonormal eigenbasis matrix of
$H( X_i)$ associated with its $k$ largest eigenvalues.
Here, the eigenbasis matrix $\wtd X_{i+1}$ may not automatically
satisfy  definiteness condition~\eqref{cond:ND}.
Since we are interested in \dproper~eigenbasis matrices
satisfying~\eqref{cond:ND},
we can align $\wtd X_{i+1}$ after solving \eqref{eq:scf000} to
\begin{equation}\label{eq:normalization}
 X_{i+1} \in \al{\widetilde X_{i+1}},
%	X_{i+1}\ \leftarrow\ \al{\wtd X_{i+1}},
\end{equation}
namely, $X_{i+1}$ is taken from $\al{\wtd X_{i+1}}$
(and hence $X_{i+1}^{\T}D\succeq 0$),
which can be obtained by the SVD (or polar decomposition) of
$\wtd X_{i+1}^{\T}D$, recalling~\eqref{eq:qform}.
We therefore arrive at an SCF-type iteration summarized in~\Cref{alg:scf},
which is essentially from~\cite{wazl:2022,zhwb:2020,Zhang:2020} as mentioned.

% main algorithm
\begin{algorithm}[t]
\caption{An SCF-type iteration for NEPv~\eqref{eq:nepv2}}
\label{alg:scf}

\begin{algorithmic}[1]
	\REQUIRE
	$X_0 \in \bbO^{n\times k}$ such that $X_0^{\T}D\succeq 0$;
	\ENSURE
	a solution to NEPv~\eqref{eq:nepv2} for the purpose of solving
	optimization problem~\eqref{eq:gopt}.
	\FOR{$i=0,1,2,\ldots$ until convergence}
		\STATE\label{alg:scf:scf}
			  solve the symmetric eigenproblem
			  $H( X_i)\widetilde  X_{i+1} = \widetilde X_{i+1} \widetilde \Lambda_{i+1}$,
			  where $\widetilde X_{i+1}\in\bbO^{n\times k}$ contains
			  the eigenvectors for the $k$ largest eigenvalues of $H(X_i)$;
		\STATE\label{alg:scf:normalize}
			 align $\widetilde X_{i+1}$ to get $X_{i+1} \in \al{\widetilde X_{i+1}}$;
	\ENDFOR
	\RETURN
	the last $X_i$ as a solution to NEPv~\eqref{eq:nepv2}.
\end{algorithmic}
\end{algorithm}

% comments on implementation
In~\Cref{alg:scf},
we assume that the initial basis matrix satisfies $X_0^{\T}D\succeq 0$;
otherwise, it can be aligned, i.e.,
%$X_0\leftarrow \al{X_0}$
$X_0\in \al{X_0}$
as in~\eqref{eq:normalization}.
We have also left out the stopping criterion for the SCF-loop.
As inspired by the common practice for
linear eigenvalue problems~\cite{abbd:1999,bddrv:2000,demm:1997},
we  introduce the normalized residual norm
to gauge the accuracy of an approximate solution $X$ of NEPv~\eqref{eq:nepv2}:
\begin{equation}\label{eq:NRes}
\mbox{NRes}(X) :=\frac{\| H(X)X-X [X^{\T}H(X) X]\|}{\|H(X)\|},
\end{equation}
where $\|\cdot\|$ is some matrix norm that is convenient to evaluate,
e.g.,
the matrix 1-norm (the maximum column sum in absolute value)
or Frobenius norm (the square root of the sum of squares of all elements).
Given a tolerance $\mbox{\tt tol}$,
a small number, the SCF-loop is considered converged if  $\mbox{NRes}(X_i)\le\mbox{\tt tol}$.

% convergence of SCF
For a unitarily invariant NEPv, alignment \eqref{eq:normalization} becomes
unnecessary.
For such NEPv, the local convergence behavior of
the plain SCF, i.e.,~\eqref{eq:scf000} with $X_{i+1}=\wtd X_{i+1}$,
has been extensively studied,
and various improvements for SCF have been developed;
see, e.g.,~\cite{Bai:2022,Cai:2018,Cances:2021,Liu:2014,Stanton:1981,Upadhyaya:2021,Yang:2009}.
In particular, the recent work~\cite{Bai:2022}
established a sharp estimation for the local convergence rate of SCF,
and it also provided a theoretical guarantee for the effectiveness of a
level-shifting technique to fix the potential divergence issue of SCF.

% convergence of the variant
For~\Cref{alg:scf},
due to the non-unitary invariance of $H(X)$ and the
extra alignment step in line~\ref{alg:scf:normalize},
most existing analysis and techniques for the plain SCF do not apply directly.
Although several recent works~\cite{wazl:2022,zhwb:2020,Zhang:2020} proved the convergence
of such algorithms for their respective special cases, their analyses
are specialized and do not apply to general NEPv~\eqref{eq:nepv2},
nor do they produce sharp estimation of  convergence rate.
New techniques are needed to better understand the
convergence (or divergence) of~\Cref{alg:scf}.

%\marginpar{\tiny A bit repetitive of the Intro.}

%
\section{Local Convergence Analysis}\label{sec:local}
We will start by showing that, locally near a \dproper~solution of NEPv \eqref{eq:nepv2},
the SCF-type iteration in \Cref{alg:scf} is the plain SCF
for an equivalent NEPv that is unitarily invariant.
By this connection, we can conveniently perform its local
convergence analysis along the line of~\cite{Bai:2022}.
In particular, a closed-form local convergence rate of \Cref{alg:scf} will be
established, and a theoretical foundation for a level-shifted variant of
SCF will be built.

\subsection{Aligned NEPv} \label{sec:aligned}
To begin with,
denote by ${\bbod}$ the set of $X\in\bbO^{n\times k}$ satisfying  \rankd~condition
\eqref{cond:rank}:
\begin{equation}\label{eq:DOM4aligned}
{\bbod}:= \{X\in\bbO^{n\times k}\,:\, \mrank(X^{\T}D) = \mrank(D)\}.
\end{equation}
%Necessarily,~\dproper~eigenbasis matrices are elements of ${\bbod}$.
Given $X\in\bbod$, the matrix $H(\wtd X)$ will not change
as $\wtd X$ varies in $\al{X}$, due to~\Cref{thm:rank-consistency}\eqref{lem:rankeq:i:h},
and so $H(\al{X}):=H(\wtd X)$ for $\wtd X\in\al{X}$ is well-defined.
We hence introduce
\begin{equation}\label{eq:alignedH(x)}
G(X):=H(\al{X})
\quad\mbox{for $X\in\bbod$},
%\quad\mbox{for $X\in\bbO^{n\times k}$ satisfying the \rankd~condition},
\end{equation}
and call $G$ the {\em aligned\/} function of $H$ in~\eqref{eq:hx22}.
Accordingly, we introduce the \emph{aligned NEPv\/} of NEPv~\eqref{eq:nepv2}
as
\begin{equation}\label{eq:alnepv}
	G(X)X = X\Lambda  \quad\mbox{for $X\in\bbod$}.
\end{equation}
By~\Cref{cor:al(X)-characterization}\eqref{cor:al:id}, the aligned
function $G$ in~\eqref{eq:alignedH(x)}
is unitarily invariant, i.e.,
\begin{equation}\label{eq:univarG}
	\text{$G(XQ) = G(X)$\quad for\quad $X\in{\bbod}$ and $Q\in\bbO^{k\times k}$,}
\end{equation}
that is, the aligned NEPv~\eqref{eq:alnepv} is an unitarily invariant NEPv.

% region
We emphasize that both the aligned matrix-valued function
$G$ in~\eqref{eq:alignedH(x)} and the aligned NEPv \eqref{eq:alnepv} have
their domain of existence ${\bbod}$.
Luckily, the domain ${\bbod}$ is a relative open set on the Stiefel manifold,
because  \rankd~condition~\eqref{cond:rank} always hold under
small perturbations, as shown in~\Cref{lm:rank-preservation} below.

\iffalse
Recall that $\al{X}$ is a set containing all possible alignments $\widetilde X=XQ$ of
$X\in\bbO^{n\times k}$ for some $Q\in\bbO^{k\times k}$ such that $\widetilde
X^{\T}D\succeq 0$, or, equivalently, $Q$ taking the form of \eqref{eq:qform}.
So the coefficient matrix $G(X)$ from~\eqref{eq:alnepv} should be understood as
\begin{equation}\label{eq:evalgx}
	G(X) = H(\widetilde X)
	\quad \text{for some} \quad
	\widetilde X \in \al{X}.
\end{equation}
At the first glance, this $G(X)$ may potentially not be well-defined since $H(\widetilde X)$ may vary as $\widetilde X$ varies in $\al{X}$.
Such ambiguity, however, does not occur if $X$ also satisfies the rank-consistency
condition \eqref{cond:rank}.
%\begin{equation}\label{eq:rankd}
%	\mrank(X^{\T}D) = \mrank(D).
%\end{equation}
According to~\Cref{thm:rank-consistency}\eqref{lem:rankeq:i:rk} and~\eqref{lem:rankeq:i:h},
the rank-consistency
condition \eqref{cond:rank} is sufficient and necessary for
$H(\widetilde X)$ to be the same for all $\widetilde X \in \al{X}$,
independent of the choice of any particular alignment
$\widetilde X\in\al{X}$.
\fi

%
\begin{lemma}\label{lm:rank-preservation}
Let $Y\in{\bbod}$. % satisfy $\mrank(Y^{\T}D) = \mrank(D)$. Then $\mrank(X^{\T}D) = \mrank(D)$
For $X\in\bbO^{n\times k}$, if $\cR(X)$ is
sufficiently close to $\cR(Y)$,
then $X\in{\bbod}$.
\end{lemma}

\begin{proof}
The sufficient closeness of $\cR(X)$  to $\cR(Y)$ implies that there
exists $Q\in\bbO^{k\times k}$
such that %$\|XQ-Y\|_{\F}$ is sufficiently tiny.
$\|XQ-Y\|$ is %sufficiently tiny.
sufficiently small.
The condition $\mrank(Y^{\T}D) = \mrank(D)$ implies
that $Y^{\T}D$ has exactly $\ell=\mrank(D)$ nonzero singular values.
Notice
$$
Q^{\T} X^{\T}D=Y^{\T}D+(XQ-Y)^{\T}D,
$$
where the second term on the right-hand side is in the order of
$\|XQ-Y\|$,
which can be made as small as needed by letting $\cR(X)$ be sufficiently close to $\cR(Y)$.
Because the singular values of a matrix are continuous with respect to matrix entries
\cite{stsu:1990,li:2014HLA}, $Q^{\T}X^{\T}D$ has at least $\ell$ nonzero singular values,
implying $\mrank(Q^{\T} X^{\T}D)\ge \ell$. On the other hand,
$\mrank(Q^{\T}X^{\T}D) = \mrank(X^{\T}D)\le\mrank(D)=\ell$.
Hence, we conclude $\mrank(X^{\T}D)=\mrank(D)$, as was to be shown.
\end{proof}

%% CPD
Let $X\in\bbod$. The {\em canonical polar decomposition\/} of
$X^{\T}D$~\cite[Chapter~8]{Higham:2008} refers to
\begin{subequations}\label{eq:polarc0}
\begin{equation}\label{eq:polarc0-1}
	X^{\T}D = Q_{\bf o} M,
\end{equation}
where, in terms of SVD~\eqref{eq:svdxtd},
\begin{equation}\label{eq:polarqm}
	Q_{\bf o} = U_1V_1^{\T}
	\qquad\text{and}\qquad
	M =  V_1\Sigma_1 V_1^{\T}.
\end{equation}
\end{subequations}
By definition,
$X^{\T}D$ is factorized as the product of a partial isometry
$Q_{\bf o}$ (i.e., $\|Q_{\bf o} x\| = \|x\|$ for all $x\in\cR(Q_{\bf o}^{\T})$ where
$\|\cdot\|$ is the Euclidian vector norm)
and a symmetric positive semi-definite $M$.
Note that  $Q_{\bf o}\in\bbR^{k\times k}$ is not necessarily orthogonal, as in
contrast to an orthogonal factor of  polar decomposition~\eqref{eq:polar}.
The canonical polar factors $Q_{\bf o}$ and $M$ in~\eqref{eq:polarqm}
are uniquely defined (i.e., independent of any freedom in the SVD; see,
e.g., \cite[p.220]{begr:2003}, \cite[Chapter~8]{Higham:2008},
and \cite{li:1993b,li:1995,li:2014HLA,lisu:2002}).

Over  domain ${\bbod}$,  $G(X)$ also admits
an expression in terms of the canonical polar factors of $X^{\T}D$, with which we can
conveniently show the differentiability of $G(X)$ and obtain its
derivative.

\begin{lemma}\label{lem:welldef}
	Let $X\in{\bbod}$ and $G(X)$ be defined by~\eqref{eq:alnepv}, and let
	$X^{\T}D$ have the canonical polar decomposition in~\eqref{eq:polarc0}.
	Then we have %$G(X)$ is well-defined and given by		
	\begin{subequations}\label{eq:g2x}
	\begin{equation}\label{eq:g2x-1}
				G(X) = H_{\phi}(X) + G_\psi(X),
			\end{equation}
			where
			\begin{equation}\label{eq:gpx}
				G_\psi(X):= \tr(M)\cdot H_{\psi}(X) +\psi(X)\cdot
				\left(DQ_{\bf o}^{\T}X^{\T} + XQ_{\bf o}D^{\T}\right),
			\end{equation}
	\end{subequations}
	and $H_{\phi}(X)$ and $H_{\psi}(X)$ are from~\eqref{eq:der}.

\end{lemma}

\begin{proof}
	For any $\widetilde X \in\al{X}$, it follows from the expression~\eqref{eq:hxab} of $H$ that
	\begin{equation}\label{eq:hxq}
		H(\widetilde X) = H_{\phi}(X) + \tr(\widetilde X ^{\T}D)\cdot H_{\psi}(X) +\psi(X)\cdot
		(D\widetilde X^{\T}+\widetilde XD^{\T}).
	\end{equation}
	Recall SVD~\eqref{eq:svdxtd}.
	By $\mrank(X^{\T}D) = \mrank(D)$
	and~\Cref{thm:rank-consistency}\eqref{lem:rankeq:i:null}, we have $DV_2=0$,
	which implies
	\begin{equation}\label{eq:dxt}
		D \widetilde X^{\T} = D \left( V_1 U_1^{\T} + V_2\Omega^{\T} U_2^{\T} \right)X^{\T}
			= D\cdot(V_1 U_1^{\T})\cdot X^{\T}
			= D \cdot Q_{\bf o}^{\T} \cdot  X^{\T},
	\end{equation}
	where the first equality is by the expression of $\al{X}$ in~\eqref{eq:al(X)-characterization},
	and the last equality is because of the expressions in \eqref{eq:polarqm}.
	Also, $\widetilde X^{\T}D\equiv M$, independent of which $\widetilde X\in\al{X}$.
	So~\eqref{eq:hxq} leads directly to $H(\widetilde X) = H_{\phi}(X) + G_\psi(X)$,
	which does not change as $\widetilde X$ varies in $\al{X}$.
	This proves~\eqref{eq:g2x}.
\end{proof}

% equivalence
The next theorem makes it precise in what sense the original
NEPv~\eqref{eq:nepv2} and its aligned NEPv~\eqref{eq:alnepv} are
equivalent.

\begin{theorem}\label{thm:NEPv=alignedNEPv}
%Let $X\in\bbO^{n\times k}$.
Let $X\in\bbO^{n\times k}$ satisfy  \rankd~condition \eqref{cond:rank},
i.e., $X\in\bbod$.
\begin{enumerate}[{\rm (a)}]
	\item\label{lem:nepv:a}
		%If $X\in\bbO^{n\times k}$
		If $X$
		is a \dproper~solution to NEPv~\eqref{eq:nepv2},
		then $X$ is a solution to the aligned NEPv~\eqref{eq:alnepv}.
		%with a same set of eigenvalue.

	\item\label{lem:nepv:b}
		%If $X\in\bbO^{n\times k}$
		If $X$ is a solution to the
		aligned NEPv~\eqref{eq:alnepv},
		then any alignment $\widetilde X\in \al{X}$
		is a \dproper~solution to NEPv~\eqref{eq:nepv2}.
		%with a same set of eigenvalue.
\end{enumerate}

% ---  A SIMPLIFIED VERSION:
%$X\in\bbO^{n\times k}$ is a \Blue{proper eigenvector} of the aligned NEPv~\eqref{eq:alnepv}
%if and only if
%$\al{X}\in\bbO^{n\times k}$ is a \Blue{proper eigenvector} of NEPv~\eqref{eq:nepv}.

\end{theorem}

\begin{proof}
Consider item~\eqref{lem:nepv:a}. That $X$ is \dproper\ implies
$X\in\bbod$ and $X\in\al{X}$ and hence $G(X) = H(\al{X}) = H(X)$. Together
with $H(X) X = X\Lambda$, we conclude $G(X)  X = X \Lambda$,
proving item~\eqref{lem:nepv:a}.

Now turn to item~\eqref{lem:nepv:b}.
%Suppose that $X$ %$X\in\bbO^{n\times k}$
%is a solution to the aligned NEPv~\eqref{eq:alnepv},
%which implies $X\in\bbod$
%(as the problem is only defined on ${\bbod}$).
By the definition of $G$ in~\eqref{eq:alignedH(x)},
$G(X) = H(\al{X}) = H(\widetilde X)$ for any $\widetilde X\in \al{X}$.
Recalling~\eqref{eq:al(X)-characterization} that $\wtd X=XQ$
for some $Q\in\bbO^{k\times k}$, we have
$$
G(X) X = X\Lambda
\qquad\Rightarrow\qquad
H(\widetilde X)\cdot  \widetilde X = \widetilde X\cdot \widetilde \Lambda,
$$
where $\widetilde \Lambda = Q^{\T} \Lambda Q$.
Namely, $\widetilde X\in\bbO^{n\times k}$ is a solution to  NEPv~\eqref{eq:nepv2}.
%The definition of $\al{X}$ implies
Moreover, \Cref{cor:al(X)-characterization}\eqref{cor:al:ib} implies
$\wtd X^{\T}D\succeq 0$, and
so $\wtd X$ is also \dproper.
\end{proof}

%% comments
The aligned NEPv~\eqref{eq:alnepv} is defined
through~\eqref{eq:alignedH(x)} and  NEPv~\eqref{eq:nepv2}.
Alternatively, this new NEPv can be directly derived from the first-order
optimality condition of a different maximization problem
equivalent to~\eqref{eq:gopt}; see our discussions in~\Cref{sec:bimax}.
This alternative interpretation will become useful when it comes to justify
the sufficient condition for the convergence of
level-shifted SCF later in~\Cref{sec:ls}.
%\marginpar{\tiny Does \Cref{sec:ls} use \Cref{sec:bimax}?}
% Proof in Appendix C follows from Appendix B

\subsection{Differentiability for Aligned NEPv}
In this subsection, we will show that  $G(X)$ for the aligned
NEPv~\eqref{eq:alnepv}
is differentiable in $X\in\bbR^{n\times k}$,
at  $X$ in domain ${\bbod}$ defined in \eqref{eq:DOM4aligned},
and establish closed-form expressions to various derivatives for the purpose
of analyzing the local convergence of the plain SCF on the aligned NEPv.
This will in turn allow us to derive a sharp  estimate for the rate of convergence of
our SCF-type iteration in \Cref{alg:scf} for  NEPv~\eqref{eq:nepv2}.
By the expression of $G(X)$ in~\eqref{eq:g2x}, it is sufficient to
consider the differentiability of the canonical polar factors $Q_{\bf o}$ and
$M$ of $X^{\T}D$ with respect to $X\in\bbR^{n\times k}$.

We begin with an alternative formula for the
canonical polar factors of $X^{\T}D$ for $X\in\bbod$.
Let $\rkd:=\mrank(D)$ and factorize  $D$ as
\begin{equation}\label{eq:dfact}
D = D_1 P^{\T},
\end{equation}
where $D_1\in\bbR^{n\times \rkd}$ has full column rank
and  $P\in\bbO^{k\times \rkd}$.
For $X\in\bbod$, we have
\begin{equation}\label{eq:rank}
	\mrank(X^{\T}D_1) = \mrank(X^{\T}D) = \mrank(D) \equiv \rkd,
\end{equation}
and, hence, $X^{\T}D_1\in\bbR^{k\times \rkd}$ has full column rank.

\begin{lemma} \label{lem:polar}
	Let $D$ be factorized as in~\eqref{eq:dfact}, $X\in{\bbod}$, and the
polar decomposition of $X^{\T}D_1\in\bbR^{k\times \rkd}$ be
	\begin{equation}\label{eq:polarc}
		X^{\T}D_1=  Q_1 M_1.
	\end{equation}
	\begin{enumerate}[{\rm (a)}]
		\item \label{i:lemma:polar:a}
			The canonical polar decomposition of $X^{\T}D$ is given by
			$X^{\T}D=Q_{\bf o} M$, where
			\begin{equation}\label{eq:qmxa}
				Q_{\bf o} = Q_1P^{\T}
				\quad\text{and}\quad
				M = PM_1P^{\T}.
			\end{equation}
		\item \label{i:lemma:polar:b}
			Both $Q_{\bf o}$ and $M$ are Fr\'echet differentiable with respect to $X\in\bbR^{n\times k}$, and
            their Fr\'echet derivatives along direction $E\in\bbR^{n\times k}$ are given by
            \begin{equation}\label{eq:dmdq}
				\begin{aligned}
					\D\!M(X)[E] = P L  P^{\T},
					\quad %\text{and}\quad
					\D Q_{\bf o}(X)[E] = \left(\,E^{\T}DP - Q_{\bf o} PL\,\right) M_1^{-1}P^{\T},
				\end{aligned}
			\end{equation}
			respectively, where
			$L$ is the solution to the Lyapunov equation
			\begin{equation}\label{eq:leqs}
			M_1 L + LM_1 =  D_1^{\T}\,(\,XE^{\T}+EX^{\T}\,)\,D_1.
			\end{equation}

		\item \label{i:lemma:polar:c}
			It holds that
			\begin{equation}\label{eq:trdmxe}
				\tr\left(\,\D\!M(X)[E]\,\right) = \tr \left(Q_{\bf o} D^{\T} E\right).
			\end{equation}

	\end{enumerate}
\end{lemma}

\begin{proof}
For item~\eqref{i:lemma:polar:a},
observe that $Q_{\bf o}$ in~\eqref{eq:qmxa} is a partial isometry
(i.e., $\|Q_{\bf o} x\| = \|x\|$ for all $x\in\cR(Q_{\bf o}^{\T})$)
and $M\succeq 0$.
By~\eqref{eq:dfact}, \eqref{eq:polarc}, and \eqref{eq:qmxa}, we get
\[
Q_{\bf o} M = Q_1M_1P^{\T} = X^{\T}D_1P^{\T}  = X^{\T}D,
\]
yielding the canonical polar decomposition of $X^{\T}D$, since the canonical polar decomposition
is always unique (see, e.g., \cite[p.220]{begr:2003} and \cite[Chapter~8]{Higham:2008}).

For item~\eqref{i:lemma:polar:b}, observe that $X^{\T}D_1$ has full column rank.
According to~\Cref{lem:polard} in the appendix,
the polar factors $M_1$ and $Q_1$ from~\eqref{eq:polarc}
are differentiable in $X$, and they have derivatives along direction $E\in\bbR^{n\times k}$
given by~\footnote{
Notice the fact that, for $Z=X^{\T}D_1$ and $Y=E^{\T} D_1$,
the Fr\'echet derivative of a differentiable function
$F(Z)$ satisfies $\D F(Z)[Y] = \D \wtd F(X)[E]$
with $\wtd F(X) := F(X^{\T}D_1)$,
as can be verified by straightforward expansions of the functions.
}
\begin{equation}\label{eq:dm1q1}
	\D\!M_1(X)[E] =  L
	\quad\text{and}\quad
	\D Q_1(X)[E] = (E^{\T}D_1 - Q_1 L) \cdot M_1^{-1},
\end{equation}
where $L$ satisfies~\eqref{eq:leqs}.
Consequently, by~\eqref{eq:qmxa}, the canonical polar factors
$M$ and $Q_{\bf o}$ (depending on $M_1$ and $Q_1$) are differentiable in $X$ as well, and
their derivatives along direction $E\in\bbR^{n\times k}$ are
\[
	\D\!M(X)[E] = P\cdot \D\!M_1(X)[E] \cdot P^{\T}
	\quad\text{and}\quad
	\D Q_{\bf o}(X)[E] = \D Q_1(X)[E] \cdot P^{\T}.
\]
Then by~\eqref{eq:dm1q1} we get~\eqref{eq:dmdq}.

For item~\eqref{i:lemma:polar:c},
since $P^{\T}P=I_{\rkd}$, \eqref{eq:dmdq} implies
$\tr\left(\D\!M(X)[E]\right) =\tr\left(L\right)$.
Post-multiplying both sides of \eqref{eq:leqs} by
$M_1^{-1}$ and then taking  trace, we obtain
\begin{equation}\label{eq:trm1l}
	2\,\tr(L)
	=
	\tr\left(D_1^{\T}XE^{\T}D_1M_1^{-1}\right)
	+
	\tr\left(D_1^{\T}EX^{\T}D_1M_1^{-1}\right),
\end{equation}
where we have used $\tr(M_1LM_1^{-1}) %= \tr(LM_1^{-1}M_1)
= \tr(L)$.
The two terms on the right hand side of \eqref{eq:trm1l} are identical
because of $\tr(AM_1^{-1}) = \tr(M_1^{-1}A) = \tr(A^{\T}M_1^{-1})$
for all $A\in\bbR^{\rkd\times \rkd}$, where the last equality is by transposing
the argument and $M_1^{\T}=M_1$.
%where the last equality is obtained by transposing the matrix
%\[
%\tr\left(D_1^{\T}EX^{\T}D_1M_1^{-1}\right)
%=
%\tr\left(M_1^{-1}D_1^{\T}EX^{\T}D_1\right)
%=
%\tr\left(D_1^{\T}XE^{\T}D_1M_1^{-1}\right),
%\]
%where the last equality is obtained by transposing the matrix and $M_1^{\T}=M_1$.
So \eqref{eq:trm1l} leads to
\[
\tr(L)
= \tr\left(D_1^{\T}EX^{\T}D_1M_1^{-1}\right)
= \tr\left(D_1^{\T}EQ_1\right)
= \tr\left(Q_1D_1^{\T}E\right),
\]
which yields \eqref{eq:trdmxe} upon noticing
 $Q_1D_1^{\T}= (Q_1 P^{\T}) (D_1P^{\T})^{\T} = Q_{\bf o} D^{\T}$.
\end{proof}

Suppose that  $H(X)$ in~\eqref{eq:hx22} is differentiable
at  $X\in\bbR^{n\times k}$ that has orthonormal columns
(by~\Cref{lem:dfx}, this always holds if $\phi$ and $\psi$ in~\eqref{eq:gopt} are twice
differentiable functions).
Then the differentiability of the canonical polar factors, as
in~\Cref{lem:polar}, implies that
$G(X)$ in~\eqref{eq:g2x} is differentiable at any $X\in{\bbod}$,
%\sout{as an element of $\bbR^{n\times k}$},
i.e., satisfying \rankd~condition~\eqref{cond:rank}.
By a straightforward application of the chain rule of differentiation,
we derive from~\eqref{eq:g2x} the
following formula for assembling and calculating the derivative ${\bf D} G(X)[E]$.

\begin{corollary}\label{cor:dervative}
Let $X\in\bbod$ and assume $\phi$ and $\psi$ in~\eqref{eq:gopt} are twice differentiable
at $X$. % \sout{as an element of $\bbR^{n\times k}$}.
Then,
\begin{multline}\label{eq:dgx}
	{\bf D} G(X)[E] =
	{\bf D}H_{\phi}(X)[E]
	+ \tr({\bf D}M(X)[E])\cdot H_{\psi}(X) \\
	+ \tr(M)\cdot {\bf D}H_{\psi}(X)[E]
	+ {\bf D}\psi(X)[E] \cdot \left(DQ_{\bf o}^{\T}X^{\T} + X Q_{\bf o} D^{\T}\right)\\
	+ \psi(X) \cdot \left(D\cdot {\bf D}Q_{\bf o}(X)[E]^{\T}\cdot X^{\T} + X\cdot
	{\bf D}Q_{\bf o}(X)[E]\cdot D^{\T}\right) \\
	+ \psi(X) \cdot \left(DQ_{\bf o}^{\T} E^{\T} + E Q_{\bf o} D^{\T}\right),
\end{multline}
where
${\bf D}M$ and ${\bf D}Q_{\bf o}$ are given by~\eqref{eq:dmdq},
and $H_{\phi}$ and $H_{\psi}$
are from the expression of $H(X)$ in~\eqref{eq:nepv2}.
\end{corollary}

\subsection{Rate of Convergence}
We now perform a local convergence analysis of~\Cref{alg:scf}. Recall that the creation of
NEPv~\eqref{eq:nepv2} aims to solve  optimization problem \eqref{eq:gopt}
whose global maximizers are provably \dproper\ eigenbasis matrices of the NEPv.
Because of this, and as we are performing a local convergence analysis,
we may assume without loss of generality that the initial guess $X_0$ is
such that $\cR(X_0)$ is sufficiently close to $\cR(X_*)$,
where $X_*\in\bbO^{n\times k}$ is a \dproper~solution of
NEPv~\eqref{eq:nepv2},
i.e., satisfying both conditions \eqref{cond:ND} and \eqref{cond:rank}:
\begin{equation}\label{eq:neccX*}
X_*^{\T}D \succeq 0
\quad\text{and}\quad
\mrank(X_*^{\T}D) = \mrank(D)=:\rkd.
\end{equation}
The key observation of our analysis is that,
locally around $\cR(X_*)$,
we can identify the SCF-type iteration in~\Cref{alg:scf}
as the plain SCF for the aligned NEPv~\eqref{eq:alnepv}:
each iterative step of the SCF-type iteration in~\Cref{alg:scf} satisfy
\begin{equation}\label{eq:scf00}
	G( X_i)  X_{i+1} =  X_{i+1} \Lambda_{i+1},
	\qquad \text{for $i=0,1,2,\dots$,}
\end{equation}
where the eigenvalues of $\Lambda_{i+1}\in\bbR^{k\times k}$ are the
$k$ largest eigenvalues of $G( X_i)$, provided $X_i\in\bbod$,
which ensures that $G( X_i)$ in~\eqref{eq:scf00} is well-defined
by~\eqref{eq:alignedH(x)}.
The conditions $X_i\in\bbod$ always hold if
$\cR(X_0)$ is sufficiently close to $\cR(X_*)$  in the case of convergence;
whereas in the case of divergence, it is still reasonable to assume that,
at least for the first few SCF-type iterative steps of~\Cref{alg:scf}
before $\cR(X_i)$ deviates too far from $\cR(X_*)$.

For the plain SCF~\eqref{eq:scf00},
since $G(X)$ is unitarily invariant according to~\eqref{eq:univarG},
we can apply the existing local convergence results for unitarily invariant
NEPv in, e.g.,~\cite{Bai:2022,Cai:2018}.
Our goal in the following is to  establish the local convergence rate
of~\Cref{alg:scf} through the plain SCF~\eqref{eq:scf00}.

First, let $G(X_*)$ have the eigenvalue decomposition
\begin{equation}\label{eq:eigdgx}
G(X_*)=
[X_*, X_{*\bot}]
\begin{bmatrix} \Lambda_* & \\ & \Lambda_{*\bot} \end{bmatrix}
[X_*, X_{*\bot}]^{\T},
\end{equation}
where
$\Lambda_* = \mbox{diag} (\lambda_1,\dots,\lambda_k)$,
$\Lambda_{*\bot} = \mbox{diag} (\lambda_{k+1},\dots,\lambda_n)$, and
$\lambda_1\geq\cdots\geq\lambda_k \ge \lambda_{k+1}\geq \cdots\geq \lambda_n$.
As in \cite{Bai:2022,Cai:2018}, we assume
$$
\lambda_k - \lambda_{k+1} > 0,
$$
%which has to be assumed;
otherwise the eigenspace for the $k$
largest eigenvalues of $G(X_*)$ is not unique~\cite{daka:1970,li:2014HLA,stsu:1990}.
Define
\[
	S(X_*) \in\mathbb R^{(n-k)\times k}
	\qquad\text{by\qquad $[S(X_*)]_{ij} = \left( \lambda_j - \lambda_{k+i}
	\right)^{-1}$},
\]
and linear operator
$\scrL \colon \mathbb R^{(n-k)\times k}\to \bbR^{(n-k)\times k}$ by
\begin{equation}\label{eq:lz}
	\scrL(Z) := S(X_*) \odot \left( X_{*\bot}^{\HH}\ \cdot \mathbf
	DG(X_*)[X_{*\bot}Z]\cdot X_*\right),
\end{equation}
where $\odot$ is the element-wise multiplication, also known as the matrix Hadamard  product,
and $\mathbf DG$ is given by~\eqref{eq:dgx}.

According to the convergence analysis in~\cite{Bai:2022}, the spectral radius $\rho(\scrL)$
of  linear operator $\scrL$ in~\eqref{eq:lz} %i.e., the largest absolute value of its eigenvalues,
is the local convergence rate of the plain SCF~\eqref{eq:scf00}
and, hence, that of  \Cref{alg:scf} as well.
A restatement of~\cite[Theorem~4.2]{Bai:2022} for the plain SCF~\eqref{eq:scf00} yields
the following theorem for the local {\em convergence-in-subspace}
(i.e., the convergence of $\cR(X_i)$ as $i\to\infty$)
of the SCF-type iteration in \Cref{alg:scf}. % on NEPv~\eqref{eq:nepv2}.

\begin{theorem}\label{cor:conv}
	Let $X_*\in\bbO^{n\times k}$ be a \dproper~solution to  NEPv~\eqref{eq:nepv2}
	such that the corresponding $\Lambda_*\equiv X_*^{\T}H(X_*)X_*$ contains the
	$k$ largest eigenvalues of $H(X_*)$.
	Assume that
	\begin{equation}\label{eq:gaphx*}
	\lambda_k(H(X_*)) > \lambda_{k+1}(H(X_*))
	\end{equation}
	and $H(X)$ is differentiable at $X_*$,
	and let $\rho(\scrL)$ be the spectral radius of the linear
	operator $\scrL$ defined by~\eqref{eq:lz}.
	\begin{enumerate}[\rm (a)]
		\item\label{cor:conv:i:sprd}
			If $\rho(\scrL) < 1$, then~\Cref{alg:scf} is locally
			convergent-in-subspace to $X_*$,
			with an asymptotic average convergence rate bounded by
			$\rho(\scrL)$.
		\item
			If $\rho(\scrL) > 1$, then~\Cref{alg:scf} is locally divergent-in-subspace from
			$X_*$.
	\end{enumerate}
\end{theorem}

Convergence-in-subspace, or divergence-in-subspace for that matter,
is measured by
the canonical angles between subspaces. In the case of \Cref{cor:conv}, it is about
\[
\mbox{whether}\quad	\Theta(X_{i},X_*) \to 0\quad\mbox{or}\quad
\Theta(X_{i},X_*) \not\to 0\qquad \text{as $i\to \infty$,}
\]
where $\Theta(X_i,X_*)$ denotes the diagonal matrix of the canonical angles between the subspaces
$\cR(X_i)$ and $\cR(X_*)$ \cite{Bai:2022,Cai:2018,stsu:1990}.
%(recall that
%$\Theta(X,Y) = \mbox{diag}(\arccos(\sigma_1),\dots,\arccos(\sigma_k))$
%with $\sigma_1,\sigma_2,\dots,\sigma_k$ the singular values of $X^{\T}Y$).
So, \Cref{cor:conv}\eqref{cor:conv:i:sprd} can be interpreted as:
Given an initial $X_0$ with sufficiently small % sufficiently tiny
$\Theta(X_0,X_*)$,
for an arbitrarily small $\epsilon>0$,
\begin{equation}\label{eq:convrate}
\|\Theta(X_{i+m},X_*) \|
\ \leq c\, [\rho(\scrL)+\epsilon]^m\cdot \|\Theta(X_i,X_*) \|
\qquad\text{as $i,m\to \infty$},
\end{equation}
where $c$ is some constant and $\|\cdot\|$ is any unitarily invariant matrix norm,
such as the $2$-norm and the Frobenius norm.
The reader is referred to~\cite[Section~4]{Bai:2022} for more discussions on the
local convergence rate of SCF.

%% rank deficient D
We emphasize that the convergence result in~\Cref{cor:conv}
holds even if $D\in\bbR^{n\times k}$ is rank-deficient (i.e., $\mbox{rank}(D)< k$).
In the rank-deficient case,  according to~\eqref{eq:al(X)-characterization},
the alignment function $\al{\wtd X_{i+1}}$ at line~\ref{alg:scf:normalize}~\Cref{alg:scf}
is multi-valued, and so approximation $X_{i+1}$ is not uniquely defined.
However, provided~$X_*$ is a~\dproper~eigenbasis matrix
satisfying the gap condition~\eqref{eq:gaphx*},
the sequence of eigenspaces $\{\cR(X_i)\}$ by~\Cref{alg:scf}
is unique as each $\cR(X_i)$ is sufficiently close to $\cR(X_*)$,
and the rate of local convergence-in-subspace is given by $\rho(\scrL)$.
Later in~\Cref{sec:egs}, we will numerically demonstrate the results of \Cref{cor:conv}.

\section{Level-Shifted SCF}\label{sec:ls}
For NEPv from the special optimization problems \eqref{eq:gopt} studied in
\cite{wazl:2022,zhwb:2020,Zhang:2020},
\Cref{alg:scf} is provably convergent globally from
any initial guess.
However, for NEPv arising from a general optimization problem~\eqref{eq:gopt},
the algorithm may diverge (even with very accurate initial guesses),
as we shall demonstrate by numerical examples in~\Cref{sec:egs}.

When \Cref{alg:scf} diverges, we may apply a level-shifting scheme,
which has been commonly adopted to unitarily
invariant NEPv for fixing
the issue of eigenvalue mispositioning,
by which we mean not all eigenvalues
of $\Lambda$ at optimality  are among the largest ones
and, as a result,
\Cref{alg:scf} is inevitably divergent;
% and divergence issues of
%SCF;
see, e.g.,~\cite{Bai:2018,Saunders:1973,Thogersen:2004,Yang:2007}.
The level-shifting technique simply modifies the matrix-valued function $H(X)$ to $H(X)+\sigma XX^{\T}$,
where $\sigma\in\bbR$
is a preselected level-shift. Note that the addendum
$\sigma XX^{\T}$ is unitarily invariant.
In an obvious way, this scheme
can be conveniently adapted to~\Cref{alg:scf}, by changing
$H(X_i)$ at line 2 to $H(X_i)+\sigma X_iX_i^{\T}$.
%with everything else staying the same.
For ease of reference, we outline the level-shifted variant of \Cref{alg:scf}
in \Cref{alg:LS-scf}, where again we may stop the iteration
if $\mbox{NRes}(X_i)\le\mbox{\tt tol}$, as commented earlier for \Cref{alg:scf}.

\begin{algorithm}[t]
	\caption{A Level-Shifted SCF-type iteration for NEPv~\eqref{eq:nepv2}}
	\label{alg:LS-scf}
\begin{algorithmic}[1]
\REQUIRE $X_0 \in \bbO^{n\times k}$ such that $X_0^{\T}D\succeq 0$, and a level-shift $\sigma$;
\ENSURE  a solution to NEPv~\eqref{eq:nepv2} for the purpose of solving optimization problem~\eqref{eq:gopt}.
	\FOR{$i=0,1,\ldots$ until convergence}
	    \STATE\label{alg:LS-scf:scf}
	          solve the symmetric eigenproblem $[H( X_i)+\sigma X_iX_i^{\T}]\widetilde  X_{i+1} =
	          \widetilde X_{i+1} \widetilde \Lambda_{i+1}$, where $\widetilde X_{i+1}\in\bbO^{n\times k}$ contains
	          the eigenvectors for the $k$ largest eigenvalues of $H( X_i)+\sigma X_iX_i^{\T}$;
	    \STATE\label{alg:LS-scf:normalize}
	         align $\widetilde X_{i+1}$ to get $X_{i+1} \in \al{\widetilde X_{i+1}}$;
	\ENDFOR
\RETURN the last $X_i$ as a solution to NEPv~\eqref{eq:nepv2}.
\end{algorithmic}
\end{algorithm}

Let $X_*$ be a \dproper~solution of  NEPv~\eqref{eq:nepv2}.
Similarly to~\Cref{alg:scf},
we can study the local convergence of the level-shifted SCF in~\Cref{alg:LS-scf}
to $X_*$ through the aligned NEPv~\eqref{eq:alnepv}.
First, let us level-shift the aligned NEPv~\eqref{eq:alnepv} to
\begin{equation}\label{eq:gsigma}
	G_{\sigma}(X) X = X \Lambda_\sigma
	\text{\qquad with\qquad $G_{\sigma}( X):= G(X) + \sigma XX^{\T}$.}
\end{equation}
The two NEPv are equivalent in that:
$(X_*,\Lambda_*)$ is a solution to NEPv~\eqref{eq:alnepv}
$G(X)X=X\Lambda$, if and only if $(X_*, \Lambda_{\sigma*})$ with
$\Lambda_{\sigma*} = \Lambda_* + \sigma I$ is a solution to the level-shifted
NEPv~\eqref{eq:gsigma}.
Then, by a straightforward verification, the sequence of $\{X_i\}$
from~\Cref{alg:LS-scf} satisfy
\begin{equation}\label{eq:lsscf}
	G_{\sigma}(X_i) X_{i+1} = X_{i+1} \Lambda_{\sigma,i+1}
	\quad\text{for $i=0,1,\dots$},
\end{equation}
where the eigenvalues of $\Lambda_{\sigma,i+1}$ are the $k$ largest eigenvalues of
$G_{\sigma}(X_i)$. We have again assumed $X_i\in\bbod$,
which holds locally if $\cR(X_i)$ is close to $\cR(X_*)$,
%the same
as in~\eqref{eq:scf00}.
We can see that~\eqref{eq:lsscf} is exactly the plain SCF for the
level-shifted NEPv~\eqref{eq:gsigma}.
By this interpretation, we explain in the following two
major benefits for applying the level-shifted SCF in~\Cref{alg:LS-scf}.

\smallskip\noindent
{\bf Eigenvalue repositioning.}
Previously,
for the purpose of solving  optimization problem~\eqref{eq:gopt},
we assumed the desired solutions of the aligned NEPv~\eqref{eq:alnepv}
are those with $\Lambda$
corresponding to the $k$ largest eigenvalues of $G(X)$.
This assumption stems more from past research experiences on special cases of
optimization problem~\eqref{eq:gopt} \cite{wazl:2022,Zhang:2020,zhwb:2020}
than from rigorous mathematical proofs. % for \eqref{eq:gopt} in general.
For optimization problem~\eqref{eq:gopt} in
its generality, there is indeed no guarantee that $\Lambda$ has
such a property at optimality;
see, e.g.,~\cite{Bai:2018} for an example of optimization problem~\eqref{eq:gopt}
with $D=0$, where the target eigenvalue is not an extreme eigenvalue.

%\marginpar{\tiny any counterexample?} %Until it is proven otherwise,
%Hence, we have to keep in mind that at optimality
%$\Lambda$ may or may not correspond to the $k$ largest eigenvalues of $G(X)$.

Having that said, we also notice that the assumption above is actually
not essential for the application of SCF iteration,
because otherwise we can level-shift  NEPv~\eqref{eq:alnepv} to NEPv~\eqref{eq:gsigma}.
By the eigendecomposition~\eqref{eq:eigdgx},
we have
\begin{equation}\label{eq:eigdgsx}
G_{\sigma}(X_*)
%= G(X_*) + \sigma X_*X_*^{\T}
=
[X_*, X_{*\bot}]
\begin{bmatrix} \Lambda_* + \sigma I& \\ & \Lambda_{*\bot} \end{bmatrix}
[X_*, X_{*\bot}]^{\T}.
\end{equation}
Since $\Lambda_*$ and $\Lambda_{*\bot}$ are fixed,
each of the $k$ eigenvalues in $\Lambda_* + \sigma I$ can be made larger than
those in $\Lambda_{*\bot}$ by choosing $\sigma$ sufficiently large.
Particularly, if
\begin{equation}\label{eq:lssigma}
	\sigma > \lambda_{\max}(G(X_*))- \lambda_{\min}(G(X_*)),
\end{equation}
then the eigenvalues of $\Lambda_{\sigma*}= \Lambda_* + \sigma I$ always consist of the $k$
largest eigenvalues of $G_{\sigma}(X_*)$.
In this case, we can apply the plain SCF~\eqref{eq:lsscf}
to the level-shifted NEPv~\eqref{eq:gsigma},
using the top $k$ eigenvalues.
A similar idea of using level-shifting to reposition the desired eigenvalues of NEPv
has also been explored in~\cite{Bai:2018}.

\smallskip\noindent
{\bf Convergence of level-shifted SCF.}
As another important property, level-shifting can also fix the potential
divergence issue of the plain SCF. It is well-known that, under mild assumptions,
the level-shifted SCF is always locally convergent
when $\sigma$ is sufficiently large for unitarily invariant NEPv;
see, e.g.,~\cite{Bai:2022,Cances:2010}.
For \Cref{alg:LS-scf},
due to its equivalence to the plain level-shifted SCF in~\eqref{eq:lsscf}
for the aligned NEPv, we expect such convergence property to still hold.

We first consider the local convergence rate of \Cref{alg:LS-scf}.
Similarly to $\scrL$ in~\eqref{eq:lz},
we define  linear operator
$\scrL_\sigma \colon \mathbb R^{(n-k)\times k}\to \bbR^{(n-k)\times k}$
for the level-shifted NEPv~\eqref{eq:gsigma} at
a \dproper~solution $X_*$ as
\begin{equation}\label{eq:lzs}
	\scrL_\sigma(Z) := S_\sigma(X_*) \odot \left( X_{*\bot}^{\HH}\ \cdot \mathbf
	DG_\sigma(X_*)[X_{*\bot}Z]\cdot X_*\right),
\end{equation}
where $S_\sigma(X_*) \in\mathbb R^{(n-k)\times k}$
is with entries $[S_\sigma(X_*)]_{ij} = \left( \lambda_j - \lambda_{k+i} +\sigma
\right)^{-1}$,  recalling the eigen-decomposition~\eqref{eq:eigdgx}.
It follows from~\eqref{eq:gsigma} that $\mathbf DG_\sigma(X_*)[Y] \equiv \mathbf
DG(X_*)[Y] + \sigma (X_*Y^{\T} + YX_*^{\T})$,
and hence
\begin{equation}\label{eq:lzsq}
	\scrL_\sigma(Z)
	= S_\sigma(X_*) \odot \cQ(Z)  + Z,
\end{equation}
where $\cQ\colon \bbR^{(n-k)\times k}\to \bbR^{(n-k)\times k}$
is a linear operator given by
\begin{align}\label{eq:qop}
	\cQ(Z)= \Lambda_{*\bot} Z - Z\Lambda_* + X_{\bot*}^{\T}\D
	G(X_*)[X_{*\bot} Z]X_*,
\end{align}
similarly to~\cite[Theorem 5.1]{Bai:2022}.
By the convergence analysis in~\cite{Bai:2022},
as also discussed in~\Cref{sec:local},
the spectral radius $\rho(\scrL_\sigma)$ gives the local convergence rate of
the plain level-shifted SCF~\eqref{eq:lsscf} and, thus, that of~\Cref{alg:LS-scf} as well.
To guarantee local convergence of~\Cref{alg:LS-scf} to $X_*$, we
hence need $\rho(\scrL_\sigma)<1$.

Let us establish a sufficient condition on $\sigma$ for $\rho(\scrL_\sigma)<1$ to hold.
We observe that the scaling matrix $S_\sigma(X_*)$ in~\eqref{eq:lzsq}
has positive entries that go to $0$ as $\sigma\to +\infty$.
Therefore, assuming $\cQ(Z)$ is negative definite in the sense that
\begin{equation}\label{eq:negdef}
	\tr(\,Z^{\T}\cQ(Z)\,) < 0
	\qquad\text{for all nonzero $Z\in\bbR^{(n-k)\times k}$},
\end{equation}
we will have $\scrL_\sigma(Z)<1$ for sufficiently large $\sigma$,
noticing that $\scrL_\sigma(\,\cdot\,)$ approaches the identity operator as $\sigma\to\infty$.
Precisely, under condition~\eqref{eq:negdef},
it will hold
\begin{equation}\label{eq:sigma}
	\scrL_\sigma(Z)<1
	\quad \text{for all}\quad
	\sigma > \sigma_L:=-\frac{\mu_{\min}}{2} - \left[\lambda_{k}(G(X_*)) -
	\lambda_{k+1}(G(X_*))\right],
\end{equation}
where $\mu_{\min}<0$ is the smallest eigenvalue of  linear operator
$\cQ$ in~\eqref{eq:qop}.
The result~\eqref{eq:sigma} is essentially due
to~\cite[Theorem 5.1]{Bai:2022},
which, although established for NEPv $H(X)X=X\Lambda$ with $\Lambda$
being the smallest eigenvalues of $H(X)$, can be quickly adapted
to~\eqref{eq:alnepv} by working with $-G(X)X = X\Lambda$.

Finally, the following theorem justifies~\eqref{eq:negdef} is indeed a mild
assumption for aligned NEPv~\eqref{eq:alnepv}, since the operator
$\cQ$ in~\eqref{eq:qop} is at least negative semi-definite at a global
maximizer $X_*$ of~\eqref{eq:gopt}.
Its proof is left to~\Cref{sec:proofqdef}.

\begin{theorem}\label{thm:qdef}
    Assume $\phi$ and $\psi$ in~\eqref{eq:gopt} are twice differentiable.
	Let $X_*\in\bbO^{n\times k}$ be a
	global maximizer of~\eqref{eq:gopt} %~\eqref{eq:gopt2},
	and $\cQ$ be defined in~\eqref{eq:qop}.
	Then it holds that
	\begin{equation}\label{eq:qdef}
		\tr(\,Z^{\T}\cQ(Z)\,)\leq 0
		\qquad\text{for all nonzero $Z\in\bbR^{(n-k)\times k}$}.
	\end{equation}
\end{theorem}

%corresponds to the Hessian of the optimization problem~\eqref{eq:gopt2}, and
%so it is at least negative semi-definite at
%a global maximizer $X_*$.

%\newpage

\section{Numerical Experiments}\label{sec:egs}
In this section, we verify  our theoretical
results  by numerical experiments. We will
compare the observed convergence rate with our theoretical estimate,
i.e., the spectral radius for the corresponding $\scrL$ operator in~\eqref{eq:lz}.
The spectral radius $\rho(\scrL)$ is computed by  MATLAB's
built-in function \texttt{eigs}, and the `exact' solution $X_*$
is computed by \Cref{alg:scf} or,
when it fails to converge, \Cref{alg:LS-scf} with a suitable level-shift $\sigma$.
Our stopping criterion is
$$
\NRes(X_i)\le \mbox{\tt tol}=10^{-13},
$$
where $\NRes(\,\cdot\,)$ is the normalized residual \eqref{eq:NRes} evaluated
at the most recent approximation using
the matrix 1-norm (i.e., the maximum column sum in absolute value) for computational convenience.
We refer to~\cite{Bai:2022} for more details about the computation of
$\rho(\scrL)$.
All experiments are carried out in MATLAB and run
on a Dell desktop with an Intel i9-9900K CPU and 16G memory.
	The MATLAB scripts implementing the algorithms and the data used to
	generate the numerical results in this paper
	can be accessed at \url{https://github.com/ddinglu/uinepv}.

Our testing NEPv arise from two optimization problems in the form of
%\marginpar{\tiny to work on (old)}
\eqref{eq:gopt}, both of which are defined by
three matrices $A,\,B\in\bbR^{n\times n}$ and $D\in\bbR^{n\times k}$
along with an additional parameter,
where $A,\,B$ are symmetric and $B\succ 0$, and $k < n$.
The first optimization problem is
\begin{equation}\label{eq:maxex}
	\max_{X\in\bbO^{n\times k}}f_{\alpha} (X)
	\quad\text{with}\quad
	f_{\alpha}(X):=
	(1-\alpha)\cdot \frac{\tr(X^{\T}AX)}{\tr(X^{\T}BX)}
	+\alpha\cdot \frac{\tr(X^{\T}D)}{\sqrt{\tr(X^{\T}BX)}},
\end{equation}
where $\alpha\in[0,1]$ is a parameter, and
the second one is \eqref{eq:max} mentioned earlier at the beginning of this paper:
\begin{equation}\label{eq:max'}
	\max_{X\in\bbO^{n\times k}} f_{\theta} (X)
	\quad\text{with}\quad
	f_{\theta}(X):=
	\frac{\tr(X^{\T}AX + X^{\T}D)}{[\tr(X^{\T}BX)]^\theta},
\end{equation}
where $\theta$ is a parameter. Problem~\eqref{eq:max'} has been studied
in \cite{wazl:2022}, where
\Cref{alg:scf} is proved to be linearly convergent  from any initial guess,
when $\theta\in\{0,1\}$, or $0<\theta<1$ and initially $\tr(X_0^{\T}AX_0 + X_0^{\T}D)\ge 0$,
but its theoretical rate of convergence was not
investigated. We will fill this gap and also demonstrate what could happen
when $\theta$ is outside of the interval $[0,1]$.

With $D=0$ and $\theta=1$,
both  problems  degenerate to
the same one --- the trace-ratio
optimization problem from the linear discriminant analysis (LDA).
For the case, the NEPv approach has been well studied
in~\cite{zhln:2010,zhln:2013}.
The corresponding NEPv~\eqref{eq:nepva} is unitarily
invariant, and~\Cref{alg:scf}, which coincides with the plain SCF
(i.e., line 3 replaced by $X_{i+1}=\widetilde X_{i+1}$),
is globally and locally quadratically convergent to the global optimal solution generically.
However, for the general cases with $\alpha D \neq 0$ in \eqref{eq:maxex}
or $\theta\ne 1$ in \eqref{eq:max'},
the two optimization problems are different,
and \Cref{alg:scf} generally loses quadratic
convergence or even possibly diverges (if without the help from level-shifting)
for the associated NEPv, as will be shown in a moment.

\subsection{NEPv from \eqref{eq:maxex}}\label{sec:egsa}
Optimization problem~\eqref{eq:maxex} is in the form of~\eqref{eq:gopt}
with  unitarily invariant functions
\begin{equation}\label{eq:funex}
	\phi(X) = (1-\alpha)\cdot \frac{\tr(X^{\T}A X)}{\tr(X^{\T}BX)}
	\qquad\text{and}\qquad
	\psi(X) = \frac{\alpha}{\sqrt{\tr(X^{\T}BX)}}.
\end{equation}
Its KKT condition, by~\Cref{thm:nepv}, is equivalent to  NEPv
\begin{subequations}\label{eq:nepva}
\begin{equation}\label{eq:nepva-1}
	H_{\alpha}(X) X = X\Lambda,
\end{equation}
where the subscript $\alpha$ indicates its dependence on
parameter $\alpha$, and
\begin{equation}\label{eq:hxex}
	H_{\alpha}(X) =
	\frac{2}{\tr(X^{\T}BX)} \left((1-\alpha)A - \phi(X)\cdot B\right)
	- \frac{\tr(X^{\T}D) \cdot \psi(X)}{\tr(X^{\T}BX)} \cdot B
	+ \psi(X)\cdot(DX^{\T} + XD^{\T}),
\end{equation}
\end{subequations}
which is derived, using
\begin{equation}\label{eq:dfex}
	\left\{
\begin{aligned}
	\frac {\partial \phi(X)}{\partial X} &= H_{\phi}(X) X \qquad\text{with}\qquad H_{\phi}(X)=
	\frac{2}{\tr(X^{\T}BX)} \Big[(1-\alpha)\cdot A - \phi(X)\cdot B\Big],\\
	\frac {\partial \psi(X)}{\partial X} &= H_{\psi}(X) X \qquad\text{with}\qquad H_{\psi}(X) =
	 \frac{-\psi(X)}{\tr(X^{\T}BX)} \cdot B.
	%-\psi(X)^3 \cdot B,
\end{aligned}\right.
\end{equation}
By varying $\alpha\in[0,1]$, we can construct a variety of NEPv~\eqref{eq:nepva} for testing.
For $\alpha =0$, the initial guess $X_0$ of SCF
is set to the eigenvectors of the $k$ largest eigenvalues of
the linear problem $Ax=\lambda Bx$.
As $\alpha$ increases from $0$ to $1$, we use the computed solution from the previous NEPv
as a starting guess for the next one.

For the purpose of calculating the local rate of convergence, we also
obtain by straightforward derivation
\begin{equation}\label{eq:dhex}
\left\{
\begin{aligned}
	{\bf D}H_{\phi}(X)[E] &=
	-2\frac{\tr(X^{\T}BE)}{\tr(X^{\T}BX)} \cdot H_{\phi}(X)
	-2\frac{\tr(X^{\T} H_{\phi}(X) E)}{\tr(X^{\T}BX)}  \cdot B,
	\\
	{\bf D}H_{\psi}(X)[E] &= -3\frac{\tr(X^{\T} H_{\psi}(X) E)}{\tr(X^{\T}BX)} B.
\end{aligned}\right.
\end{equation}
We construct the aligned NEPv~\eqref{eq:alnepv} through alignment
and the derivative operator $\D G$ through \eqref{eq:dgx},
using the derivatives in~\eqref{eq:dhex} together with ${\bf
D}Q_{\bf o}$ and ${\bf D}M$ from~\eqref{eq:dmdq}.
The corresponding linear operator $\scrL$ by~\eqref{eq:lz} is then
obtained.

\begin{figure}[thbp]
\begin{center}
\includegraphics[width=0.44\textwidth]{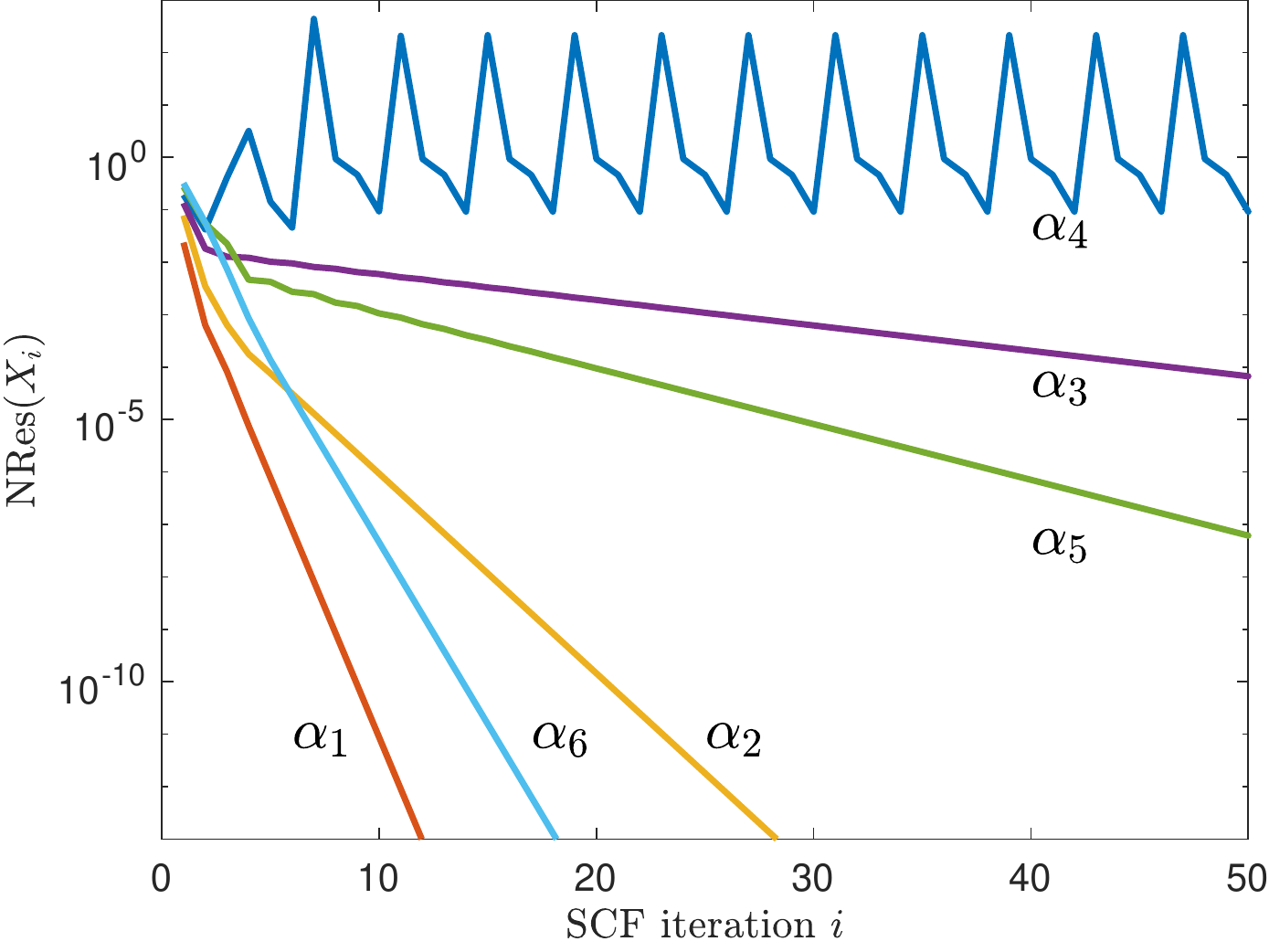}
\includegraphics[width=0.45\textwidth]{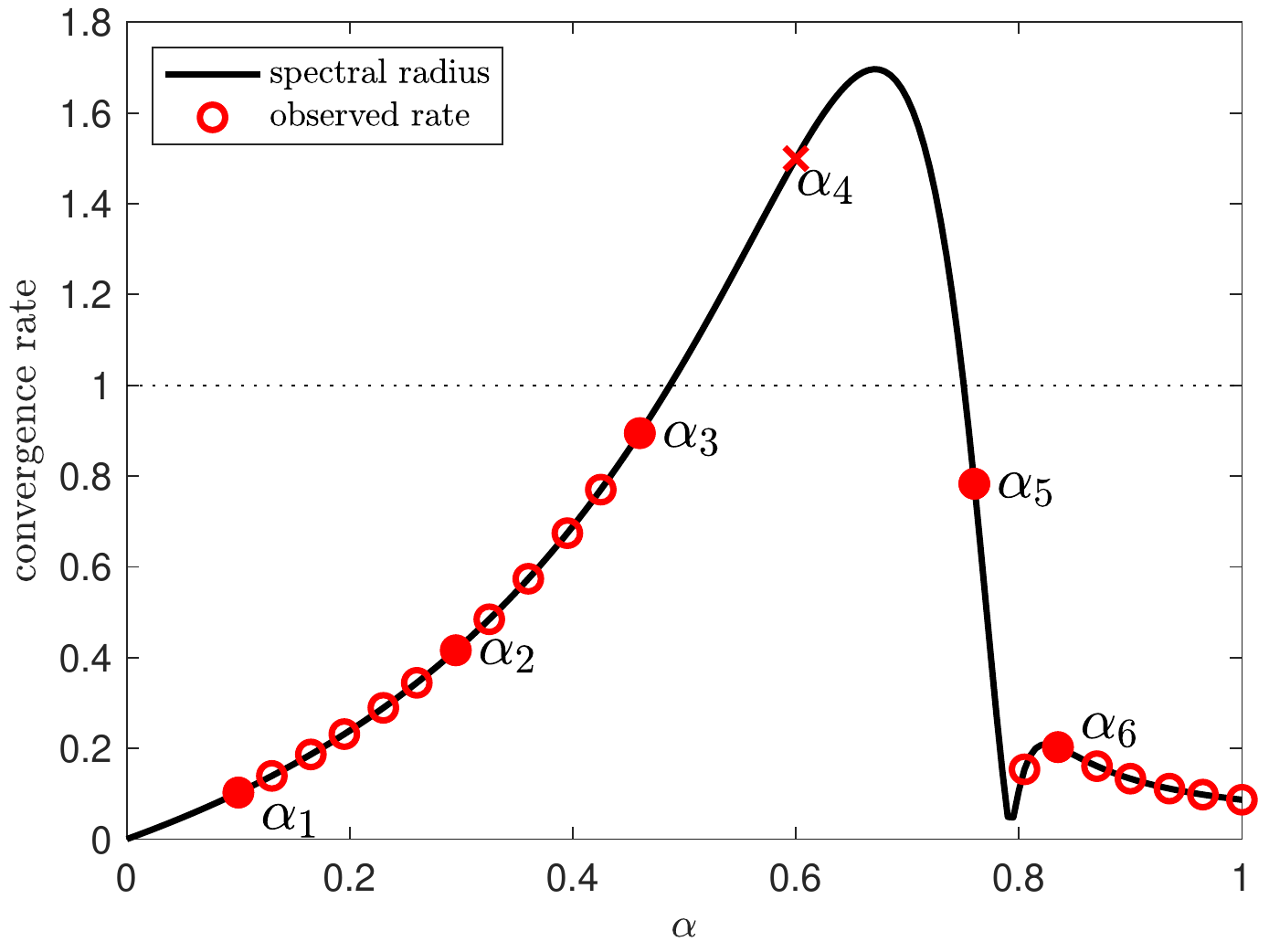}
\caption{
	\Cref{alg:scf} on Example~\ref{ex:1}.
	{\em Left:\/} The iterative history for solving NEPv~\eqref{eq:nepva} at
	a few sampled $\alpha$ (correspondingly marked as $\bullet$ and $\times$ on the right plot).
	{\em Right:\/} The curve of spectral radius $\rho(\scrL)$ as a function of
	 parameter $\alpha\in[0,1]$ (based on $200$ equally spaced $\alpha$),
	and the observed rates  of convergence (marked by $\bullet$ and $\circ$) at a number of values of $\alpha$, including those
    sampled $\alpha$ on the left plot,
	and `$\times$' indicates that SCF is divergent.
}\label{fig:rnd}
\end{center}
\end{figure}

\begin{figure}[thbp]
\begin{center}
\includegraphics[width=0.455\textwidth]{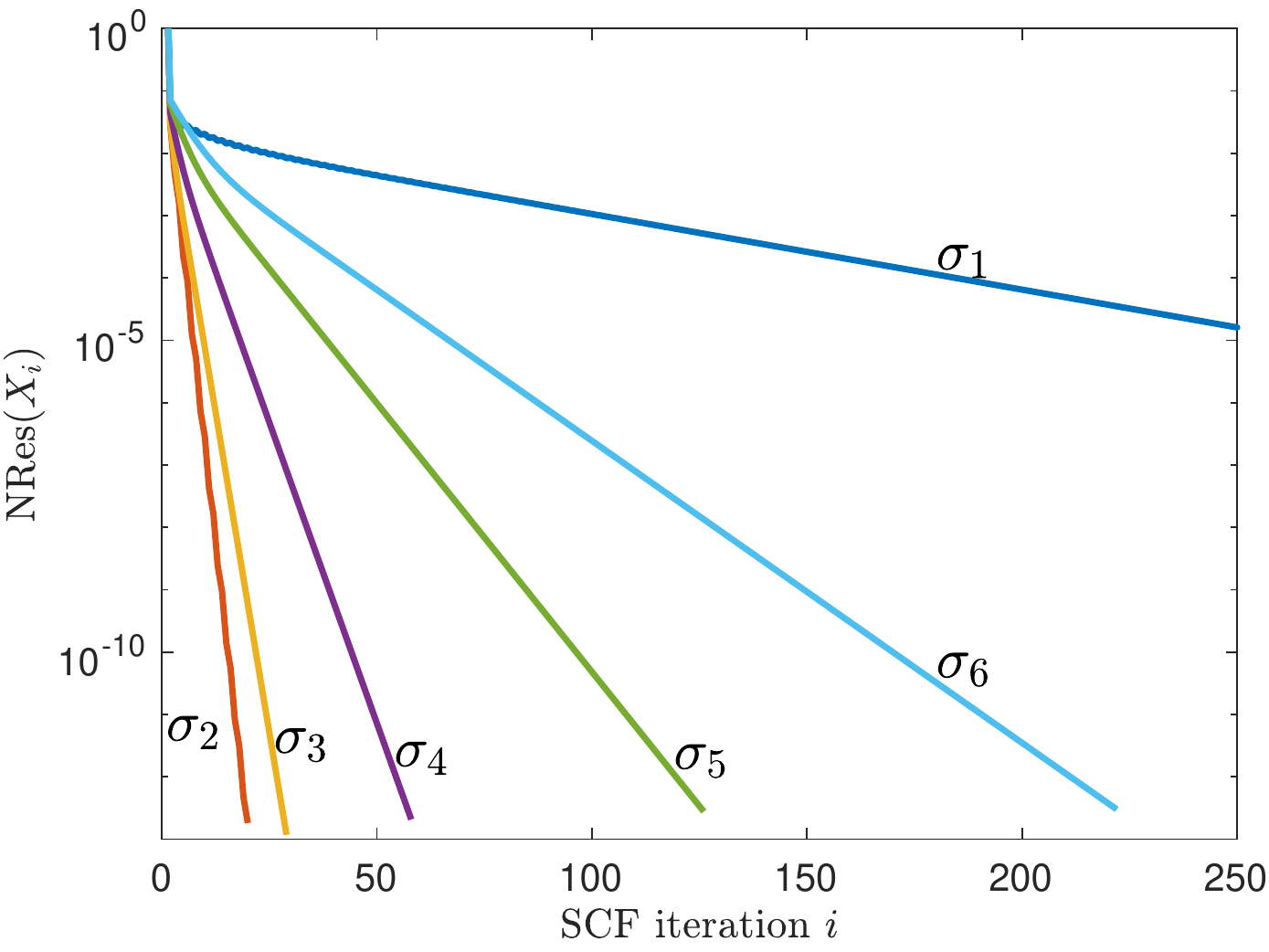}
\includegraphics[width=0.44\textwidth]{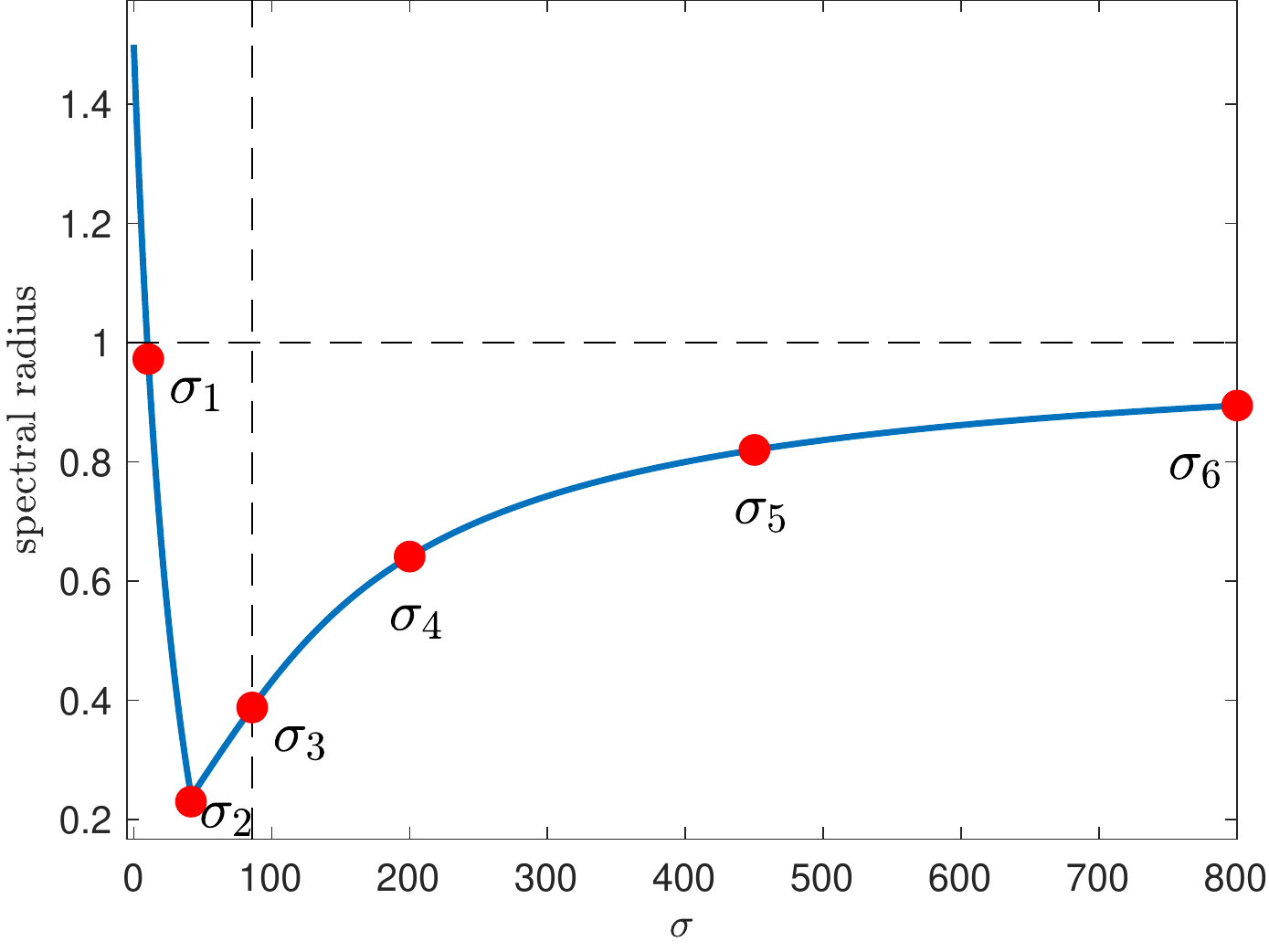}
\caption{
	\Cref{alg:LS-scf} on
	Example~\ref{ex:1} with $\alpha = 0.6$ (i.e., $\alpha_4$ in Figure~\ref{fig:rnd}, at which \Cref{alg:scf} diverges).
	{\em Left:\/} The iterative history of \Cref{alg:LS-scf} (level-shifted SCF)
	with a few sampled level-shifts $\sigma$ (correspondingly marked as $\bullet$ on the right plot).
	{\em Right:\/} The curve of spectral radius $\rho(\scrL_\sigma)$ as a function of
	the level-shift $\sigma$ and the observed rates of convergence at the sampled
	level-shifts $\sigma$. % (marked by $\bullet$).
	The vertical dashed line corresponds to the theoretical lower bound
	$\sigma_L=85.83$  given by~\eqref{eq:sigma} (as $\sigma_3$ in the plots).
}\label{fig:rndls}
\end{center}
\end{figure}

\begin{example}\label{ex:1}
{\rm %
We consider problem~\eqref{eq:maxex} with $D\in\bbR^{n}$ being a single vector, i.e.,
$k=1$. The alignment operation at line~\ref{alg:scf:normalize} of~\Cref{alg:scf}
is simply $\widetilde X_i= \pm X_i$ with $\pm$ corresponding to
the sign of $X_i^{\T}D\in\bbR$. Letting also $n=3$,
we randomly generate
\[
A =
\left[\begin{array}{rrr}
   -3.242&   -0.450&    1.807\\
   -0.450&   -1.630&    0.790\\
    1.807&    0.790&    0.226
\end{array}\right]
,\
B =
\left[\begin{array}{rrr}
    0.592&    1.873&    0.175\\
    1.873&    6.332&    0.617\\
    0.175&    0.617&    0.488
\end{array}\right]
,\
D =
\left[\begin{array}{r}
	-9.122\\
    0.421\\
    3.134
\end{array}\right].
\]
In what follows, we examine the convergence of~\Cref{alg:scf} on NEPv \eqref{eq:nepva}
as $\alpha$ varies in $[0,1]$.

\Cref{fig:rnd} compares the observed rates of convergence of~\Cref{alg:scf}
with the theoretical rates $\rho(\scrL)$, as well as the iterative histories of
normalized NEPv residuals, at several values of $\alpha$.
The left plot shows linear convergence of~\Cref{alg:scf} at
a few sampled $\alpha\neq 0$ when \Cref{alg:scf} actually converges 
%\Blue{and one $\alpha$ when \Cref{alg:scf} does not}. The right
and one $\alpha$ when \Cref{alg:scf} does not. The right
plot is the curve of spectral radius $\rho(\scrL)$ as $\alpha$ varies in $[0,1]$.
We can see that:
\begin{itemize}
  \item %
	At those $\alpha$
	where \Cref{alg:scf} converges, the observed rates of convergence and
	the theoretical ones by $\rho(\scrL)$ match very well. For example,
	\[
		\mbox{at}\,\,\alpha_3 = 0.46:\qquad
		\text{observed rate} \approx 0.894493\cdots,
		\quad
		\rho(\scrL) \approx 0.894490\cdots,
	% alpha = 0.46
	% 8.944933766963492e-01 % observed
	% 8.944904800614140e-01 % sprd
	\]
	matching up to $5$ significant decimal digits.
	We  conclude the spectral radius $\rho(\scrL)$ provides
	sharp estimation for the true convergence rates.

  \item %
	 $\rho(\scrL) > 1$ approximately for $\alpha \in [0.49, 0.75]$,
	where \Cref{alg:scf} indeed diverges numerically;
	see, e.g., the convergence behavior for $\alpha_4=0.6$ on the left plot. % in the left plot.
	Those $\rho(\scrL)$ are calculated  with solutions
	$X_*$ to NEPv \eqref{eq:nepva} computed by
	the level-shifted SCF, \Cref{alg:LS-scf}, with $\sigma = 100$.
    It is interesting to notice that \Cref{alg:scf} converges for $\alpha$ near $0$ and $1$ but diverges
    in the middle. This can be explained. In fact, optimization problem \eqref{eq:maxex} for $\alpha=0$
    is the same as the one from LDA \cite{zhln:2010,zhln:2013}, while for
    $\alpha=1$ it is the OCCA subproblem \cite{zhwb:2020}. For both cases,
    \Cref{alg:scf} provably converges.

  \item %
	$\rho(\scrL)\to 0$ as $\alpha \to  0$,
	indicating superlinear convergence of the algorithm at $\alpha = 0$,
	consistent with the fact that~\Cref{alg:scf} is quadratically convergent
	in such a case~\cite{zhln:2013}.
	At the other end $\alpha=1$, the spectral radius is
	a small number $\rho(\scrL)\approx  0.086$, indicating rapid linear
	convergence of~\Cref{alg:scf}.
\end{itemize}

To demonstrate the effectiveness of level-shifting,~\Cref{fig:rndls}
shows the convergence of level-shifted SCF (\Cref{alg:LS-scf})
on NEPv~\eqref{eq:nepva} for $\alpha = 0.6$ (i.e., $\alpha_4$ in Figure~\ref{fig:rnd}), at which \Cref{alg:scf} diverges.
%as shown in \Cref{fig:rnd}.
The left plot of \Cref{fig:rndls} illustrates the linear convergence of \Cref{alg:LS-scf} at
various level-shifts $\sigma$,  where a proper choice of
$\sigma$ can lead to a significant acceleration of convergence for SCF.
The right plot shows the spectral radius
of $\scrL_\sigma$ by~\eqref{eq:lzs} %, for the level-shifted NEPv~\eqref{eq:gsigma},
as  level-shift $\sigma$ varies,
where the optimal level-shift $\sigma_* \approx 41.88$
with the minimal value $\rho(\scrL_\sigma) \approx 0.239$.
It is observed that $\rho(\scrL_\sigma) < 1$ for all level-shifts $\sigma \gtrsim 9.87 $, but
the theoretical lower bound in~\eqref{eq:sigma} can
only %guarantees
guarantee that $\rho(\scrL_\sigma) < 1$ and that~\Cref{alg:LS-scf} is locally convergent
when level-shift $\sigma\gtrsim\sigma_L=85.83$, greatly overestimating the
observed one. % $\sigma \gtrsim 9.87$.
} %
\end{example}

\begin{figure}[thbp]
\begin{center}
\includegraphics[width=0.45\textwidth]{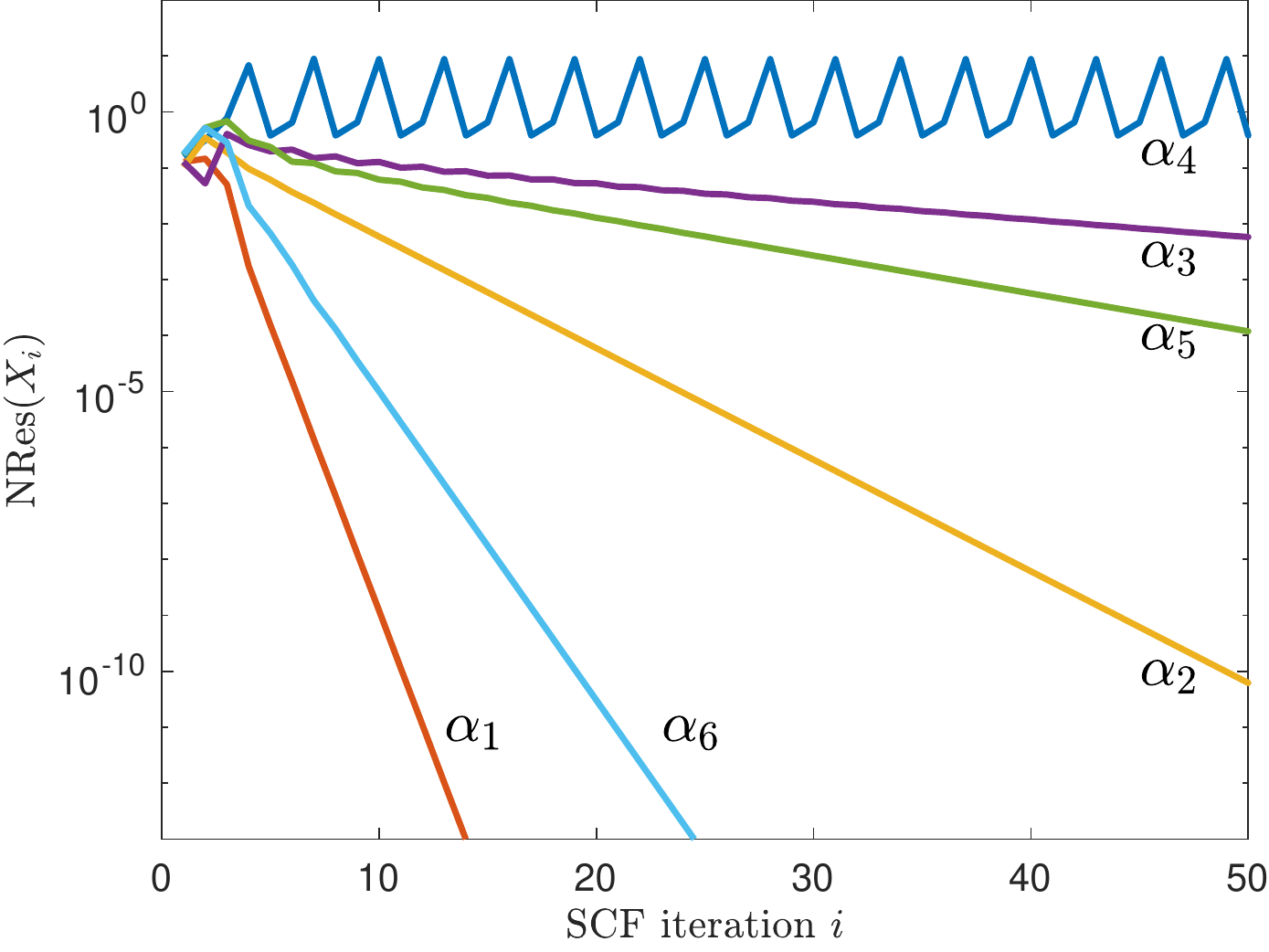}
\includegraphics[width=0.46\textwidth]{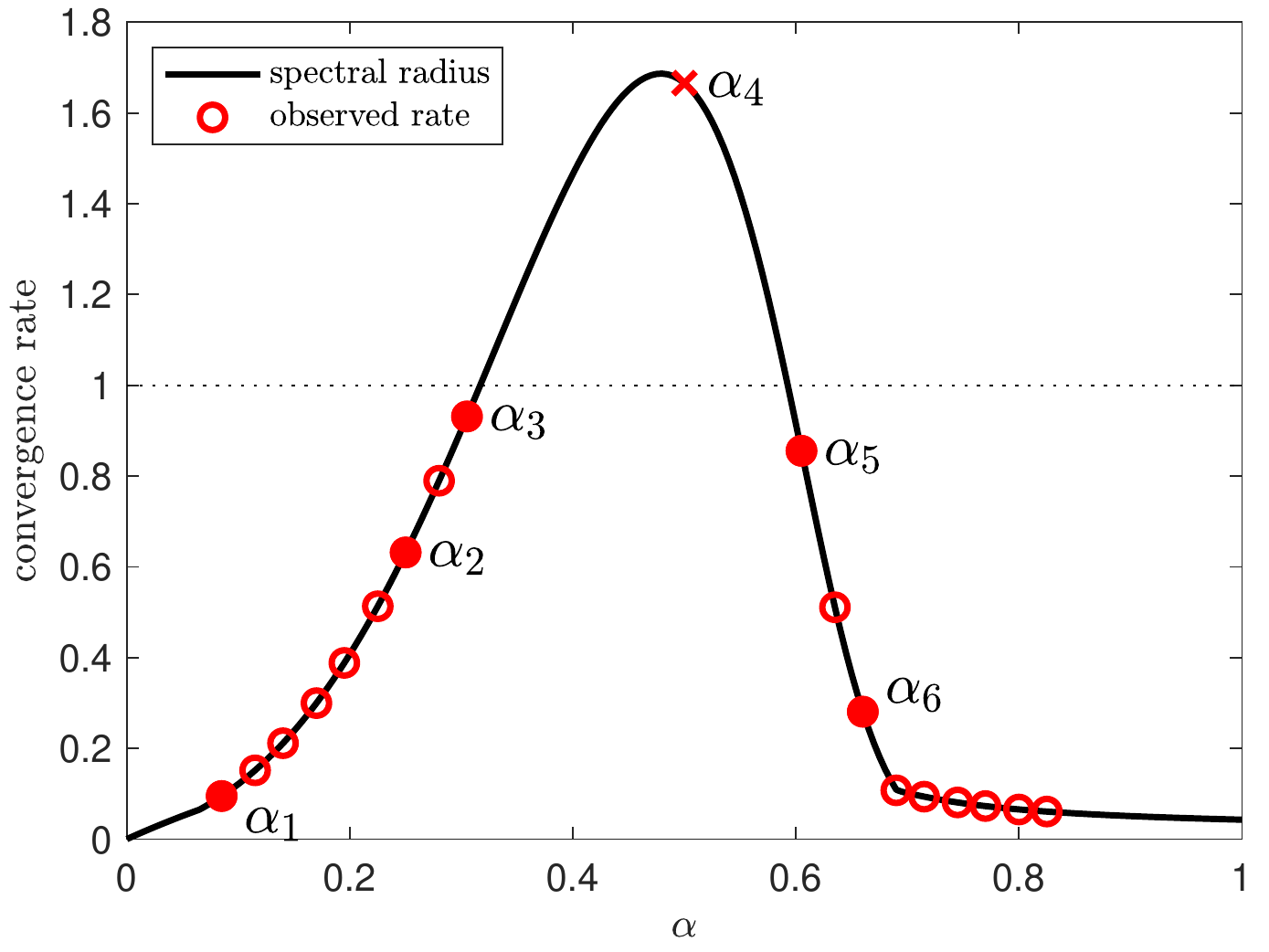}
\caption{
	\Cref{alg:scf} on Example~\ref{ex:2}.
	{\em Left:\/} The iterative history for solving NEPv~\eqref{eq:nepva} at
	a few sampled $\alpha$ (correspondingly marked as  $\bullet$ and $\times$ on the right plot).
	{\em Right:\/} the curve of spectral radius $\rho(\scrL)$ as a function of
	 parameter $\alpha\in[0,1]$ (based on $200$ equally spaced $\alpha$),
	and the observed rates of convergence (marked by $\bullet$ and $\circ$) at a number of values of $\alpha$, including those
    sampled $\alpha$ on the left plot,
	and `$\times$' indicates that SCF is divergent.
}\label{fig:rndk2}
\end{center}
\end{figure}

\begin{figure}[thbp]
\begin{center}
\includegraphics[width=0.455\textwidth]{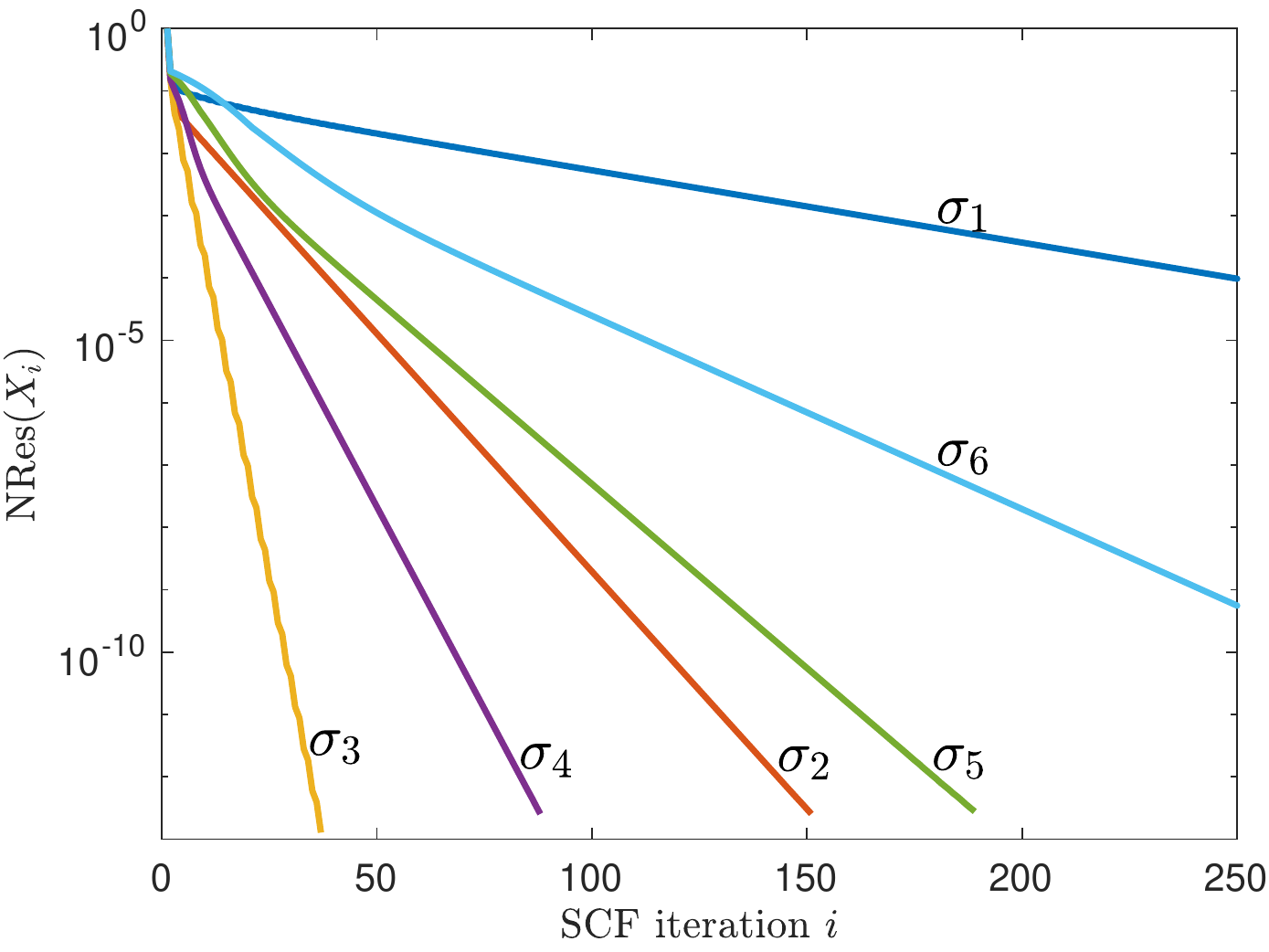}
\includegraphics[width=0.44\textwidth]{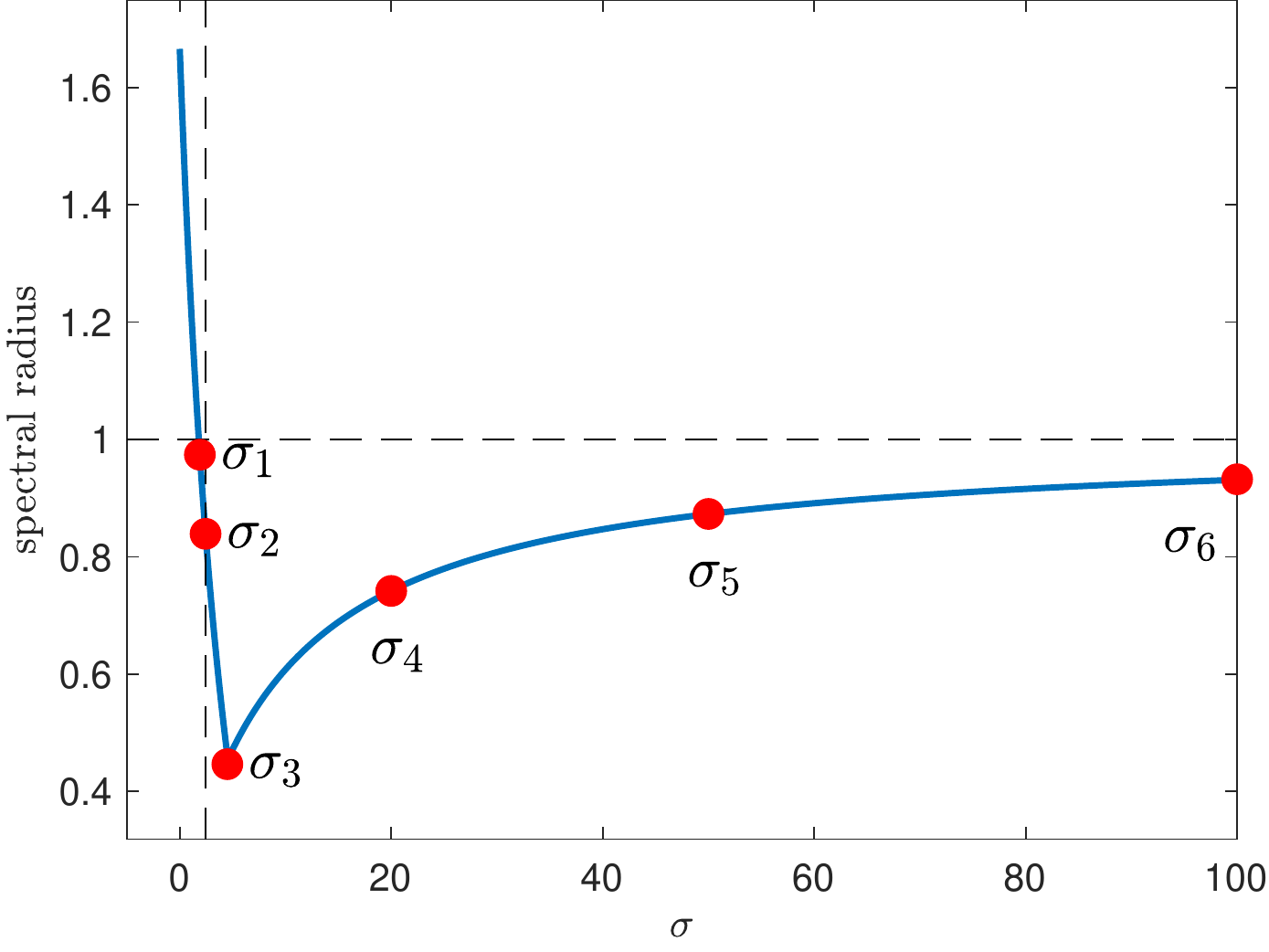}
\caption{
	\Cref{alg:LS-scf} on
	Example~\ref{ex:2} with $\alpha = 0.5$ (i.e., $\alpha_4$ in Figure~\ref{fig:rndk2}, at which \Cref{alg:scf} diverges).
	{\em Left:\/} The iterative history of \Cref{alg:LS-scf} (level-shifted SCF)
	with a few sampled level-shifts $\sigma$ (correspondingly marked as $\bullet$ on the right plot).
	{\em Right:\/} the curve of spectral radius $\rho(\scrL_\sigma)$ as a function of
	the level-shift $\sigma$ and the observed rates of convergence at the sampled
	level-shifts $\sigma$. % (marked by $\bullet$).
	The vertical dashed line corresponds to the theoretical lower bound
	$\sigma_L=2.44$ given by~\eqref{eq:sigma} (as $\sigma_2$ in the plots).
}\label{fig:rndlsk2}
\end{center}
\end{figure}

\begin{example}\label{ex:2}
{\rm %
This is an almost repeat of~\Cref{ex:1},
with the same testing matrices $A$ and $B$, except $k=2$ and
\[
	D =
\left[\begin{array}{rr}
  -1.430&  2.768\\
  -0.120& -0.630\\
   1.098&  2.229
\end{array}\right]\in\bbR^{3\times 2}.
\]
Since $k=2$, the alignment at line~\ref{alg:scf:normalize} of~\Cref{alg:scf}
does require the SVD (or the polar decomposition) of $X_i^{\T}D$.
\Cref{fig:rndk2} shows testing results by~\Cref{alg:scf} for several values of $\alpha\in[0,1]$.
We observe a similar performance of the algorithm
as in~\Cref{fig:rnd}.
We again find that $\rho(\scrL)$ provides a sharp estimation for the convergence rates.
For example,
\[
	\mbox{at}\,\,\alpha_3 = 0.305:\qquad
	\text{observed rate} \approx 0.930833\cdots,
	\quad
	\rho(\scrL) \approx 0.930798\cdots.
% alpha = 0.305
% 9.308333653286956e-01 % observed
% 9.307980854236972e-01 % sprd
\]
From the spectral radius curve as a function of $\alpha$, it is concluded that
\Cref{alg:scf} converges rapidly for $\alpha$ at
both ends of the interval $[0,1]$, but fails to converge for $\alpha$ in the
middle (approximately from $0.32$ to $0.59$),
where the spectral radius $\rho(\scrL)>1$
(those $\rho(\scrL)$ are calculated using the computed solution $X_*$
by \Cref{alg:LS-scf}  with $\sigma = 50$).
Next, we test the level-shifted SCF, \Cref{alg:LS-scf}, on
the NEPv with $\alpha = 0.5$ (i.e., $\alpha_4$ in Figure~\ref{fig:rndk2}), at which \Cref{alg:scf} diverges. As reported in~\Cref{fig:rndlsk2},
the optimal level-shift occurs at $\sigma_* \approx 4.57$ with
the spectral radius $\rho(\scrL_\sigma) \approx 0.455$.
For this NEPv, it is guaranteed that $\rho(\scrL_\sigma) < 1$ with any level-shift
$\sigma\gtrsim\sigma_L=2.44$ according to the theoretical lower bound~\eqref{eq:sigma}
 (as $\sigma_2$ in both left and right plots).
}
\end{example}

\begin{figure}[thbp]
\begin{center}
\includegraphics[width=0.45\textwidth]{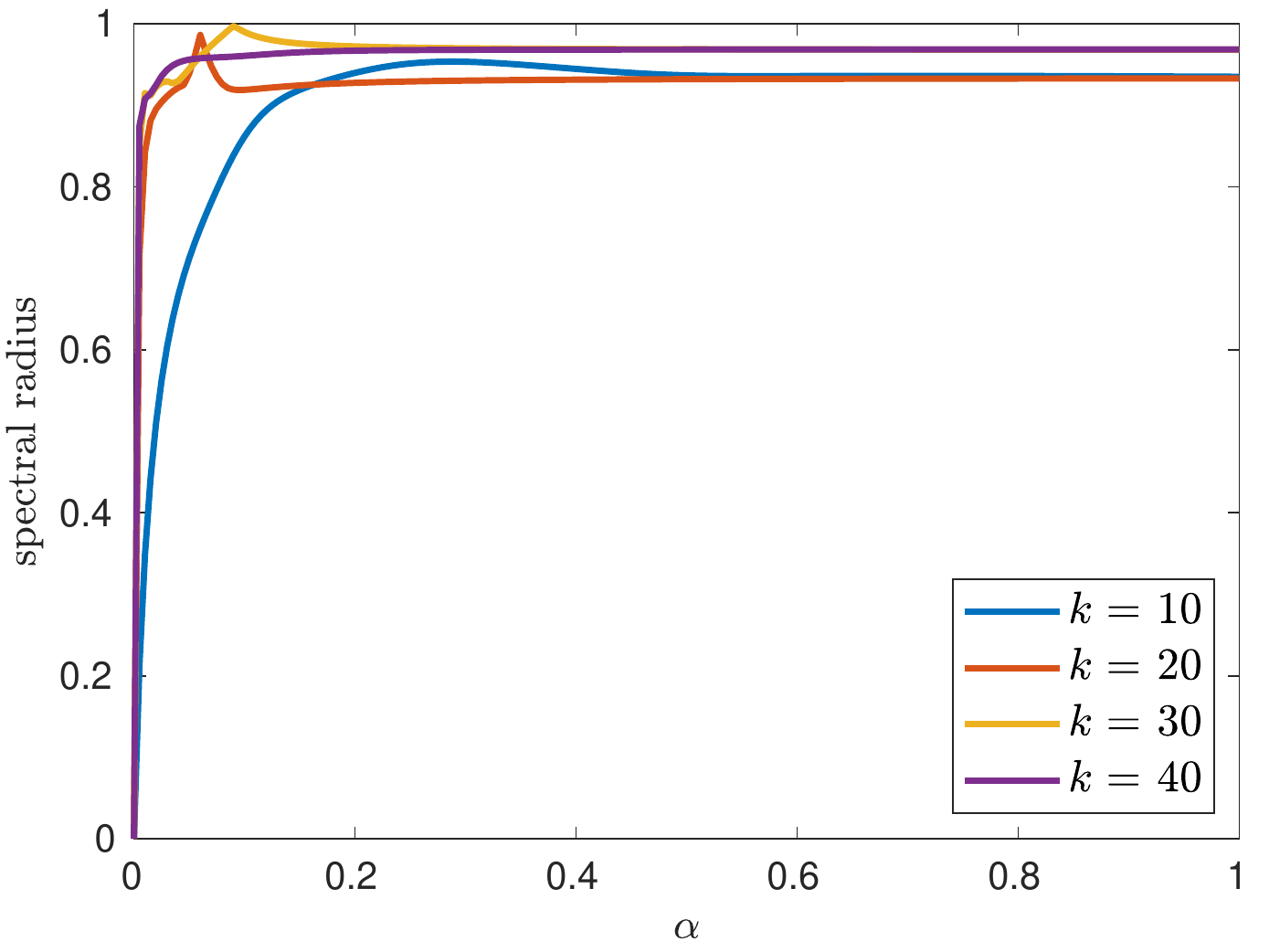}
\includegraphics[width=0.45\textwidth]{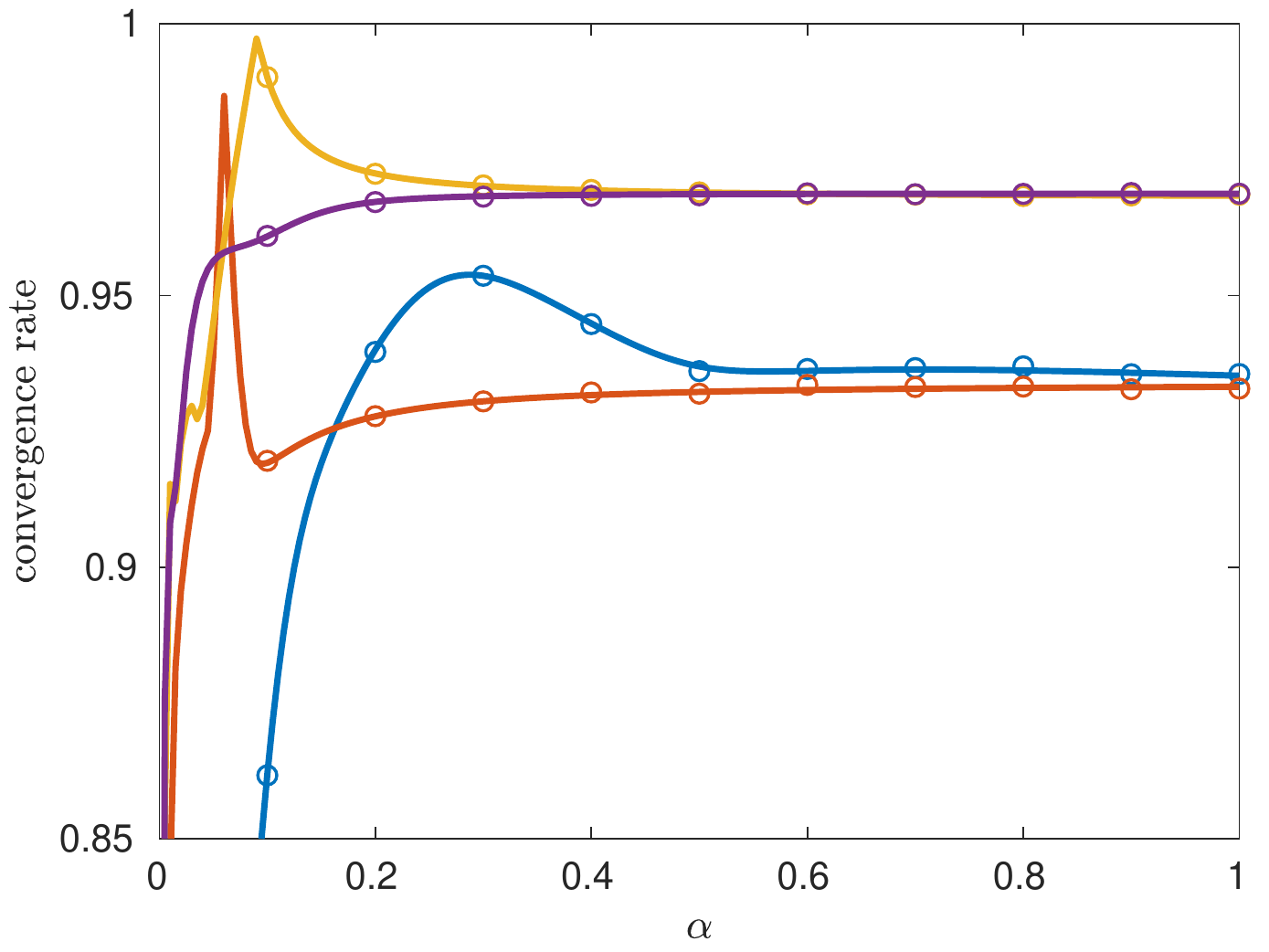}
\caption{\Cref{alg:scf} on Example~\ref{ex:3} with full-rank $D\in \bbR^{n\times k}$.
	{\em Left:\/} The curves of spectral radius $\rho(\scrL)$ as a function of $\alpha\in[0,1]$
    for different $k$ (based on $200$ equally spaced $\alpha$).
	{\em Right:\/} The observed rates of convergence of~\Cref{alg:scf} at a few sampled
	$\alpha$ (marked by $\circ$), zoomed in  for rates in $[0.85,1]$ for readability.
}\label{fig:ex2:full}
\end{center}
\end{figure}

\begin{figure}[thbp]
\begin{center}
\includegraphics[width=0.45\textwidth]{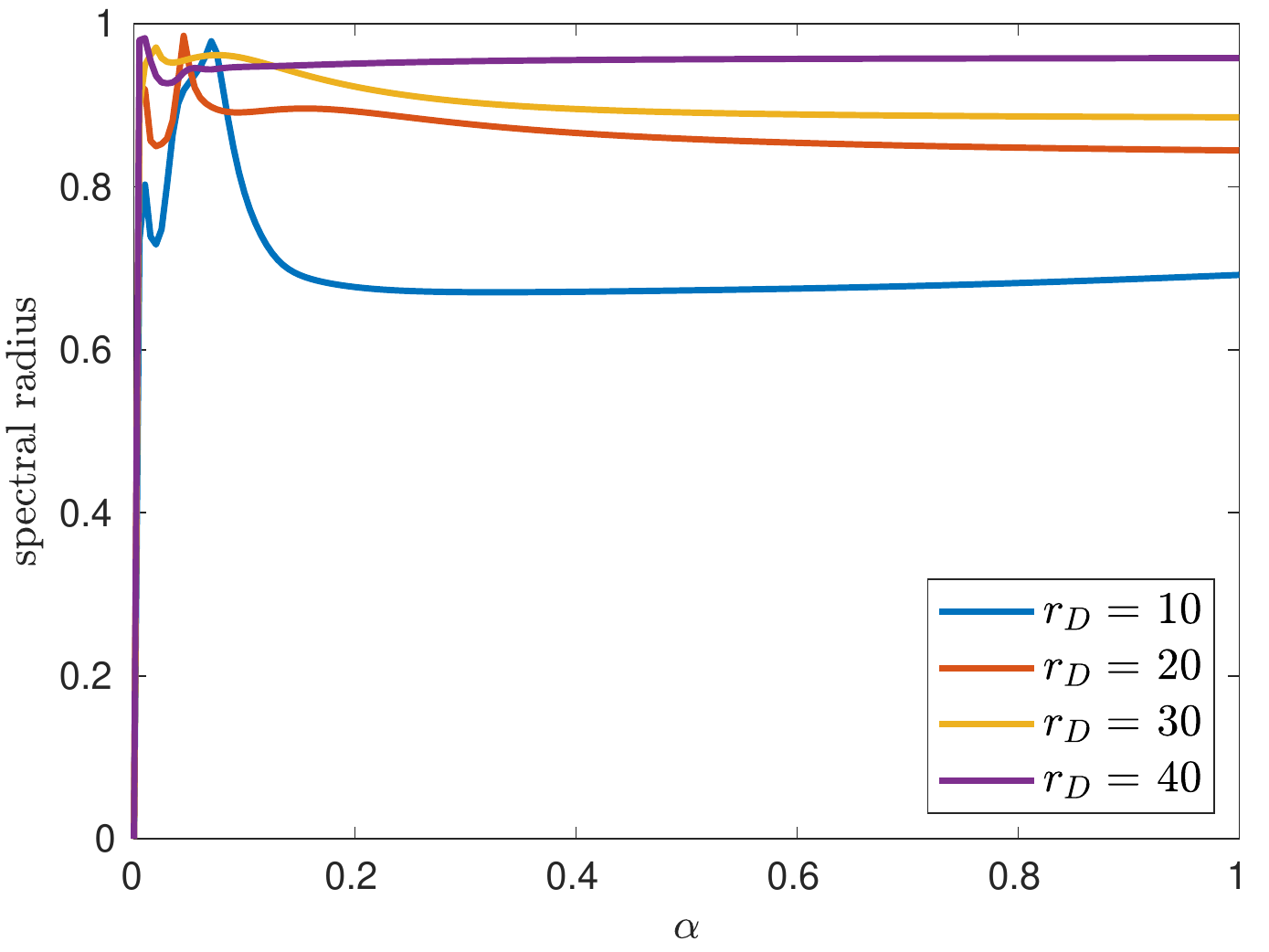}
\includegraphics[width=0.45\textwidth]{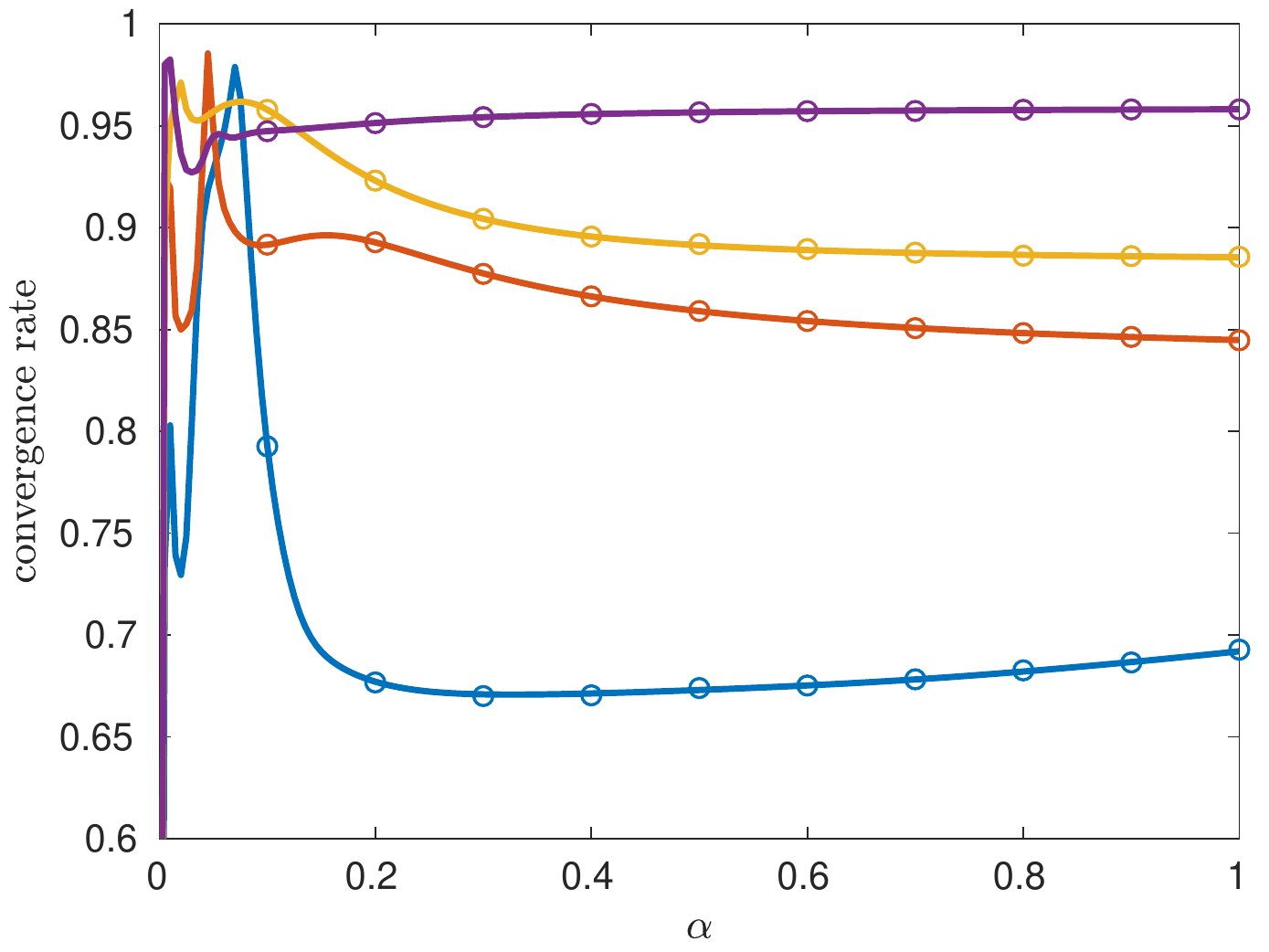}
\caption{\Cref{alg:scf} on Example~\ref{ex:3}\ with rank-deficient $D\in \bbR^{n\times k}$
	with  $k=50$ and different $\rkd\equiv\mrank(D)$. %$ =10,20,30,40$.
	{\em Left:\/} The curves of spectral radius $\rho(\scrL)$ as a function of $\alpha\in[0,1]$
	(based on $200$ equally spaced $\alpha$).
	{\em Right:\/} The observed rates of convergence of~\Cref{alg:scf} at sampled
	$\alpha$ (marked by $\circ$), zoomed in  for rates in $[0.6,1]$ for readability.
}\label{fig:ex2:singular}
\end{center}
\end{figure}

\begin{example}\label{ex:3}
{\rm %
We now consider larger $n$ and more general
$D\in\bbR^{n\times k}$ that is possibly rank-deficient unlike %than
the previous examples.
Set $A=\mbox{tridiag}(-1,2,-1)\in\bbR^{n\times n}$, a symmetric
tridiagonal matrix, and $B=\mbox{diag}(1,2,\dots,n)$, and
$n=200$.

In the first experiment, we consider full-rank $D\in\bbR^{n\times k}$,
generated by the MATLAB function $\texttt{randn}(n,k)$.
We test $k=10,20,30,40$, and for each $k$ we repeat the experiment as in the
previous examples, i.e.,  varying  parameter $\alpha$ from $0$ to $1$.
%and checking the convergence of \Cref{alg:scf}.
The results are summarized in~\Cref{fig:ex2:full},
where~\Cref{alg:scf} is locally convergent for all testing cases.
In the left plot,
the spectral radius moves from $0$ to close to $1$ as $\alpha$ increases,
showing that SCF quickly loses superlinear convergence with a tiny increase of
$\alpha$ from $0$, and for $\alpha\approx 1$,  \Cref{alg:scf} converges linearly.
Overall, the convergence appears slower than that in~\Cref{fig:rnd,fig:rndls} for
the previous examples,
and the convergence rate seems not so sensitive to varying $\alpha$.
The right plot again shows that the spectral radius provides sharp estimation for
the convergence rate, justifying our theoretical analysis
in~\Cref{sec:local}.

In the second experiment, we consider rank-deficient
$D\in\bbR^{n\times k}$ with $k=50$ fixed, constructed as
$D = D_1 P^{\T}$ with
randomly generated $D_1\in\bbR^{n\times \rkd}$ and  $P\in\bbR^{k\times \rkd}$ that has orthonormal columns.
%being orthonormal.
%The $D$ so defined has a rank $r_D < k$.
We test $\rkd = 10,20,30,40$ and repeat the
experiment.
From the results in~\Cref{fig:ex2:singular},
the spectral radius still provides sharp estimation for
the convergence rate, justifying the analysis in~\Cref{sec:local}
for rank-deficient $D$.

%The testing matrix $D$ is defined in analogous to the previous experiment
%as $D = B [V_{r_D},0_{n\times (k-r_D)}]$
%with $V_{r_D}$ consisting of the eigenvectors corresponding to the $r_D$ smallest
%eigenvalues of the matrix pencil $A-\lambda B$.

}
\end{example}

\subsection{NEPv from \eqref{eq:max'}} \label{sec:egsb}
Optimization problem~\eqref{eq:max'} is another example
of~\eqref{eq:gopt},
and it has the unitarily invariant functions
\begin{equation}\label{eq:funex'}
	\phi(X) =
	\frac{\tr(X^{\T}AX)}{[\tr(X^{\T}BX)]^\theta}
	\qquad\text{and}\qquad
	\psi(X) =
	\frac{1}{[\tr(X^{\T}BX)]^\theta}.
\end{equation}
By~\Cref{thm:nepv}, its KKT condition is equivalent to  NEPv
\begin{subequations}\label{eq:nepva'}
\begin{equation}\label{eq:nepva'-1}
	H_{\theta}(X) X = X\Lambda,
\end{equation}
where the subscript $\theta$ indicates its dependence on
 parameter $\theta$, and
\begin{equation}\label{eq:hxex'}
	H_{\theta}(X) =
	2\psi(X)\cdot \Big[A - \theta\cdot \frac{\tr(X^{\T}AX)}{\tr(X^{\T}BX)} \cdot B\Big]
	- 2\theta \frac{\tr(X^{\T}D) \cdot \psi(X)}{\tr(X^{\T}BX)} \cdot B
	+ \psi(X)\cdot(DX^{\T} + XD^{\T}),
\end{equation}
\end{subequations}
which is obtained, using
\begin{equation}\label{eq:dfex'}
	\left\{
\begin{aligned}
	\frac {\partial \phi(X)}{\partial X} &= H_{\phi}(X) X \qquad\text{with}\qquad H_{\phi}(X)=
	%\frac{2}{[\tr(X^{\T}BX)]^\theta} \Big[A - \theta\frac{\tr(X^{\T}AX)}{\tr(X^{\T}BX)} \cdot B\Big],\\
	2\psi(X)\cdot \Big[A - \theta\cdot \frac{\tr(X^{\T}AX)}{\tr(X^{\T}BX)} \cdot B\Big],\\
	\frac {\partial \psi(X)}{\partial X} &= H_{\psi}(X) X \qquad\text{with}\qquad H_{\psi}(X) =
	-2\theta \cdot \frac{\psi(X)}{\tr(X^{\T}BX)} \cdot B.
	%-\psi(X)^3 \cdot B,
\end{aligned}\right.
\end{equation}
By varying $\theta$, we construct a variety of
NEPv~\eqref{eq:nepva'} for testing.

For the purpose of calculating the local rate of convergence, we also obtain
\begin{equation}\label{eq:dhex'}
\left\{
\begin{aligned}
	{\bf D}H_{\phi}(X)[E] &=
	-2\theta \frac{\tr(X^{\T}BE)}{\tr(X^{\T}BX)} \cdot H_{\phi,1}(X)
	-2\theta\frac{\tr(X^{\T} H_{\phi}(X) E)}{\tr(X^{\T}BX)}  \cdot B,
	\\
	{\bf D}H_{\psi}(X)[E] &= -2(\theta +1)\frac{\tr(X^{\T} H_{\psi}(X)
	E)}{\tr(X^{\T}BX)} B,
\end{aligned}\right.
\end{equation}
where $H_{\phi,1}$ is $H_{\phi}$ in~\eqref{eq:dfex'} with
$\theta=1$.
We can then obtain the corresponding aligned NEPv~\eqref{eq:alnepv}
and  linear operator $\scrL$ by~\eqref{eq:lz} in
the same way as in~\Cref{sec:egsa}.

% --------------------------------------------
% THIS IS A REPEAT OF EXAMPLE 7.2: with k=2
% --------------------------------------------

\begin{figure}[thbp]
\begin{center}
\includegraphics[width=0.45\textwidth]{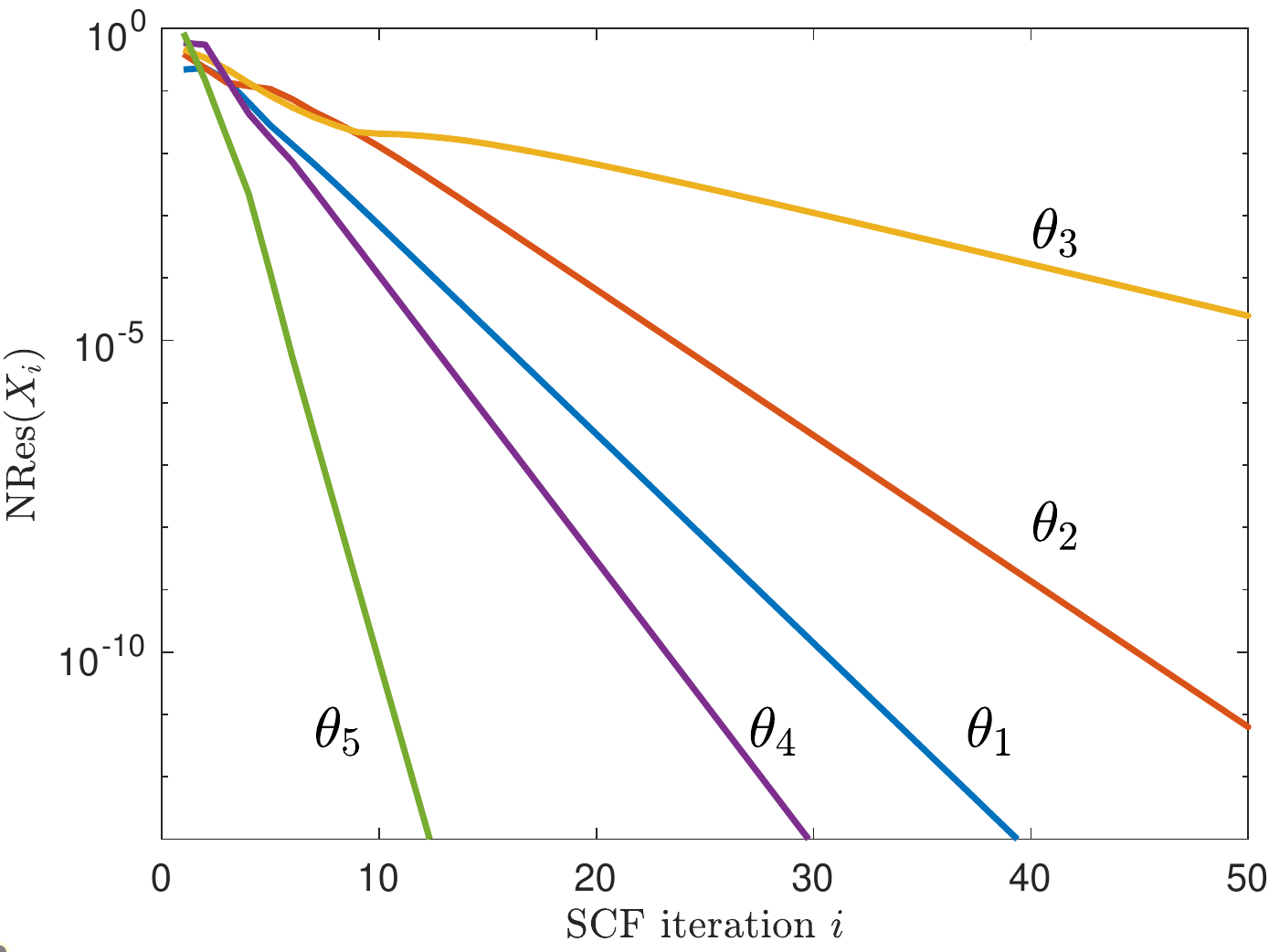}
\includegraphics[width=0.45\textwidth]{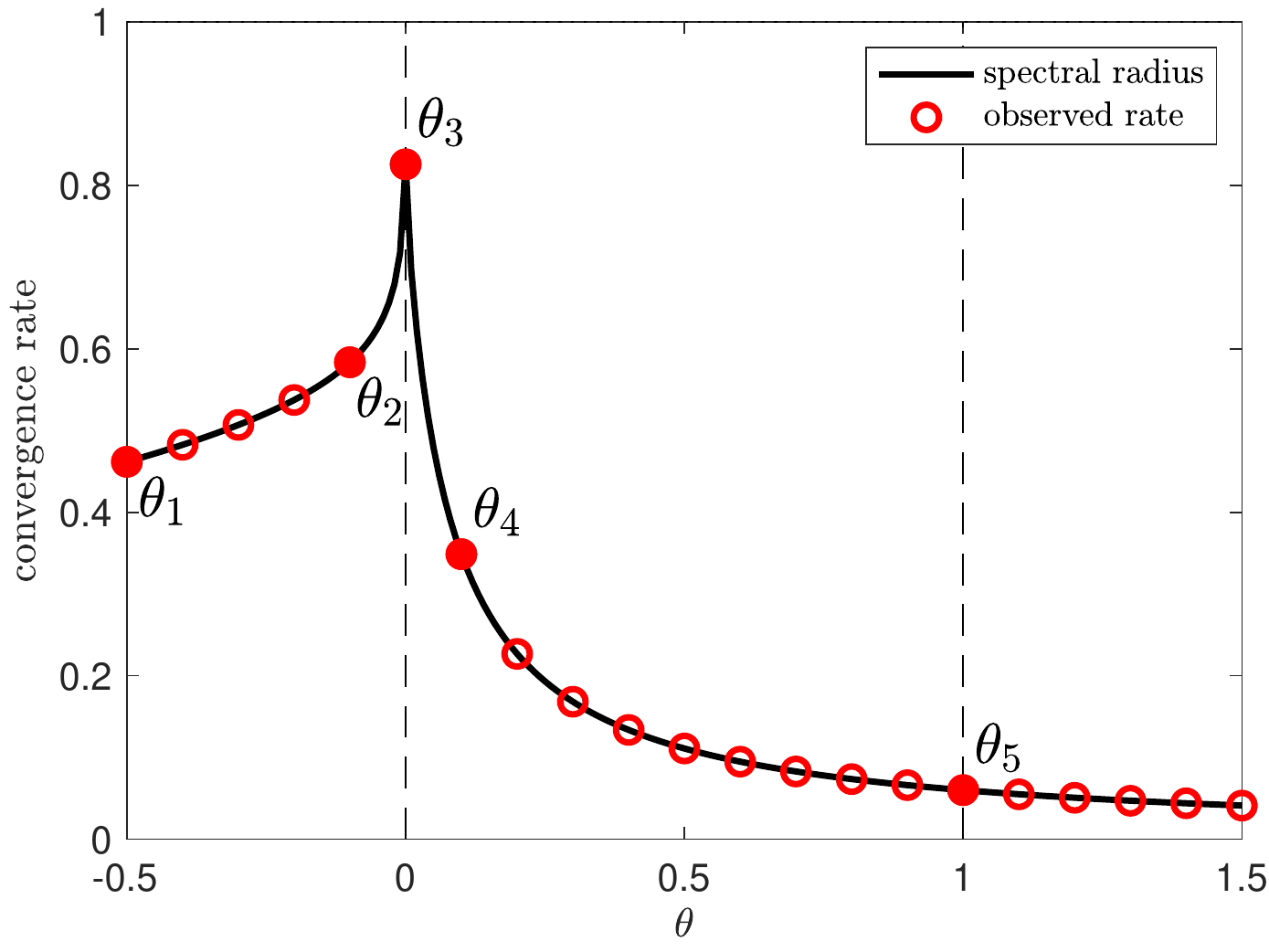}
\caption{
	\Cref{alg:scf} on Example~\ref{ex:1'}.
	{\em Left:\/} The iterative history for solving NEPv~\eqref{eq:nepva'} at a few sampled
	$\theta$ (correspondingly marked as $\bullet$ on the right plot).
	{\em Right:\/} The curve of spectral radius $\rho(\scrL)$ as a function of
	parameter $\theta\in[-0.5,1.5]$ (based on $200$ equally spaced $\theta$),
	and the observed rates of convergence (marked by $\bullet$ and $\circ$) at a number of values of $\theta$, including those
    sampled $\theta$ on the left plot.
}\label{fig:ex1':rndk1}
\end{center}
\end{figure}

\begin{figure}[th]
\begin{center}
\includegraphics[width=0.44\textwidth]{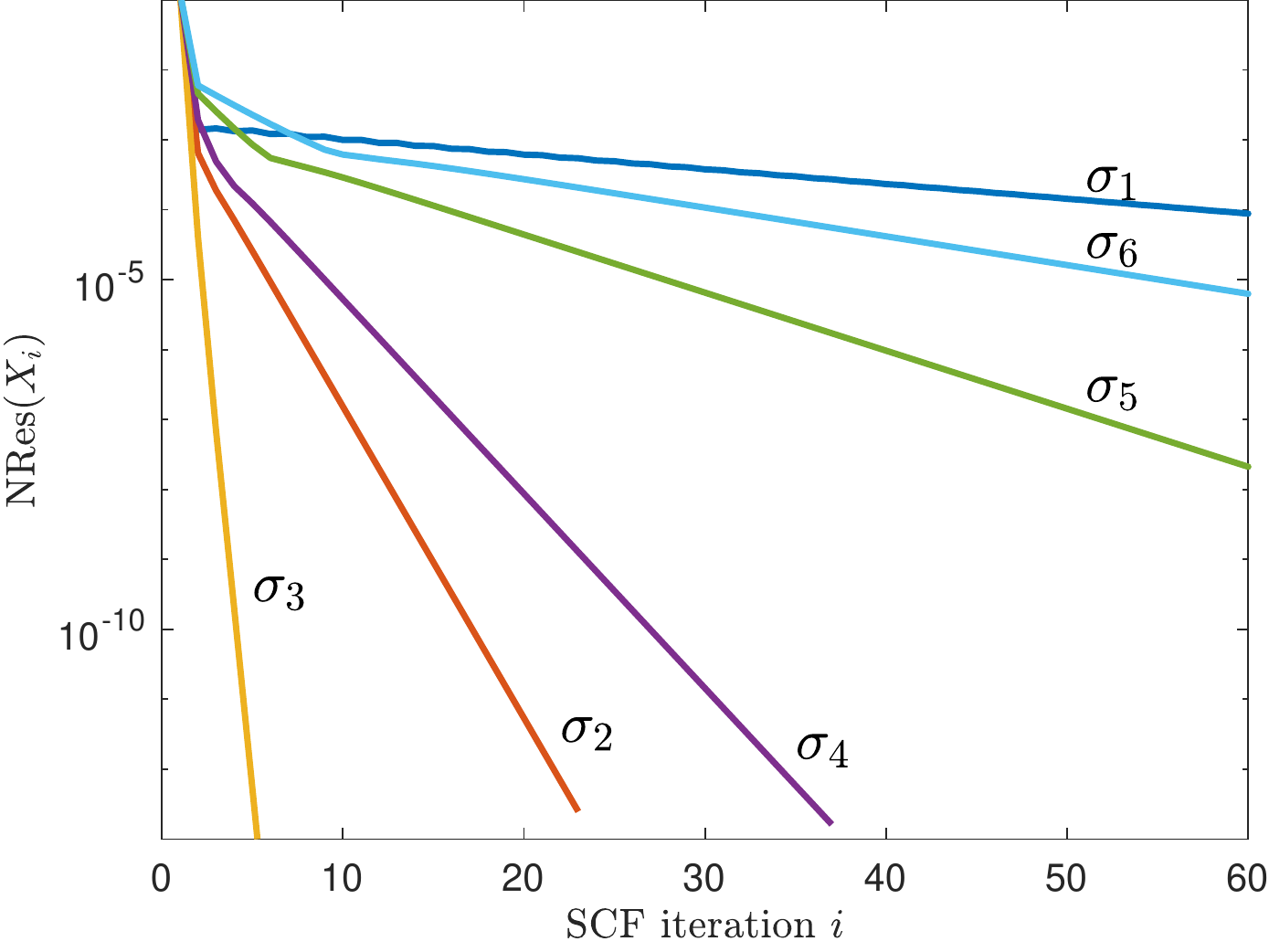}
\includegraphics[width=0.45\textwidth]{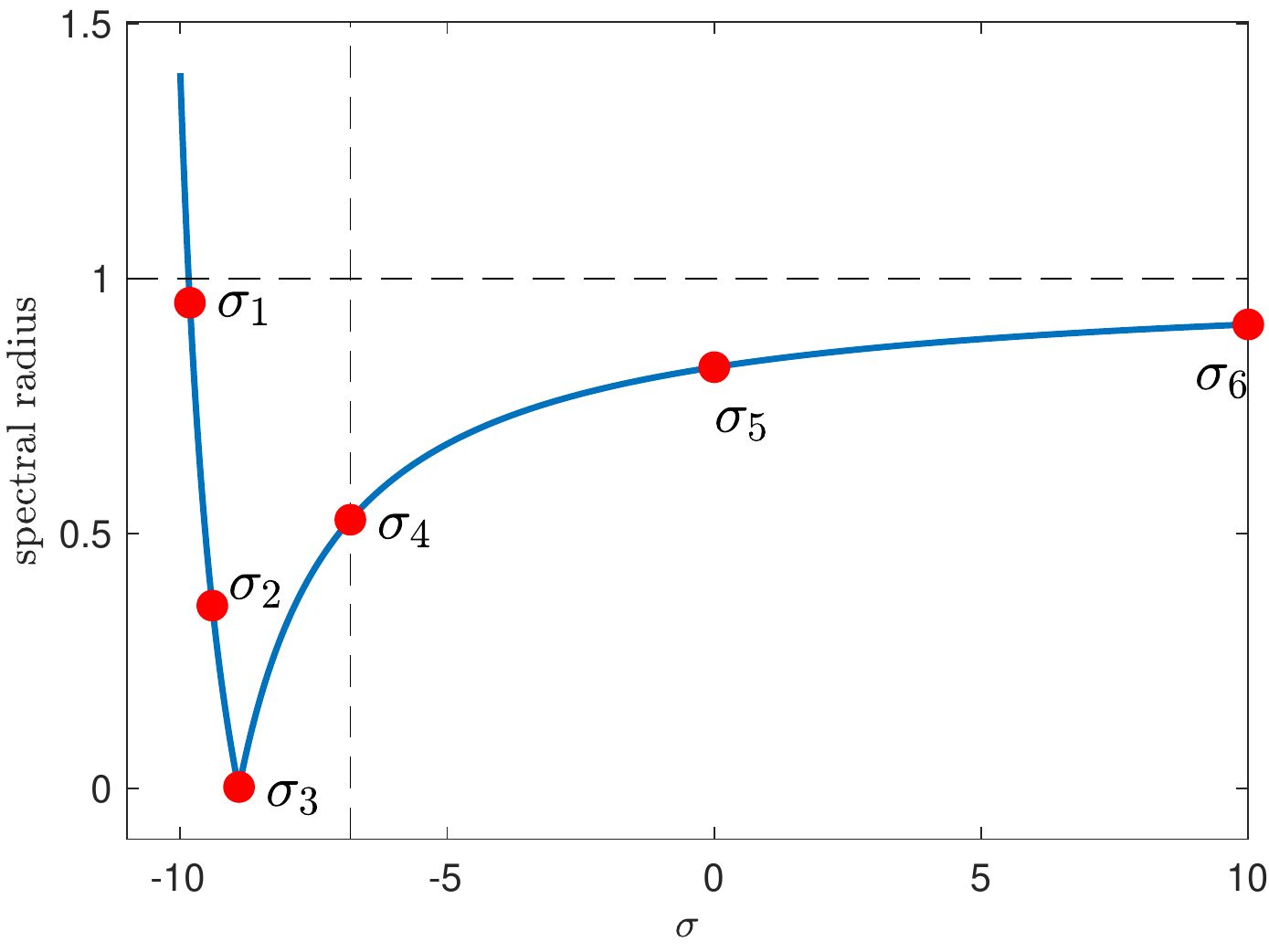}
\caption{
	\Cref{alg:LS-scf} on
	Example~\ref{ex:1'} with $\theta = 0$ (i.e., $\alpha_4$ in Figure~\ref{fig:ex1':rndk1}, at which \Cref{alg:scf} converges slowest).
	{\em Left:\/} The iterative history of \Cref{alg:LS-scf} (level-shifted SCF)
	with a few sampled level-shifts $\sigma$ (correspondingly marked as $\bullet$ on the right plot).
	{\em Right:\/} The curve of spectral radius $\rho(\scrL_\sigma)$ as a function of
	level-shift $\sigma$ and the observed rates of convergence at the sampled
	level-shifts $\sigma$. % (marked $\bullet$).
	The vertical dashed line corresponds to the theoretical lower bound
	$\sigma_L=-6.81$ by~\eqref{eq:sigma} (as $\sigma_4$ in the plots).  % lower bound $\sigma > -6.8156$.
	Negative shifts can accelerate SCF for the example.
}\label{fig:ex1':rndk1ls}
\end{center}
\end{figure}

\begin{example}\label{ex:1'}
{\rm %
Consider the random coefficient matrices $A,B\in\bbR^{n\times n}$ and
$D\in\bbR^{n\times k}$ used in~\Cref{ex:1},  where $n=3$ and $k=1$.
We first examine the convergence of~\Cref{alg:scf} on NEPv
\eqref{eq:nepva'} as $\theta$ varies in $[-0.5,1.5]$.
In~\Cref{fig:ex1':rndk1}, we see that~\Cref{alg:scf} is convergent for all
$\theta\in[0,1]$, consistent with the  global convergence
analysis in~\cite{wazl:2022}.
We can also see that the spectral radius
 captures the local convergence rate of the algorithm.
For example,
\[
	\mbox{at}\,\,\theta_4 = 0.1:\qquad
	\text{observed rate} \approx 0.348738\cdots,
	\quad
	\rho(\scrL) \approx 0.348739\cdots.
% theta = 0.1
% 0.348738785964281 % observed
% 0.348739107062505 % sprd
\]
For this  example, \Cref{alg:scf} runs faster as $\theta$ increases from $0$.
But such a rapid convergence is not guaranteed
for general NEPv~\eqref{eq:nepva'},
especially as $\theta > 1$; see~\Cref{ex:5} below.

Next, we consider the level-shifted SCF (\Cref{alg:LS-scf})
for NEPv~\eqref{eq:nepva'} at  $\theta=0$,
at which the curve of spectral radius in~\Cref{fig:ex1':rndk1} peaks.
Recall that the NEPv with $\theta=0$ arises in solving
the unbalanced orthogonal Procrustes problems; see, e.g.,~\cite{Zhang:2020}.
Since~\Cref{alg:scf} is globally convergent in this case,
there is no need to use level-shifting to fix the divergence issue of SCF.
Nevertheless, the level-shift still helps to speed up the convergence of the algorithm.
From the curve of spectral radius in \Cref{fig:ex1':rndk1ls}, the optimal
level-shift is achieved at
$\sigma_*\approx -8.91$ (as $\sigma_3$ in \Cref{fig:ex1':rndk1ls}),
at which
$\rho(\scrL_\sigma) \approx 0.7\times 10^{-3}$,
indicating significant acceleration to the algorithm.
The theoretical lower bound~\eqref{eq:sigma} predicts that
the level-shifted SCF is locally convergent for any $\sigma\gtrsim\sigma_L=-6.81$ (as $\sigma_4$ in \Cref{fig:ex1':rndk1ls}), overestimating the observed one.

An interesting observation here is that the optimal level-shift is a \emph{negative} number.
A negative $\sigma$ would reduce %(rather than increase)
the gap between the $k$th and $(k+1)$-st eigenvalues (see~\eqref{eq:eigdgsx}), and
it is remarkable that here a negative shift can even greatly accelerate SCF.
The negative level-shift for SCF has been briefly mentioned
in~\cite{Bai:2022}, but its benefits have not been fully understood.
\Cref{fig:ex1':rndk1ls} provides a concrete example to demonstrate
this intriguing  behavior of negative level-shifting.

} %

\end{example}

\begin{figure}[thbp]
\begin{center}
\includegraphics[width=0.45\textwidth]{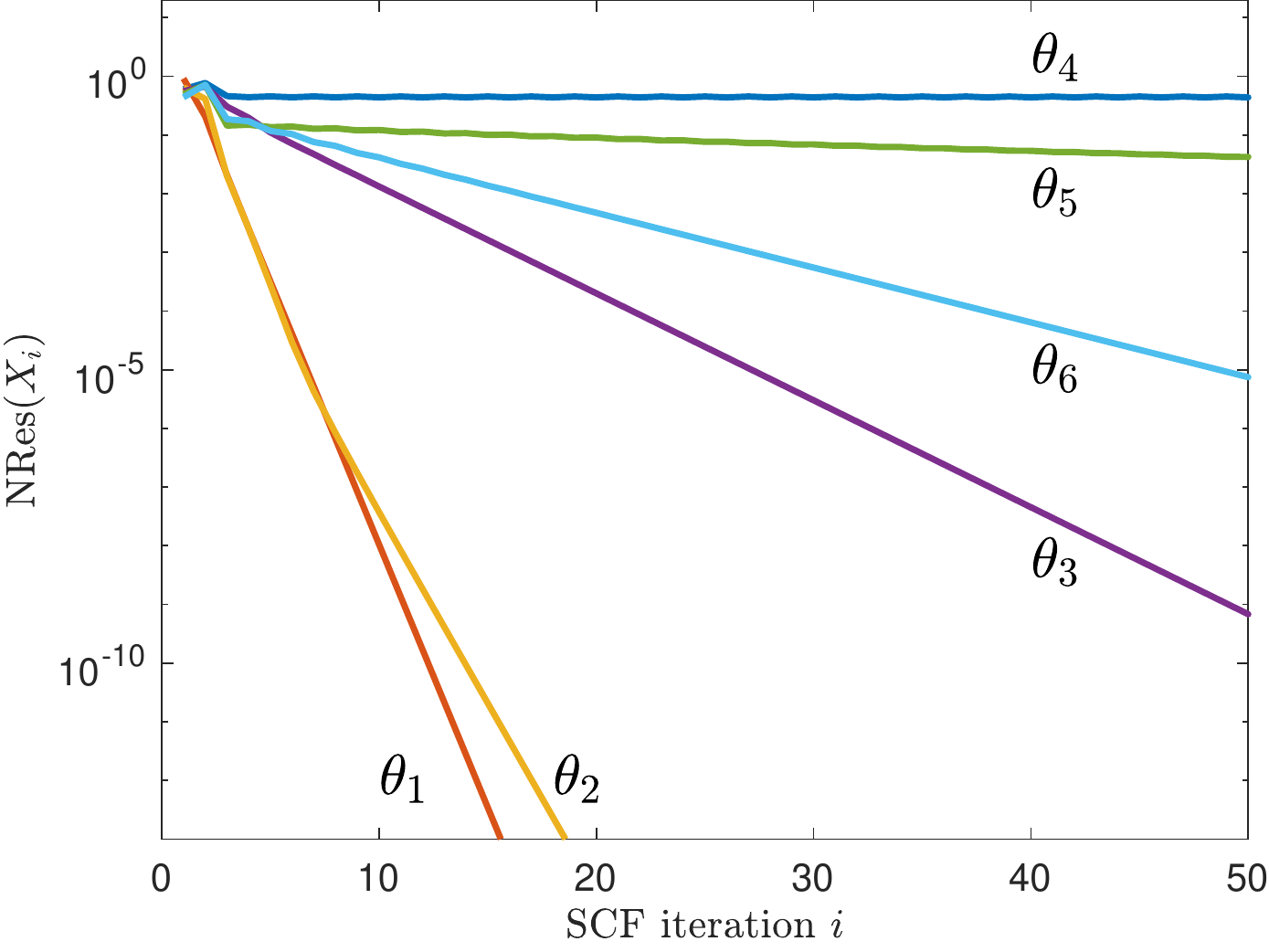}
\includegraphics[width=0.45\textwidth]{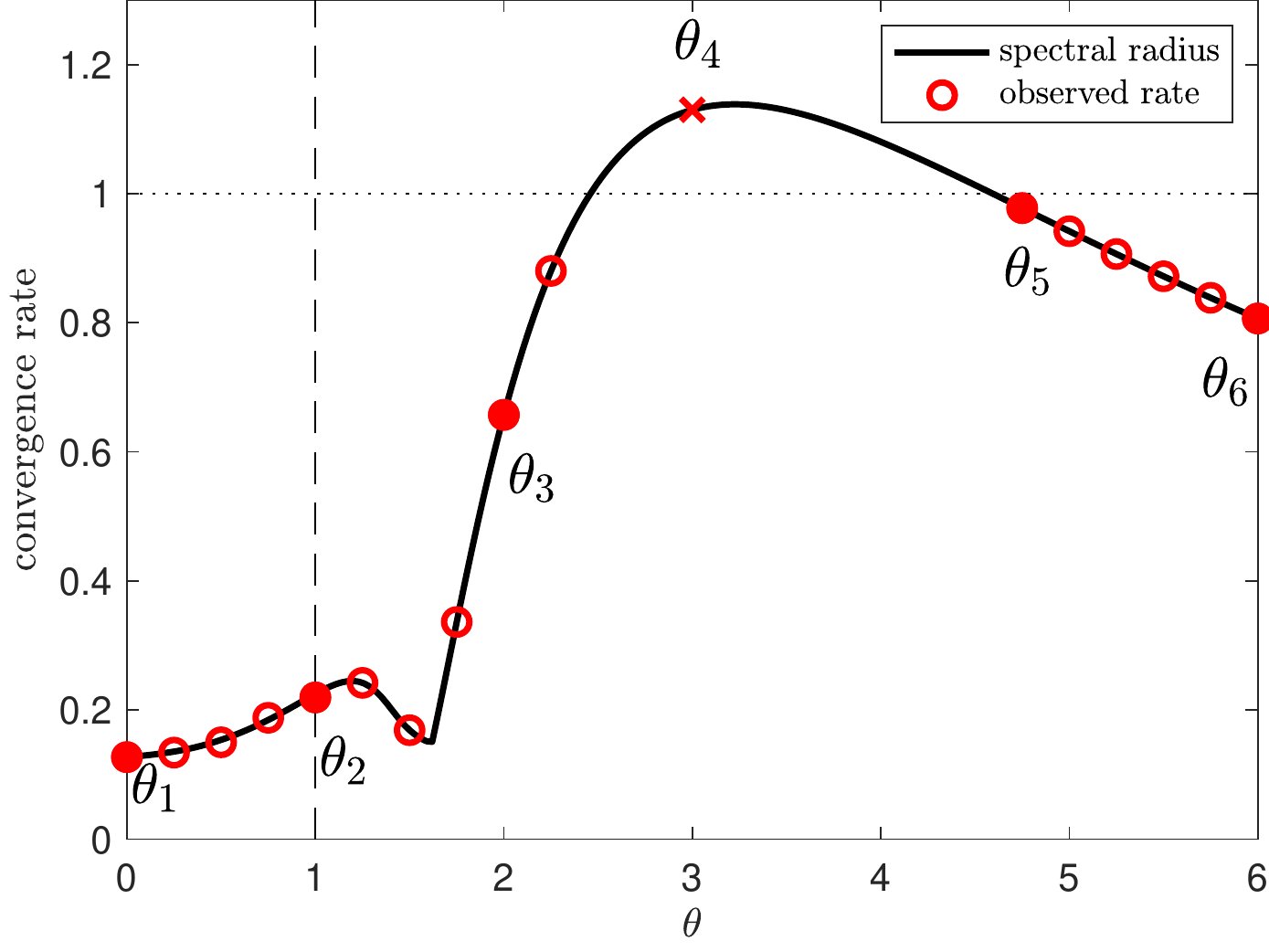}
\caption{
	\Cref{alg:scf} on Example~\ref{ex:5}.
	{\em Left:\/} The iterative history for solving NEPv~\eqref{eq:nepva'} at a few
	sampled $\theta$ (correspondingly marked as as $\bullet$ and $\times$ on the right plot).
	{\em Right:\/} the curve of spectral radius $\rho(\scrL)$
	as a function of parameter $\theta\in[0,6]$ (based on $200$ equally
	spaced $\theta$),
	and the observed rates of convergence (marked by $\bullet$ and $\circ$) at a number of values of
	$\theta$, including those
    sampled $\theta$ on the left plot,
	and `$\times$' indicates that SCF is divergent.
}\label{fig:rndk2'}
\end{center}
\end{figure}

\begin{figure}[thbp]
\begin{center}
\includegraphics[width=0.45\textwidth]{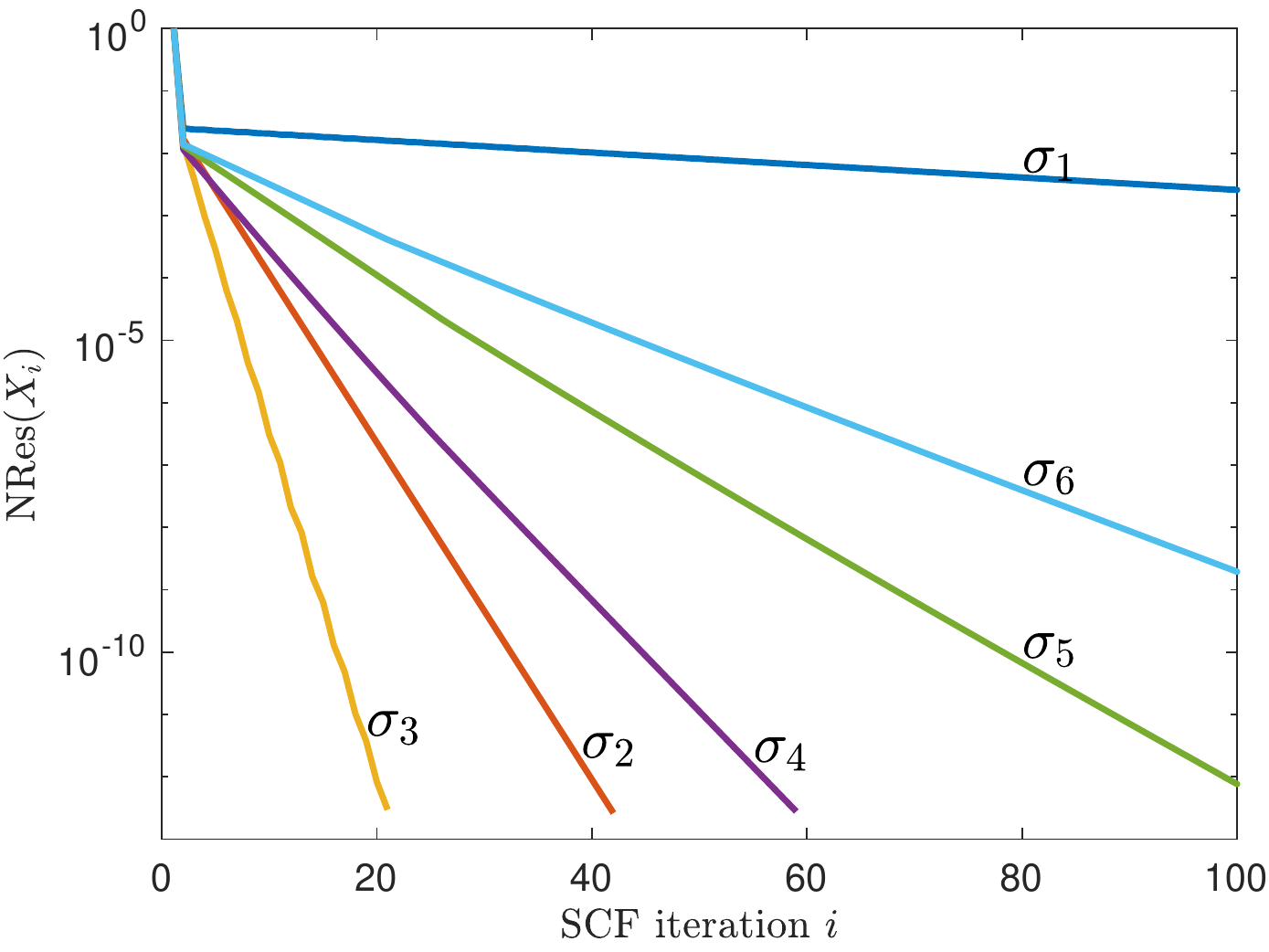}
\includegraphics[width=0.46\textwidth]{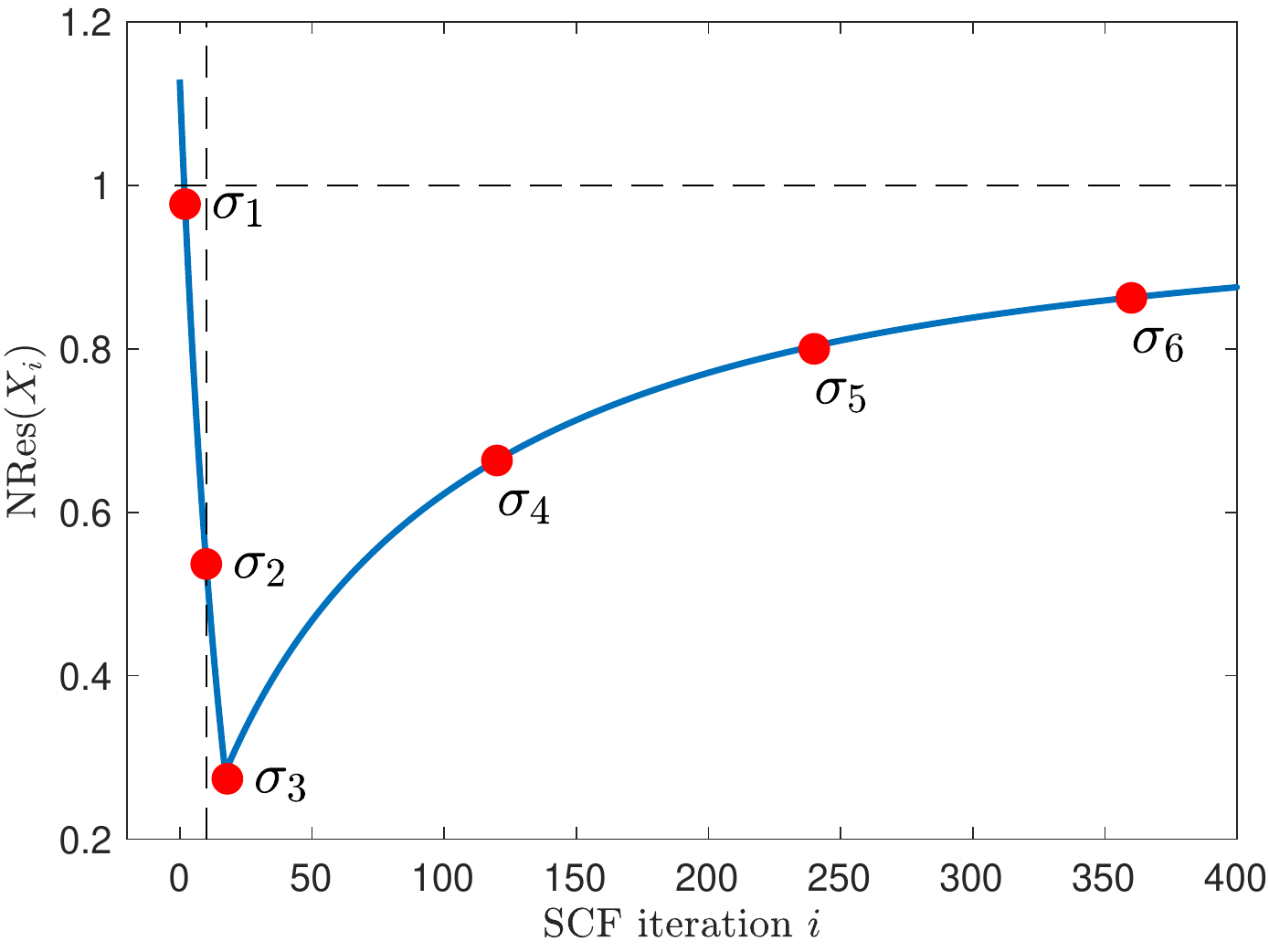}
\caption{
	\Cref{alg:LS-scf} on
	Example~\ref{ex:5} with $\theta = 3.0$ (i.e., $\theta_4$ in Figure~\ref{fig:rndk2'}, at which \Cref{alg:scf} diverges).
	{\em Left:\/} The iterative history of \Cref{alg:LS-scf} (level-shifted SCF)
	with a few sampled level-shifts $\sigma$ (correspondingly marked as $\bullet$ on the right plot).
	{\em Right:\/} the curve of spectral radius $\rho(\scrL_\sigma)$ as a function of
	level-shift $\sigma$ and the observed rates of convergence at the sampled
	level-shifts $\sigma$. % (marked $\bullet$).
	The vertical dashed line corresponds to the theoretical lower bound
	$\sigma_L=10.02$ given by~\eqref{eq:sigma} (as $\sigma_2$ in  the plots).
}\label{fig:rndk2ls}
\end{center}
\end{figure}

\begin{example}\label{ex:5}
{\rm
This example demonstrates the potential divergence issue
of~\Cref{alg:scf} for NEPv~\eqref{eq:nepva'}
with $\theta > 1$.
We consider the following randomly generated matrices
\[
A =
\left[\begin{array}{rrr}
   	1.145&   -0.095&    0.514\\
   -0.095&    0.838&    1.022\\
    0.514&    1.022&   -1.223
\end{array}\right]
,\
B =
\left[\begin{array}{rrr}
    0.582&   -0.037&    0.025\\
   -0.037&    0.183&    0.043\\
    0.025&    0.043&    0.239
\end{array}\right]
,\
D =
\left[\begin{array}{rr}
	0.760&    0.258\\
	0.011&    0.774\\
	0.180&    0.520
\end{array}\right].
\]
\Cref{fig:rndk2'} reports
the convergence of~\Cref{alg:scf} on NEPv~\eqref{eq:nepva'} as $\theta$
varies in $[0,6]$.
From the curve of spectral radius, it can be seen that \Cref{alg:scf} converges
for all $\theta\in[0,1]$,
and the spectral radius captures very well the observed convergence
rates, e.g.,
\[
	\mbox{at}\,\,\theta_5 = 4.75:\qquad
	\text{observed rate} \approx 0.977615\cdots,
	\quad
	\rho(\scrL) \approx 0.977613\cdots.
% theta = 4.8
% 9.704365828035505e-01 % observed
% 9.704361752050634e-01 % sprd
% theta = 4.75
% 9.776154721478407e-01 % observed
% 9.776135440814463e-01 % sprd
\]
The curve of spectral radius indicates that~\Cref{alg:scf}
fails to converge for $\theta$ approximately in the interval $[2.46, 4.59]$.
For example, at $\theta = 3.0$,
the normalized residual $\mbox{NRes}(X_i)$ by~\Cref{alg:scf}
 oscillate between two close numbers
$0.453\cdots$ and $0.437\cdots$ after about $20$ iterations.
%theta = 3.0, oscillating, NRes 0.332116541570830   0.341197861263449

%%
To find the solution of the NEPv when $\rho(\scrL) > 1$, we
apply the level-shifted SCF (\Cref{alg:LS-scf}) with $\sigma = 40$.
The effectiveness of level-shifting is also demonstrated in~\Cref{fig:rndk2ls}
for solving  NEPv \eqref{eq:nepva'} with $\theta=3.0$ (i.e., $\theta_4$ in Figure~\ref{fig:rndk2'}), at which \Cref{alg:scf} diverges.
The optimal level-shift is
$\sigma_* \approx 17.21$ (as $\sigma_3$ in the plots) with $\rho(\mathcal L_\sigma) \approx 0.282$.
The theoretical lower bound of the level-shift $\sigma$
by~\eqref{eq:sigma} is $\sigma_L=10.02$ (as $\sigma_2$ in the plots).
}
\end{example}

\begin{figure}[thbp]
\begin{center}
\includegraphics[width=0.45\textwidth]{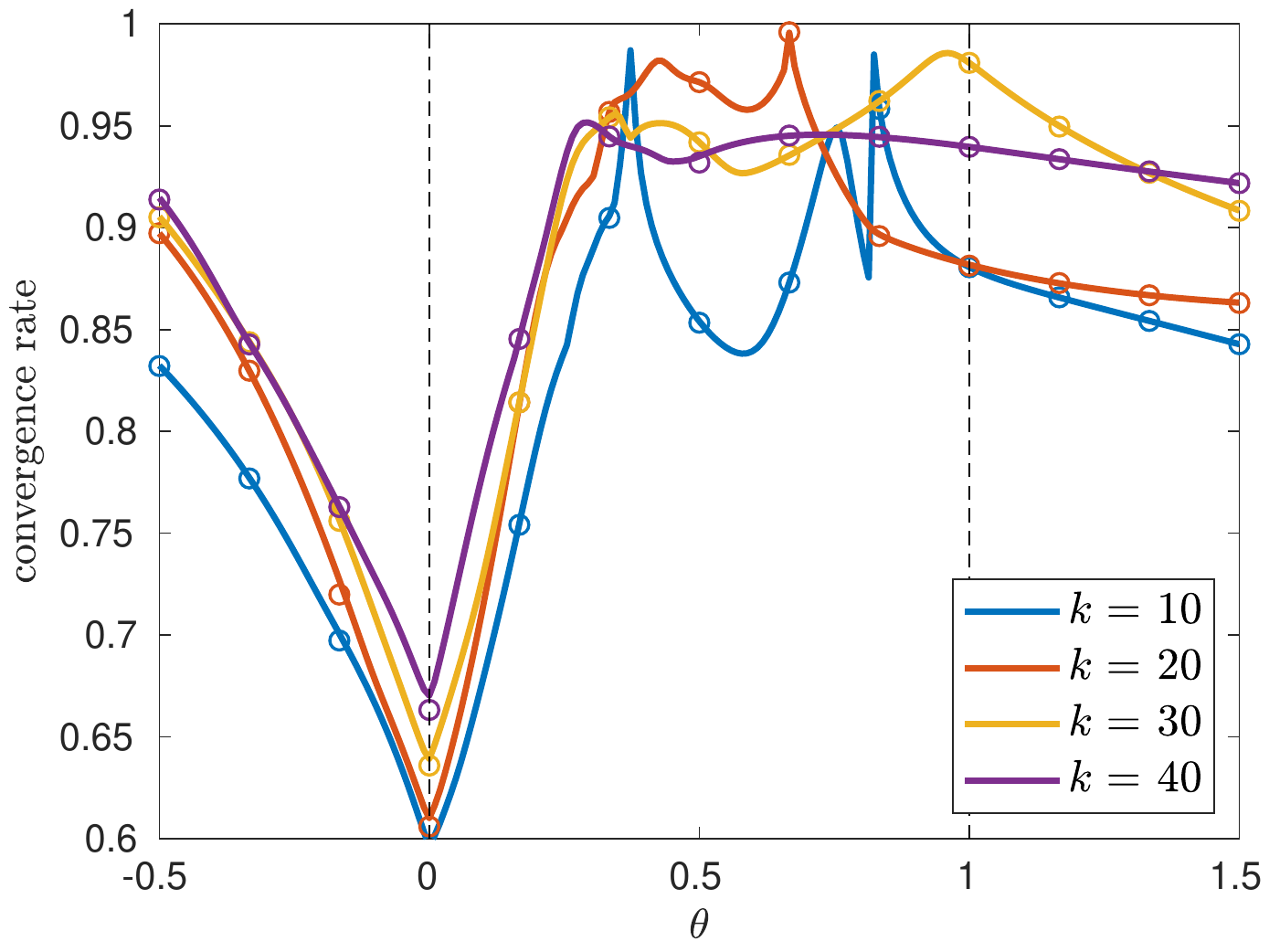}
\includegraphics[width=0.45\textwidth]{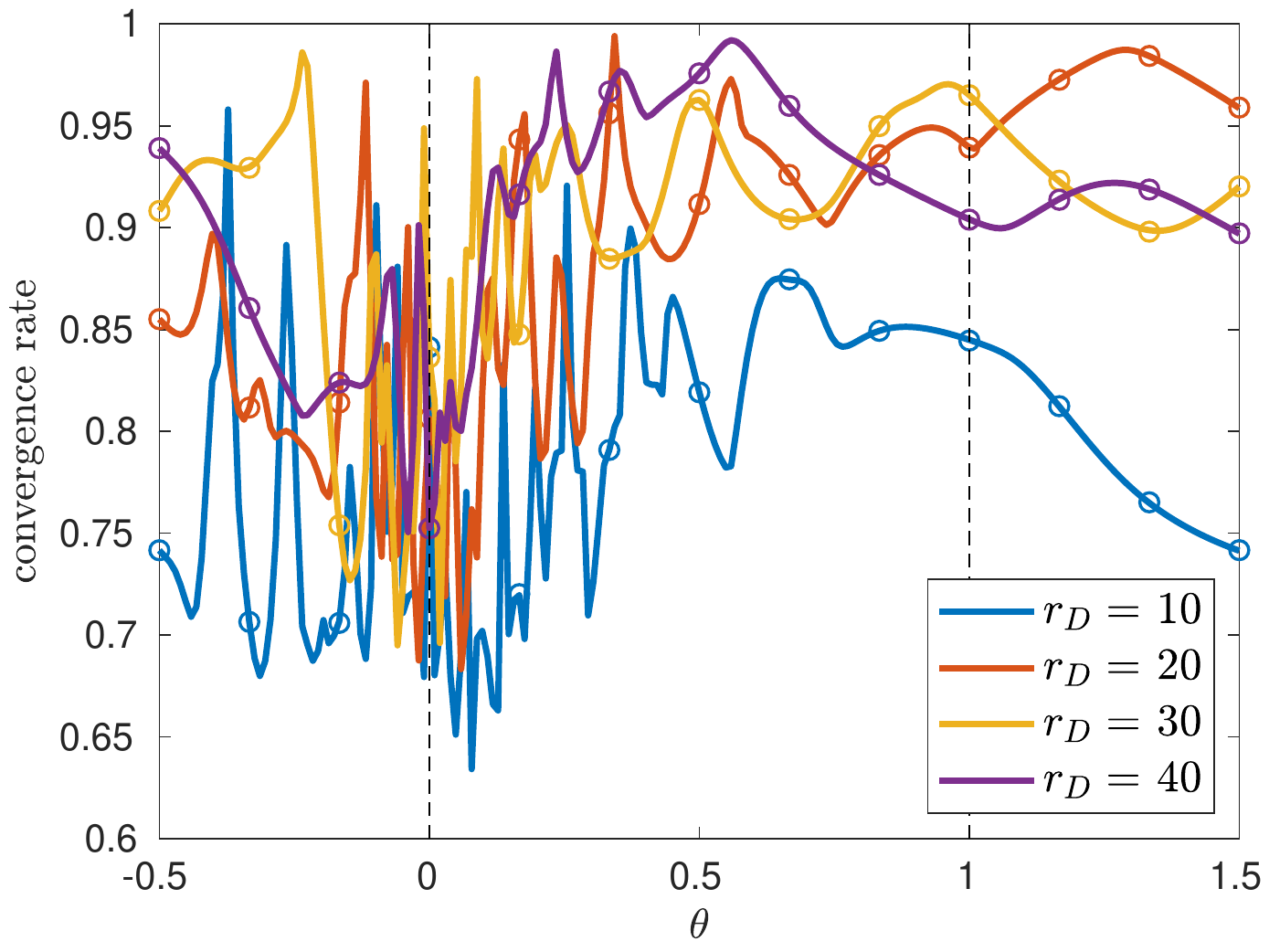}
\caption{\Cref{alg:scf} on Example~\ref{ex:6}:
	The curves of spectral radius $\rho(\scrL)$ as a function of
	$\theta\in[-0.5,1.5]$ (based on $200$ equally spaced $\theta$)
	and the observed rates of convergence of~\Cref{alg:scf} at a few sampled $\theta$
	(marked by $\circ$).
	{\em Left:\/} Full-rank $D\in \bbR^{n\times k}$.
	{\em Right:\/} Rank-deficient $D\in \bbR^{n\times k}$ with $k=50$ and rank $\rkd:=\mrank(D)<k$.
}\label{fig:rank'}
\end{center}
\end{figure}

\begin{example}\label{ex:6}
{\rm %
This example examines the convergence
of~\Cref{alg:scf} on NEPv~\eqref{eq:nepva'} for more general cases of
$D\in\bbR^{n\times k}$, as in the previous~\Cref{ex:3}.
We use the same testing matrices as in~\Cref{ex:3}.
%,
%where $A, B\in\bbR^{n\times n}$ with $n=200$ are fixed,
%and $D\in\bbR^{n\times k}$ with various $k\leq n$  are randomly generated.
Recall that a set of four full-rank $D\in\bbR^{n\times k}$, with $k=10,20,30,40$, and
the other set of four rank-deficient $D\in\bbR^{n\times 50}$ with
$\rkd:=\mrank(D)=
10,20,30,40$, are tested.
For each set of testing matrices we apply~\Cref{alg:scf} on NEPv~\eqref{eq:nepva'}
with $\theta\in[-0.5,1.5]$ and generate the curves of spectral radius
$\rho(\scrL)$ as a function in~$\theta$ as shown in~\Cref{fig:rank'}.
It is observed that the spectral radius remains less than $1$ in all
testing cases, and thus \Cref{alg:scf} is locally convergent.
The curves of spectral radius demonstrate very different patterns
for the full-rank and rank-deficient cases:
They appear % more smooth for a full-rank $D$ but
to be more sensitive to  $\theta$ for a rank-deficient $D$.
Despite their wild oscillations,
the curves of spectral radius  still match very well with our theoretical convergence
rate at all sampled $\theta$,
which further confirms our convergence analysis in~\Cref{sec:local}.
%for both the full-rank and rank-deficient $D$.
%The patterns observed do reflect the nature of the NEPv.
% is not due to the numerical error
%in computing $\rho(\scrL)$ for rank-deficient $D$,
%as we can see the
}
\end{example}

\section{Concluding Remarks}\label{sec:conclusion}
We investigated a class of NEPv~\eqref{eq:nepv2} as
%\marginpar{\tiny to work on (old)}
arising from solving optimization problem~\eqref{eq:gopt}
on the Stiefel manifold:
\begin{equation}\tag{\ref{eq:gopt}}
	\max_{X\in\bbO^{n\times k}} f(X)
 	\qquad\text{with}\qquad f(X):=\phi(X) + \psi(X)\cdot \tr(X^{\T}D),
\end{equation}
whose objective function $f(\,\cdot\,)$
is not invariant upon substitution
$X\leftarrow XQ$ with $Q\in\bbO^{k\times k}$.
Consequently, the resulting NEPv does not have
the unitary invariance property,
unlike those commonly studied as in \cite{Bai:2022}.
We have shown that any global optimizer $X_*$ of~\eqref{eq:gopt} is a \dproper~eigenbasis matrix of
NEPv~\eqref{eq:nepv2}, i.e., satisfying
$$
X_*^{\T}D\succeq 0 \quad\mbox{and}\quad \mrank(X_*^{\T}D)=\mrank(D),
$$
and that for any $X$ that has orthonormal columns and is close to $X_*$,
the NEPv can be reformulated to another NEPv, called the aligned NEPv,
that is unitarily invariant.
This novel reformulation essentially reduces
the local convergence analysis of the SCF-type
iteration in~\Cref{alg:scf} for NEPv~\eqref{eq:nepv2}
to the case that had been studied in \cite{Bai:2022}
for general unitarily invariant NEPv, once some technicalities are taken care of.
We established closed-form local convergence rates for~\Cref{alg:scf,alg:LS-scf} and
built a theoretical foundation for the application of a level-shifting scheme.
Our theoretical analysis has been confirmed by extensive numerical experiments.

Throughout this paper, our presentation is restricted to the real number field $\bbR$.
This restriction is more for simplicity and clarity than
the capability of our techniques to deal with a problem similar
to \eqref{eq:gopt} but in the complex number field $\bbC$.
In fact, our approach can be extended to handle,
more generally,
the following optimization problem on the complex Stiefel manifold:
\begin{equation}\label{eq:optgeneral'}
	\max_{X\in\bbC^{n\times k},\, X^{\HH}X=I_k} f(X)
	\qquad\text{with}\qquad f(X):=\phi(X) + \psi(X)\cdot
	\eta(\mbox{Re}(\,\tr(X^{\HH}D)\,)),
\end{equation}
where the superscript {\sc h} takes the complex conjugate transpose of
a matrix, % (or vector),
$\mbox{Re}(\,\cdot\,)$ takes the real part of a complex number,
$\phi$ and $\psi$ are real-valued functions in $X\in\bbC^{n\times k}$,
continuously differentiable with respect to the real and imaginary parts of
$X$ and satisfying the unitary invariance property:
$$
\phi(XQ)=\phi(X), \quad
\psi(XQ)=\psi(X)\quad
\mbox{for $X\in\bbC^{n\times k},\, Q\in\bbC^{k\times k}$ such that $Q^{\HH}Q=I_k$},
$$
$\psi(\,\cdot\,)>0$ is a positive function,
and $\eta:\bbR\to \bbR$ is a differentiable and monotonically increasing
function, i.e., $\eta'(t)>0$ for all $t\in\bbR$.
Due to that $X$ is a complex matrix variable, we will need to use Wirtinger
derivatives~\cite{Brandwood:1983, Kreutz:2005}
to establish the KKT condition of \eqref{eq:optgeneral'} and its
corresponding NEPv.
On the other hand, because of the monotonicity in~$\eta$,
we can still apply the same alignment operation as discussed
in~\Cref{sec:basis}
and establish analogous necessary conditions for the global maximizers
of~\eqref{eq:optgeneral'} as in~\Cref{thm:global}.
A complete treatment is left to future work.

\appendix

\section{Differentiability and Polar Decomposition}
\iffalse
We denote by $\D F (X)$ the Fr\'echet derivative
of a (Fr\'echet) differentiable function $F:\bbR^{m\times n}\to \bbR^{p\times q}$
evaluated at $X\in\bbR^{m\times n}$.
Recall that $\D F (X)[\cdot]: \bbR^{m\times n}\to \bbR^{p\times q}$ is a linear
operator satisfying
\begin{equation}\label{eq:frechet}
	F(X+E) = F(X) +  \D F (X)[E] + o(\|E\|).
\end{equation}
This also leads to
\begin{equation}\label{eq:notation}
	\D F (X)[E]  = \frac{d}{dt}
	 \Big(F(\,X+tE\,)\Big)\Big|_{t=0},
\end{equation}
namely, $\D F (X)[E] $ represents the directional derivative of $F(X)$ at $X$ in
the direction of $E$.
\fi
Recall  definition \eqref{eq:FD-defn} for the Fr\'echet derivative
of a (Fr\'echet) differentiable function
$F:\bbR^{m\times n}\to \bbR^{p\times q}$ at $X$ along direction $Y$.
In this section, we derive the expressions of the Fr\'echet derivatives
of the polar factors of a matrix having full column rank.

The polar decomposition of a matrix $Z\in\bbR^{m\times p}$
having full column rank, i.e., $\mrank(Z)=p$, is given by
\begin{subequations}\label{eq:PD-fullrank}
\begin{equation}\label{eq:PD-fullrank-1}
	Z = Q\,M,
\end{equation}
where $Q\in\bbO^{m\times p}$ is the orthonormal polar factor and
$M\in\bbR^{p\times p}$ %(symmetric positive definite)
is the symmetric polar factor that is also positive definite.
The polar factors are unique \cite{Higham:2008,li:2014HLA} and
differentiable with respect to the matrix variable $Z$. In fact, they
  can be expressed explicitly as
\begin{equation}\label{eq:mq0}
M= \left(Z^{\T}Z\right)^{1/2}
\qquad\text{and}\qquad
Q= Z\,\left(Z^{\T}Z\right)^{-1/2},
\end{equation}
\end{subequations}
where $(\,\cdot\,)^{1/2}$ is  the positive semi-definite square root
of a positive semi-definite matrix.
The Fr\'echet derivatives of $M$ and $Q$ with respect to $Z$ are detailed in
\Cref{lem:polard} below.
These results are not new, but we will provide a quick derivation for self-containedness.

\begin{lemma}\label{lem:polard}
	Let $Z\in\bbR^{m\times p}$ have full column rank.
	Its polar factors $M$ and $Q$, as given in~\eqref{eq:mq0},
	are Fr\'echet differentiable in $Z$ along direction $Y\in\bbR^{m\times p}$,
	with their derivatives given  by
	\begin{equation}\label{eq:polard}
			\D\!M(Z)[Y] = L
			\qquad\text{and}\qquad
			\D Q(Z)[Y] = \left(\,Y - Q L\,\right)
			\cdot M^{-1},
	\end{equation}
	where $L\in\bbR^{m\times p}$ is the solution to the Lyapunov equation
	\begin{equation}\label{eq:sylvester}
		M\cdot L + L\cdot M =  (Y^{\T}Z + Z^{\T}Y).
	\end{equation}
\end{lemma}

\begin{proof}
	%The differentiability of $M$ and $Q$ has been established in, e.g.,~\cite[Theorem~3.1]{Chaitin:2000}.
	The differentiability of the factors in~\eqref{eq:mq0}
	follows from the
	differentiability of the matrix square root function $(\,\cdot\,)^{1/2}$
	\cite[Chapter~3]{Higham:2008}.
	To obtain closed-form formulas for the derivatives,
	let $M+\delta M(t) = \left[(Z+t Y)^{\T}(Z+t Y)\right]^{1/2}$
	where $t\in\bbR$
	is assumed sufficiently small.
	Then $(M + \delta M(t))^2 = (Z+t Y)^{\T}(Z+t Y)$,
	expanding which to get
	\[
		\delta M(t)\cdot M + M\cdot \delta M(t)  = t(Z^{\T}Y + Y^{\T}Z)
                + \mathcal O(t^2).
	\]
	Hence, $\delta M(t) = t L + \mathcal O(t^2) $, yielding the first formula in \eqref{eq:polard}.
	Now, it follows from $Q+\delta Q(t) = (Z+t Y)(M+\delta M(t))^{-1}$ that
	\begin{align*}
		\delta Q(t) &=
		t YM^{-1} - Z M^{-1}\cdot\delta M(t)\cdot M^{-1} + \mathcal O(t^2) \\
		&= t (YM^{-1} - Q\cdot L \cdot M^{-1}) + \mathcal O(t^2),
 \end{align*}
	yielding the second formula in \eqref{eq:polard}.
\end{proof}

\section{Aligned NEPv via Bi-level Maximization}\label{sec:bimax}
We will show that the aligned NEPv~\eqref{eq:alnepv}
can also be derived from the KKT condition of another maximization problem that
is equivalent to~\eqref{eq:gopt},
as we mentioned at the end of~\Cref{sec:aligned}.
This approach provides a seemingly more direct
derivation of the aligned NEPv~\eqref{eq:alnepv} than what is presented in~\Cref{sec:aligned},
but it does not seem to yield
%its derivation relies on
\rankd~condition \eqref{cond:rank} at optimality,
a cornerstone of our local convergence analysis
of the SCF-type iteration in \Cref{alg:scf} for solving
NEPv~\eqref{eq:nepv2}, whose rationale comes from
analyzing the formulation of NEPv~\eqref{eq:nepv2}.
%We point out that the \rankd~condition was obtained
%from  whereas it is not quite clear
%whether the condition can also be derived from the new bi-level optimization
%formulation, however.

%To obtain the associated maximization problem,
We can state
optimization~\eqref{eq:gopt} equivalently as
\begin{equation}\label{eq:twolevel0}
	\max_{ X\in \bbO^{n\times k}, Q\in\bbO^{k\times k}}
		\phi(XQ) + \psi( XQ)\cdot \tr(Q^{\T} X^{\T}D),
\end{equation}
in the sense that any maximizer of one optimization problem
will lead to a maximizer of the other.
Due to that $\phi$ and $\psi$ are unitary invariant,
we can drop $Q$ from the argument of both $\phi(XQ)$ and $\psi(XQ)$,
and write~\eqref{eq:twolevel0} as a \emph{bi-level optimization}:
\begin{equation}\label{eq:twolevel}
	\max_{X\in \bbO^{n\times k}}
	\left[ \phi(X) + \psi(X)\cdot
		\left(\max_{Q\in\bbO^{k\times k}} \tr(Q^{\T} X^{\T}D)\right)\right],
\end{equation}
where we have used $\psi(X) >0$.
Recall that the inner optimization has been considered in~\Cref{sec:basis},
and its solution can be given in terms of the polar decomposition~\eqref{eq:polar}.
Consequently, maximization problem~\eqref{eq:twolevel} is reduced to
\begin{equation}\label{eq:gopt2}
	\max_{X\in\bbO^{n\times k}} g(X)
	\qquad\text{with}\qquad
	g(X) := \phi(X) + \psi(X)\cdot \tr(M),
\end{equation}
where $M\succeq 0 $ is the positive semidefinite polar factor of
$X^{\T}D = QM$ in~\eqref{eq:polar}.
The two maximization problems \eqref{eq:gopt2} and \eqref{eq:gopt} are
clearly equivalent: $X_*$ is a global maximizer of~\eqref{eq:gopt2}
if and only if any $\widetilde X_*\in\al{X_*}$ is a global maximizer of  original~\eqref{eq:gopt}.

To derive the first-order optimality condition of \eqref{eq:gopt2},
we define its Lagrangian function as
\begin{equation}\label{eq:laggx}
	\cL(X,\Gamma) = \phi(X) + \psi(X)\cdot \tr(M)
         - \frac{1}{2} \tr\big(\Gamma^{\T}[X^{\T}X - I_k]\big),
\end{equation}
where the symmetric $\Gamma\in\bbR^{k\times k}$ is the matrix
of multipliers.
%and we suppressed the dependence of $\cL$ on $\Gamma$ for
%convenience of notation.
Recall that $M$ depends on $X$.
Under the \rankd~condition $\mrank(X^{\T}D) = \mrank(D)$,
$M\equiv M(X)$ is differentiable with respect to $X$ by \Cref{lem:polar}.
Then the function $\gamma(X):=\tr(M)$ has its derivative along direction $E$
given by
\[
	\D \gamma(X)[E] = \tr\left(\D\!M(X)[E]\right) = \tr\left(Q_{\bf o}D^{\T}E\right)
	\qquad\Rightarrow\qquad
	\frac {\partial \gamma(X)}{\partial X} = D Q_{\bf o}^{\T}.
\]
Hence,
\begin{align}
    \frac {\partial \cL(X,\Gamma)}{\partial X}
	%\partial \cL(X)
	&=
	H_{\phi}(X)X + \tr(M)\cdot H_{\psi}(X)X +\psi(X)\cdot DQ_{\bf o}^{\T}
		- X\Gamma \label{eq:partiall} \\
	& = G(X)X -   X\cdot \Lambda(X), \notag
\end{align}
where
$ \Lambda(X) =\psi(X) Q_{\bf o} D^{\T} X + \Gamma$, and we have used~\eqref{eq:der}
for the derivatives of $\phi$ and $\psi$,
and the definition of $G(X)$ in~\eqref{eq:g2x}.
Finally, the first-order optimality condition
$\partial \cL(X,\Gamma)/\partial X =0$ leads
to the aligned NEPv~\eqref{eq:alnepv}.

%The derivation above, although more simplified, relies on the \rankd~condition;
%otherwise, the Lagrangian function $\cL$ is not differentiable.
%The rationale behind the use of the \rankd~condition comes
%from some place else, not
%%This rank condition, however, does not follow directly from
%the new optimization~\eqref{eq:gopt2}, however.

%

\section{Proof of~\Cref{thm:qdef}}\label{sec:proofqdef}

We exploit the fact that the global maximizer $X_*$
of~\eqref{eq:gopt} is also a global maximizer of~\eqref{eq:gopt2}.
Let $\cL$ be the Lagrangian function for the constrained
optimization problem~\eqref{eq:gopt2}, as in~\eqref{eq:laggx}.
For $X,E\in\bbR^{n\times k}$, expand the Lagrangian function up to the second order to get
\[
\cL(X+t E,\Gamma)
= \cL(X,\Gamma) + t\D_1 \cL(X,\Gamma)[E] + t^2\frac{1}{2}  \D_1^2\cL(X,\Gamma)[E,E] +\mathcal O(t^3)
\]
for $t\in\bbR$ sufficiently small,
where $\D_1\cL(\cdot,\cdot)[\cdot]$ stands for partial differentiation with respect to the first matrix argument
of $\cL$, and $\D_1^2\cL(X,\Gamma)(\cdot,\cdot)$,
a bilinear form, is the partial Hessian operator of $\cL$
with respect to the first matrix argument.
By the standard second-order optimality
condition~\cite[Theorem~12.5]{Nocedal:2006},
any global maximizer $X_*$ of~\eqref{eq:gopt2} must satisfy
\begin{equation}\label{eq:hess1}
\D_1^2\cL(X_*,\Gamma_*)[E,E] \leq 0 \qquad \text{for all $E\in\bbR^{n\times k}$ with $E^{\T}X_*=0$},
\end{equation}
where $\Gamma_*$ contains the optimal multipliers associated with $X_*$.
On the other hand, we claim that
\begin{equation}\label{eq:hessid}
	 \D_1^2\cL(X_*,\Gamma_*)[E,E] = \tr(\,Z^{\T}\cQ(Z)\,)
	\qquad\text{for $E=X_{*\bot}Z$},
\end{equation}
where $\cQ(Z)$ is given as in \eqref{eq:qop}.
To justify~\eqref{eq:hessid},
we find, by the gradient of $\cL$ in~\eqref{eq:partiall}, that
\begin{align*}
	\D_1 \cL(X,\Gamma)[E] = \tr\Big(E^{\T} G(X)X -E^{\T}X\cdot \Lambda (X) \Big),
\end{align*}
where $ \Lambda(X) =\psi(X) Q_{\bf o} D^{\T} X + \Gamma$.
Differentiate it with respect to $X$ to obtain
\begin{align}
	\D_1^2 \cL(X,\Gamma)[E,E] &= \D_1 \Big(\D_1 \cL(X,\Gamma)[E]\Big)[E] \nonumber\\
				   &= \tr\Big(E^{\T} \cdot \D_1 G(X)[E]\cdot X
         + E^{\T}G(X)E -E^{\T}X \cdot \D \Lambda (X)[E] - E^{\T}E \cdot
	 \Lambda(X) \Big). \nonumber
		 %\label{eq:qdef:pf-1}
\end{align}
Noticing for $X=X_*$ and $E=X_{*\bot}Z$, we have
\[
	E^{\T}X = 0,\qquad E^{\T}E = Z^{\T}Z, \qquad E^{\T}G(X_*)E = Z^{\T}\Lambda_{*\bot}Z,
\]
where the last equation follows from~\eqref{eq:eigdgsx},
and by the first-order optimality condition $G(X_*)X_* = X_*\cdot\Lambda(X_*)$
we have $\Lambda(X_*) = X_*^{\T}G(X_*)X_* = \Lambda_*$.
Therefore, %for $X=X_*$ and $E=X_{*\bot}Z$,
\begin{align*}
	\D_1^2 \cL(X_*,\Gamma_*)[E,E] =
	\tr\Big(Z^{\T}X_{*\bot}^{\T} \cdot \D G(X_*)[X_{*\bot}Z]\cdot X_*+
	Z^{\T}\Lambda_{*\bot}Z - Z^{\T}Z \Lambda_* \Big)
	= \tr(Z^{\T}\cQ(Z)),
\end{align*}
which is~\eqref{eq:hessid}.
Finally,~\eqref{eq:qdef} is a consequence of~\eqref{eq:hess1}~and~\eqref{eq:hessid}.

%{\Blue
\section*{Acknowledgment}
The authors wish to thank the anonymous referees for their insightful suggestions and comments that improved the presentation of the paper.

\bibliographystyle{plain}
\bibliography{ref,refLi}

\end{document}